\documentclass{amsart}
\usepackage{amssymb}
\usepackage{graphicx,enumerate,tikz}
\usetikzlibrary{calc,arrows}
\begin{document}

\newtheorem{definition}{Definition}
\newtheorem{proposition}{Proposition}
\newtheorem{theorem}{Theorem}
\newtheorem{lemma}{Lemma}
\newtheorem{remark}{Remark}
\newtheorem{claim}{Claim}
\newtheorem{fact}{Fact}
\def\A{\mathcal{A}}

\def\arccot{\mathrm{arccot}}
\def\coord#1#2{\langle#2,\,#1\rangle}
\def\coordi#1#2#3{\langle#2,\,#1\rangle}
\def\degenerate{\mathrm{dgn}}
\def\fra#1#2{{#1}/{#2}}
\def\Nset{\mathbf{N}}
\def\P{P}
\def\PQ{Q}
\def\Q{P'}
\def\Refl#1{R({#1})}
\def\Rset{\mathbf{R}}
\def\SO{\mathit{SO}}
\def\S{\mathcal{S}}
\def\Sset{\mathbf{S}}
\def\TRPZ{PDW}
\def\T{\mathcal{T}}

\title[Spherical isohedral tilings over pseudo-double wheels]{Classification of spherical tilings by congruent quadrangles over
  pseudo-double wheels ({II})\\
--- the  isohedral case}
\date{\today}            
\thanks{The  author was supported by JSPS KAKENHI Grant Number 16K05247.}


\author[Y. Akama]{Yohji Akama}
\address{Mathematical Institute\\
  Graduate School of Science \\
  Tohoku University\\
  Sendai 980-0845 JAPAN} 
\email{yoji.akama.e8@tohoku.ac.jp}
\urladdr{http://www.math.tohoku.ac.jp/akama/stcq/}

  \begin{abstract}
We classify all edge-to-edge spherical isohedral 4-gonal tilings such that the skeletons are pseudo-double wheels. 
 For this, we characterize these spherical tilings  by a quadratic equation for the cosine of an  edge-length. 
 By the classification, we see:
 there are indeed two non-congruent, edge-to-edge spherical isohedral 4-gonal tilings 
 such that the skeletons are the same pseudo-double wheel and the cyclic list of the four inner angles of the tiles are the same. 
 This contrasts with that every edge-to-edge spherical tiling by congruent 3-gons is determined by the skeleton and the inner angles of the skeleton. 
We show that for a particular spherical isohedral tiling over the pseudo-double wheel of twelve faces, 
the quadratic equation has a double solution and 
 the copies of the tile also organize a spherical non-isohedral tiling
   over the same skeleton. 
  \end{abstract}
\maketitle

\section{Introduction}
\label{sec:introduction} Throughout this paper, we are concerned with
edge-to-edge tilings.  A tiling $\T$ is called \emph{isohedral}~(or
\emph{tile-transitive}), if for any pair of tiles of $\T$, there is a
symmetry operation of $\T$ that transforms one tile to the other.  In characterizing the
\emph{skeletons} of spherical (isohedral) tilings, an important graph is a
\emph{pseudo-double wheel}~(the dual graph of the skeleton of an
antiprism~\cite[p.~19]{Deza}. See Figure~\ref{summary}~(above)). It satisfies the following:
\begin{itemize}
\item The
skeletons of spherical tilings by spherical 4-gons are generated from
pseudo-double wheels by means of applications of two local
expansions~\cite{MR2186681}.  
\item  The skeletons of spherical isohedral
tilings consists of pseudo-double wheels, an infinite series of graphs,
and eighteen sporadic graphs~\cite{MR661770}. 
\end{itemize}
In Section~\ref{sec:char}, we prove: for every spherical
tiling $\T$ by congruent spherical 4-gons with the skeleton being a pseudo-double wheel $G$, $\T$ is isohedral if and only if every graph automorphism~\cite[Sect.~1.1]{Deza} of $G$ respects the  edge-length and inner angles of $\T$.

\def\wwd{4}
  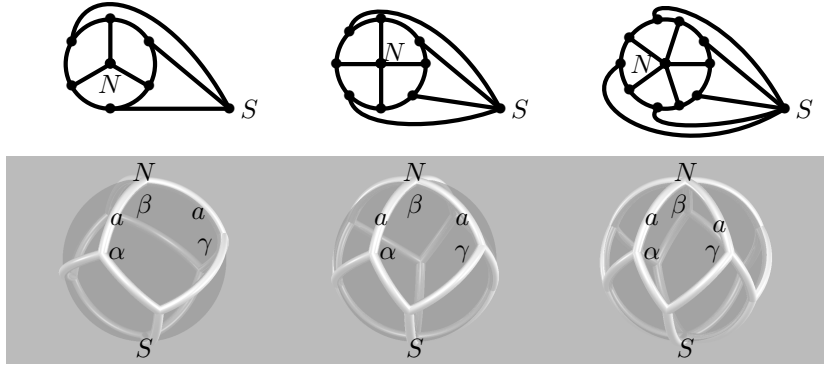
\begin{figure}[ht]
		\begin{tikzpicture}[scale=0.9]
     \node at (-\wwd,0)
		 {\includegraphics[scale=0.13]{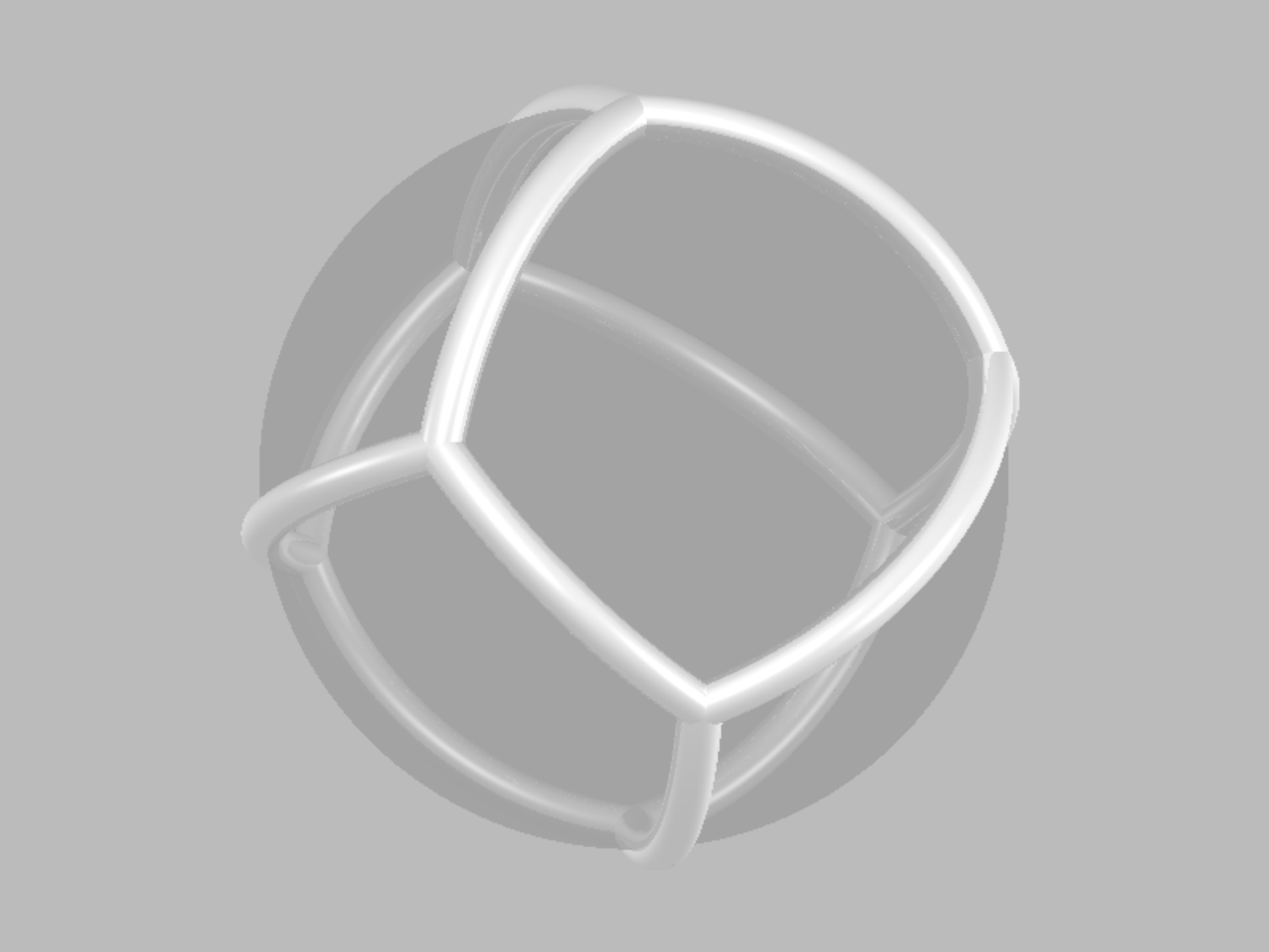}};
     \node at (-\wwd,1.3) {$N$};
     \node at (-\wwd,   .8) {$\beta$};
     \node at (-\wwd-.4,.6) {$a$};
     \node at (-\wwd+.8, .7) {$a$};
     \node at (-\wwd-.4, .1)    {$\alpha$};
		 \node at (-\wwd+.9, .2)   {$\gamma$};
		 \node at (-\wwd,-1.3) {$S$};
     \node at (0,0)
		 {\includegraphics[scale=0.13]{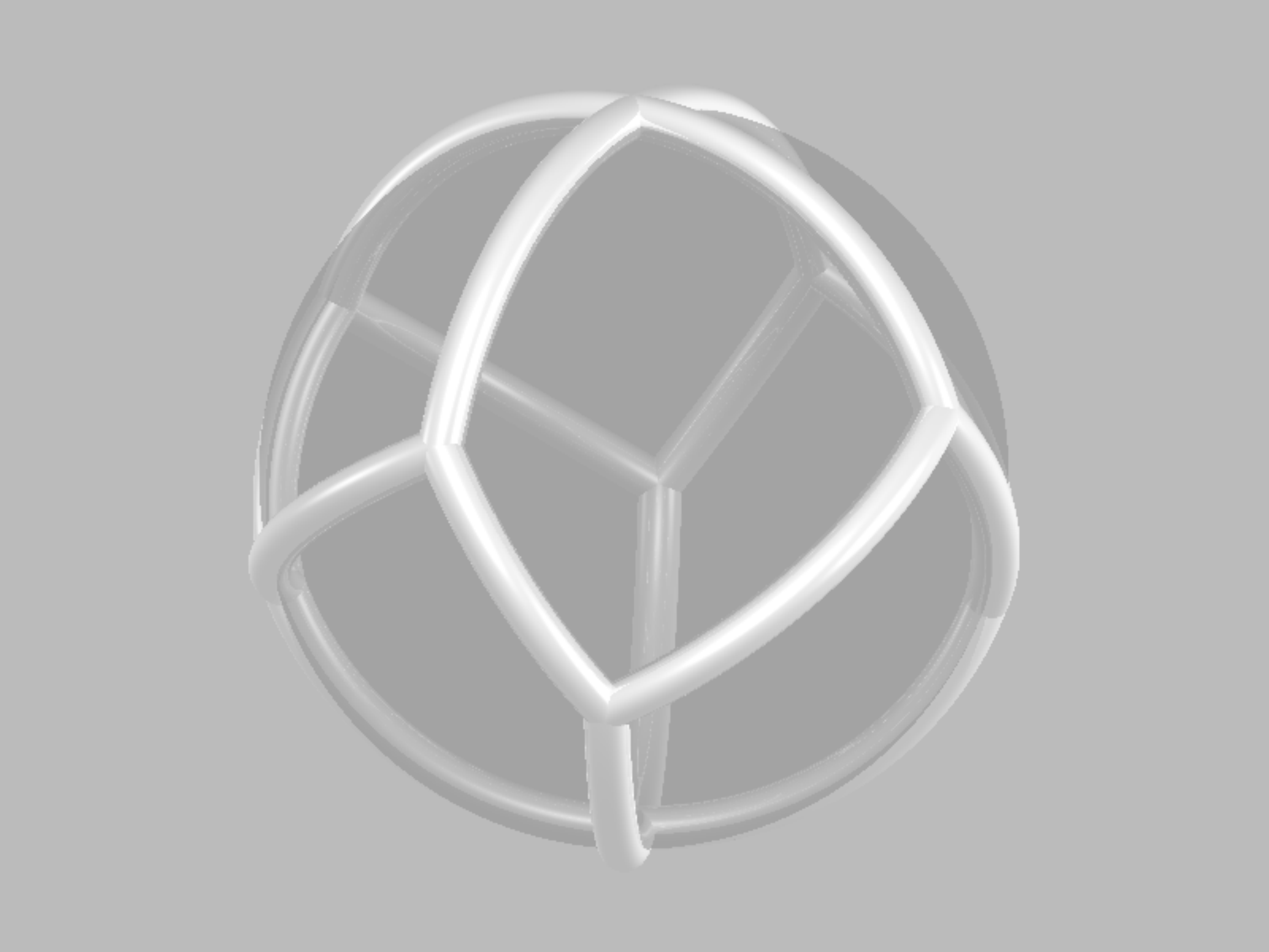}};
     \node at (0,1.3) {$N$};
     \node at (0,.8) {$\beta$};
     \node at (-.5, .6) {$a$};
     \node at (.7,  .6) {$a$};
     \node at (-.4, .1)    {$\alpha$};
		 \node at (.7, .1)   {$\gamma$};
		 \node at (0,-1.3) {$S$};
     \node at (4,0)
		 {\includegraphics[scale=0.13]{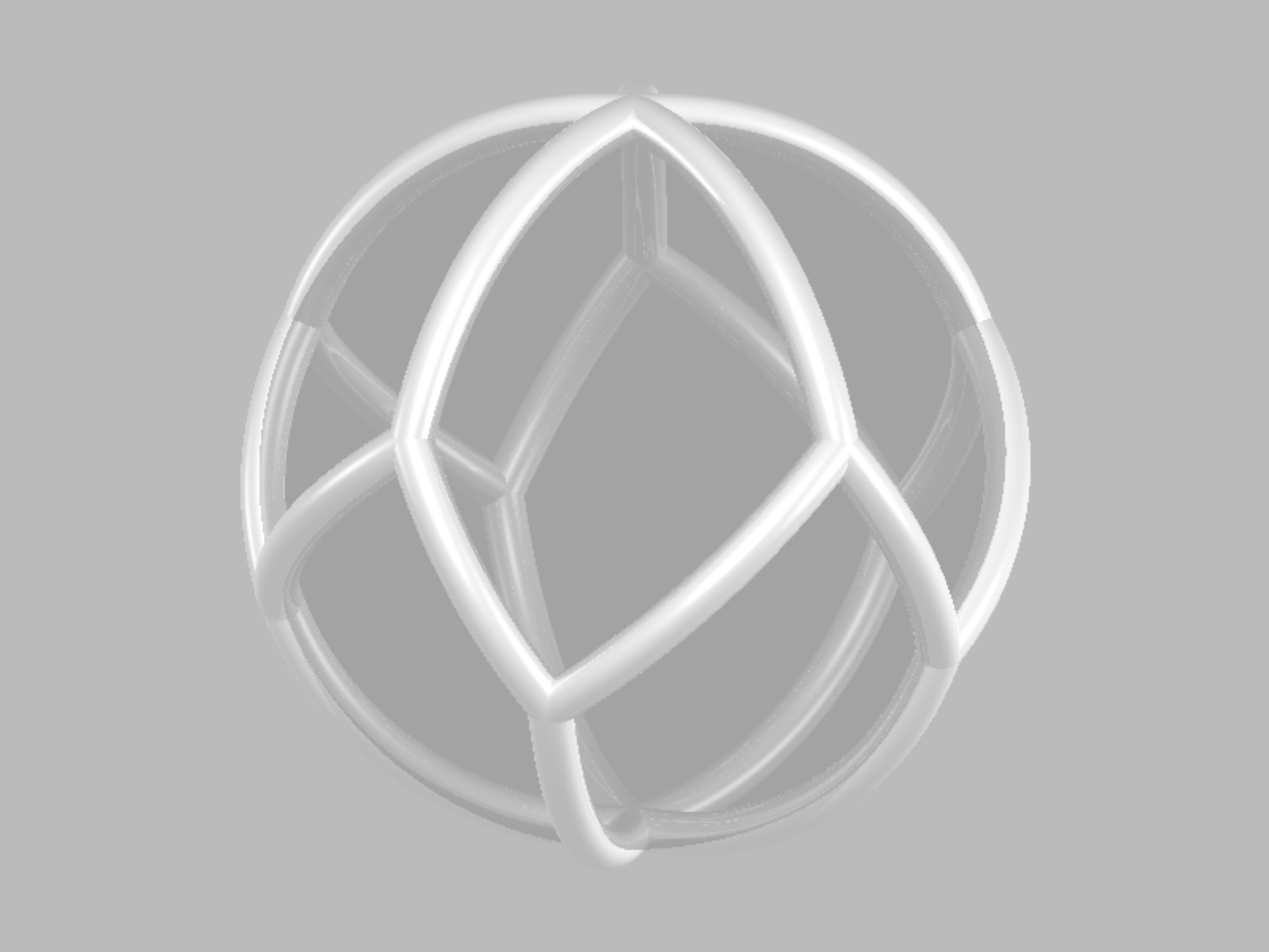}};
		 \node at (\wwd,1.3) {$N$};
     \node at (\wwd-.1,.8) {$\beta$};
     \node at (-.5+\wwd, .6) {$a$};
     \node at (.5+\wwd,.5) {$a$};
     \node at (-.5+\wwd, .1)    {$\alpha$};
		 \node at (.4+\wwd, .1)   {$\gamma$};
		 \node at (\wwd,-1.3) {$S$};
	       \def\rds{0.6}
	       \def\hgt{2.9}
	       \begin{scope}[shift={(-4.5,\hgt)},scale=1.1]

		\draw[ultra thick] (30:\rds) node {$\bullet$} -- (0.4*\wwd,-\rds) node[right] {$S$};
		\draw[ultra thick] (150:\rds) node {$\bullet$} .. controls (-.1*\wwd,2*\rds) and
		(.2*\wwd,1.5*\rds) .. (0.4*\wwd,-\rds)node {$\bullet$};
		\draw[ultra thick] (270:\rds) node {$\bullet$}-- (0.4*\wwd,-\rds);
		\draw[ultra thick] (0,0) node[below] {$N$} circle (\rds);
		\draw[ultra thick] (0,0) node {$\bullet$}-- (90:\rds)node {$\bullet$};
		\draw[ultra thick] (0,0) node {$\bullet$}-- (210:\rds)node {$\bullet$};
		\draw[ultra thick] (0,0) node {$\bullet$}-- (-30:\rds)node {$\bullet$};
	       \end{scope}
	       \begin{scope}[shift={(-.5,\hgt)},scale=1.1]	       \def\rds{0.6}\def\wwd{4}
                 \node at(45:\rds*.4) {$N$};
		\draw[ultra thick] (30:\rds) node {$\bullet$} -- (0.4*\wwd,-\rds) node[right] {$S$};
		\draw[ultra thick] (135:\rds) node {$\bullet$} .. controls (-.1*\wwd,1.5*\rds) and
		(.2*\wwd,1.5*\rds) .. (0.4*\wwd,-\rds)node {$\bullet$};
		\draw[ultra thick] (-135:\rds) node {$\bullet$} .. controls (-.1*\wwd,-1.5*\rds) and
		(.2*\wwd,-1.5*\rds) .. (0.4*\wwd,-\rds);
		\draw[ultra thick] (-45:\rds) node {$\bullet$}-- (0.4*\wwd,-\rds);
		\draw[ultra thick] (0,0) circle (\rds);
		\draw[ultra thick] (0,0) node {$\bullet$}-- (0:\rds)node {$\bullet$};
		\draw[ultra thick] (0,0) node {$\bullet$}-- (90:\rds)node {$\bullet$};
		\draw[ultra thick] (0,0) node {$\bullet$}-- (180:\rds)node {$\bullet$};
		\draw[ultra thick] (0,0) node {$\bullet$}-- (270:\rds)node {$\bullet$};
	       \end{scope}
	       \begin{scope}[shift={(3.7,\hgt)},scale=1.1]	       \def\rds{0.6}\def\wwd{4}

		\draw[ultra thick] (30:\rds) node {$\bullet$} -- (0.4*\wwd,-\rds) node[right] {$S$};
		\draw[ultra thick] (100:\rds)node {$\bullet$} .. controls (-.1*\wwd,1.5*\rds) and
		(.2*\wwd,1.5*\rds) .. (0.4*\wwd,-\rds)node {$\bullet$};
		\draw[ultra thick] (180:\rds)node {$\bullet$} .. controls (200:2.5*\rds) and
		(275:2.5*\rds) .. (0.4*\wwd,-\rds);
		\draw[ultra thick] (-100:\rds)node {$\bullet$} .. controls (-.1*\wwd,-1.5*\rds) and
		(.2*\wwd,-1.5*\rds) .. (0.4*\wwd,-\rds);
		\draw[ultra thick] (-45:\rds) node {$\bullet$}-- (0.4*\wwd,-\rds);
		\draw[ultra thick] (0,0) circle (\rds);

		\draw[ultra thick] (0,0) node[left] {$N$} -- (0:\rds)node {$\bullet$};
		\draw[ultra thick] (0,0) node {$\bullet$}-- (72:\rds)node {$\bullet$};
		\draw[ultra thick] (0,0) node {$\bullet$}-- (144:\rds)node {$\bullet$};
		\draw[ultra thick] (0,0) node {$\bullet$}-- (216:\rds)node {$\bullet$};
		\draw[ultra thick] (0,0) node {$\bullet$}-- (288:\rds)node {$\bullet$};
	       \end{scope}

		\end{tikzpicture}
   \caption{The above are pseudo-double wheels of $2n$ faces~($n=3,4,5$). The below are
   spherical isohedral tilings by $2n$ congruent quadrangles such that the skeletons are
 pseudo-double wheels~($n=3,4,5$).  It holds that $(\cos a)^2-\cot(\pi/n)(\cot \alpha +
  \cot\gamma)\cos a - \cot\alpha\cot\gamma=0$.
 \label{summary}}
  \end{figure}
 
In any spherical tiling
by congruent quadrangles, the tile has a pair of adjacent, equilateral
edges~\cite{agaoka:quad}. For spherical tilings by congruent quadrangles
$T$ with the skeleton being pseudo-double wheels, fix the notation for
the angles and the edges of the quadrangular tile $T$ as
Figure~\ref{fig:4gon}.  
\begin{figure}[ht]
   \begin{tikzpicture}[scale=.8]
    \node at (0.3,0.2) {\Large $\beta$};
    \draw[->] (0.8,0) arc [start angle=0, end angle=64, radius=.7];

    \node at (0.4,1.2) {\Large $a$};
   \draw (0,0)--(1,2);

    \node at (1.1,1.7) {\Large $\alpha$};
    \draw[->] (.8,1.5) arc [start angle=260, end angle=310, radius=.9];

    \node at (1.7,2.1) {\Large $b$};

   \draw[line width=1mm] (1,2)--(2.3,1.55);
    \node at (2.1,1.38) {\Large $\delta$};
   \draw[->] (1.8,1.6) arc [start angle=170, end angle=260,radius=.5];
    
   \draw[dotted, line width=1mm] (2.3,1.55) -- (2.3,0);
    \node at (2.1,0.2) {\Large $\gamma$};
    \draw[<-] (1.7,0.07) arc [start angle=179, end angle=98, radius=.6];

    \node at (2.5,0.9) {\Large $c$};
    \draw (2.23,0) -- (0,0);
    \node at (1.1,-0.2) {\Large $a$};
   \end{tikzpicture}
 \caption{The notation for angles and edges of the quadrangular tile. 
 Some among $\alpha,\beta,\gamma,\delta$ are equal, and
 some among $a,b,c$ are equal. 
 \label{fig:4gon}}
\end{figure}
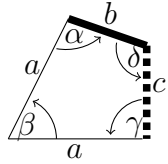
For the spherical \emph{isohedral} tilings such that the skeletons are pseudo-double
wheels, we bijectively parameterize the tiles with the pair of the
edge-length $a$ of the tile and a typical angle, in
Section~\ref{sec:tiles}. Then
we characterize the tiles as
follows~(Section~\ref{sec:q}). 
Given a spherical 4-gon $T$ such that $\alpha,\beta,\gamma$ are adjacent
inner angles and $\beta$ is an inner angle between two edges of length
$a$. $T$ is a tile of some spherical \emph{isohedral} tiling $\T$ by $2n$ congruent spherical 4-gons
with the skeleton of $\T$ being a pseudo-double wheel, if and only if
\begin{align*}
(\cos a)^2-\cot\frac{\pi}{n}(\cot\alpha+\cot\gamma)\cos a -
\cot\alpha\cot\gamma=0.\end{align*}
For notations, see Figure~\ref{summary}~(below).
In Section~\ref{sec:main}, by solving this equation, we classify all the tiles of spherical
isohedral tilings such that the skeleton of the tilings are
pseudo-double wheels. 
By the classification, we see:
 there are indeed two non-congruent, edge-to-edge spherical isohedral 4-gonal tilings 
 such that the skeletons are the same pseudo-double wheel and the cyclic list of the four inner angles of the tiles are the same. 
 This contrasts with that every edge-to-edge spherical tiling by congruent 3-gons is determined by the skeleton and the inner angles of the skeleton~\cite{MR1954054}. 
 In Section~\ref{sec:noniso}, we show that for a particular spherical isohedral tiling over the pseudo-double wheel of twelve faces, 
the quadratic equation has a double solution. Moreover, 
 the copies of the tile also organize a less symmetric, spherical non-isohedral tiling $\T$ over the same skeleton.  
 Based on this tiling $\T$ and Gr\"unbaum-Shephard's characterization theorem~\cite{MR661770} of the skeletons of spherical isohedral tilings, we briefly discuss our classification of spherical isohedral tilings over pseudo-double wheels.

\section{Basic definitions\label{sec:basic definitions}}

By a \emph{spherical $4$-gon}, we mean a topological disk $T$
on the two-dimensional unit sphere $\Sset^2$ such that $T$ is
circumscribed by four straight edges, (1) any inner angle between
adjacent edges of $T$ is strictly between 0 and $2\pi$ but not $\pi$, and
(2) $T$ is contained in the interior of a hemisphere.  By  ``quadrangle,'' we mean a
``spherical 4-gon.''  The congruence on the sphere is just the orthogonal
transformation, and ``sphere'' and ``spherical'' means the
two-dimensional unit sphere $\Sset^2$.  We identify spherical tilings modulo a special
orthogonal group $SO(3)$.

\begin{definition}[pseudo-double wheel~\protect{\cite{MR2186681}}]
For an even number $F\ge6$, a \emph{pseudo-double
wheel\/} of $F$ faces is a map such that

\begin{itemize}
\item the graph is obtained from a cycle
$(v_0, v_1,
v_2, \ldots, v_{F-1})$, by adjoining a new vertex $N$ to each $v_{2i}$
 $(0\le i<F/2)$ and then by adjoining a new vertex $S$ to each $v_{2i+1}$
 $(0\le i<F/2)$. We identify the suffix $i$ of the vertex $v_i$ modulo
 $F$. 
\item
The cyclic order 
at the vertex $N$ is defined as follows: the edge $N v_{2i+2}$ is next
      to the edge $N v_{2i}$. 
The cyclic order at the vertex $v_{2i}$ ($0\le i\le F/2$) is: the edge
 $v_{2i} N$ is next to the edge $v_{2i} v_{2i+1}$, and $v_{2i} v_{2i+1}$ is next to the edge $v_{2i}
v_{2i-1}$.
The cyclic order 
at the vertex $S$ is: the edge $S v_{2i-1}$ is next to the edge $S v_{2i+1}$. 
The cyclic order at the vertex $v_{2i+1}$ ($0\le i< F/2$) is:
the edge $v_{2i+1} S$ is next to the edge $v_{2i+1} v_{2i}$, and $v_{2i+1} v_{2i}$ is
      next to the edge $v_{2i+1}v_{2i+2}$.
\end{itemize}
\end{definition}
The skeleton of the cube is the pseudo-double wheel of six faces.
\def\wwd{4}
 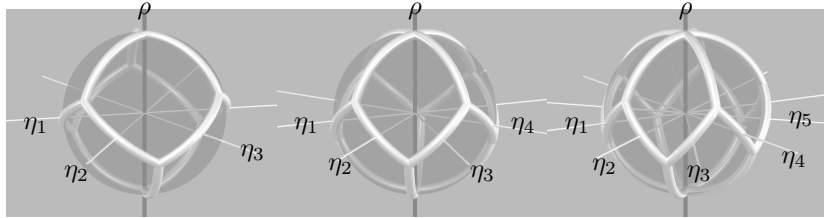
\begin{figure}[ht]
    \begin{tikzpicture}[scale=.9]
    \node at (-\wwd,0) {\includegraphics[scale=0.13]{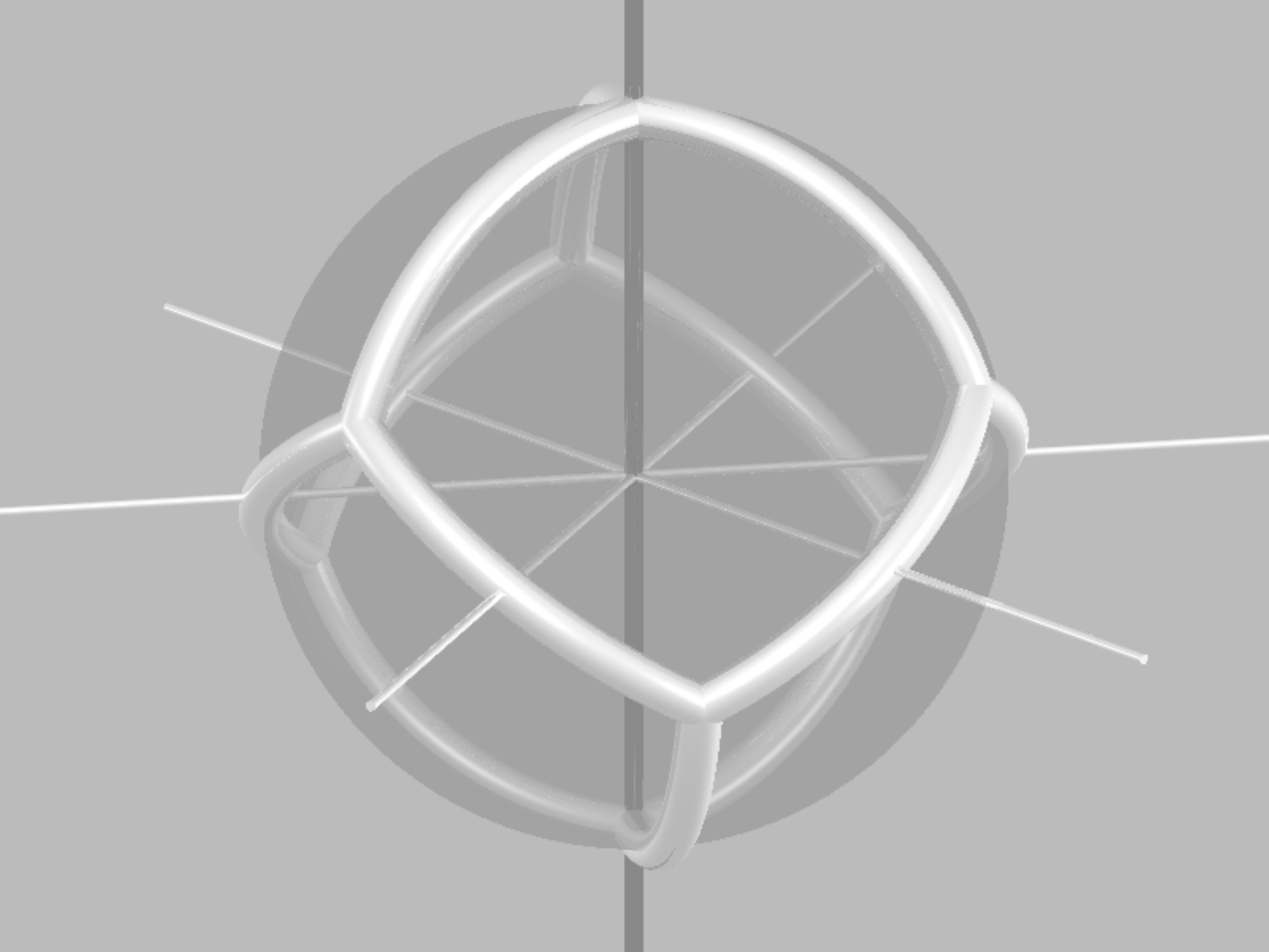}};
     \node at (-\wwd,1.5) {$\rho$};
     \node at (-5.6,-.2) {$\eta_1$};
     \node at (-5,-.9) {$\eta_2$};
     \node at (-2.4,-.6) {$\eta_3$};
     \node at (0,0) {\includegraphics[scale=0.13]{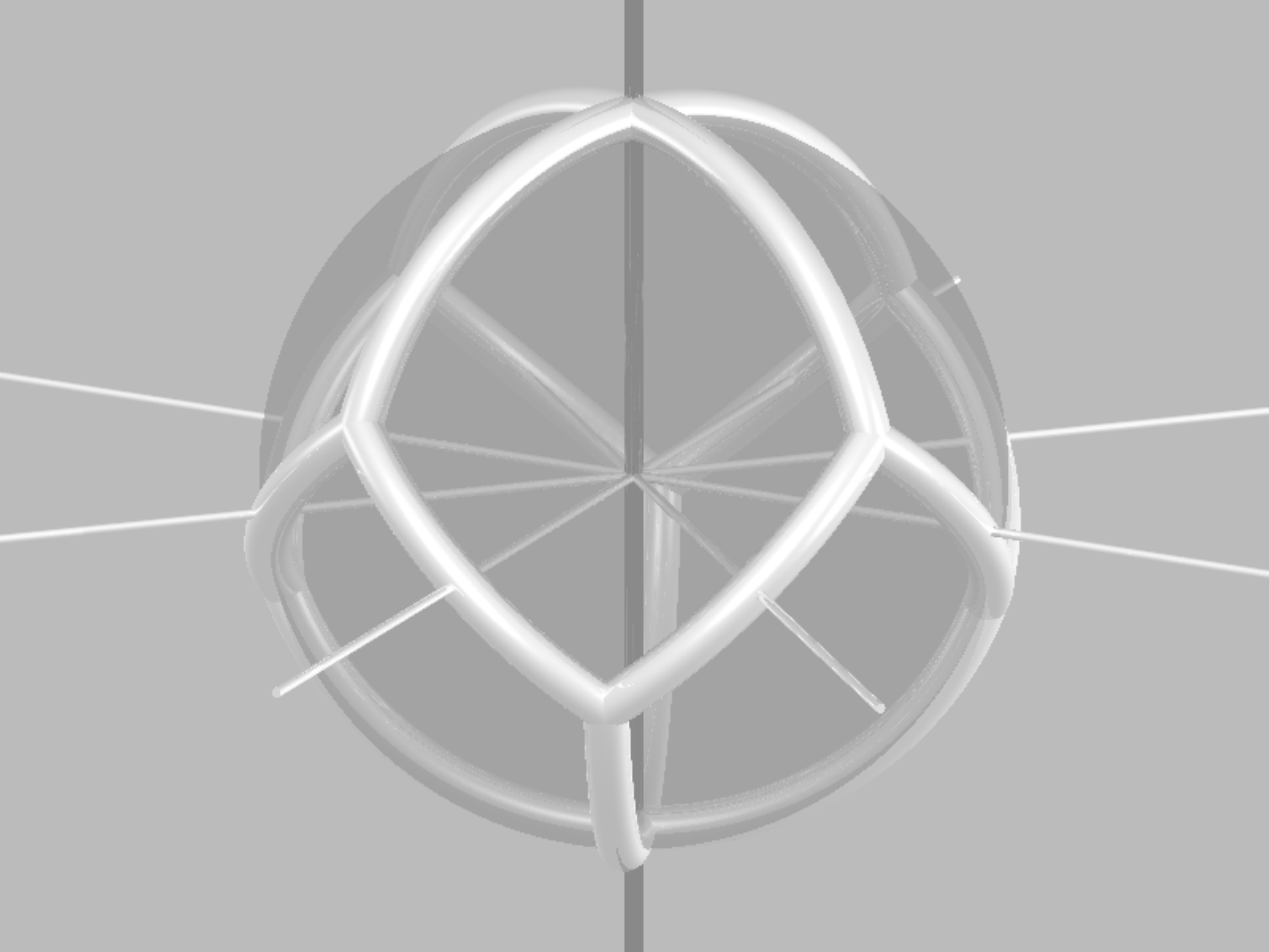}};
     \node at (0,1.5) {$\rho$};
     \node at (-5.6+\wwd,-.2) {$\eta_1$};
     \node at (-5.1+\wwd,-.8) {$\eta_2$};
     \node at (-3+\wwd,-.9) {$\eta_3$};
     \node at (-2.4+\wwd,-.2) {$\eta_4$};
     \node at (4,0) {\includegraphics[scale=0.13]{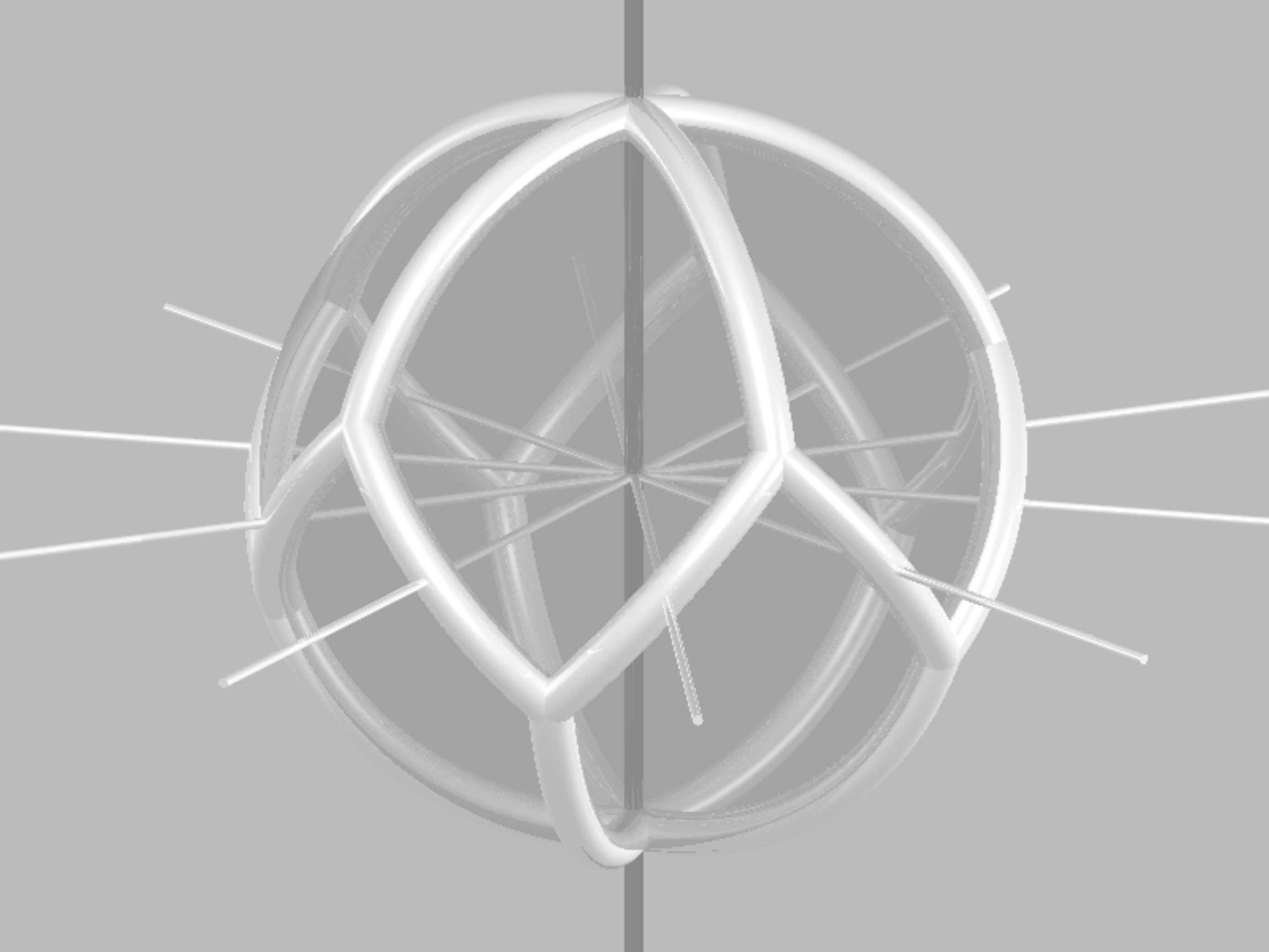}};
     \node at (\wwd,1.5) {$\rho$};
     \node at (-5.6+2*\wwd, -.2) {$\eta_1$};
     \node at (-5.2+2*\wwd,-.8) {$\eta_2$};
     \node at (-3.8+2*\wwd,-.9) {$\eta_3$};
     \node at (-2.4+2*\wwd,-.7) {$\eta_4$};
     \node at (-2.3+2*\wwd,-.1) {$\eta_5$};
    \end{tikzpicture}
\caption{Spherical tilings by $2n$ congruent quadrangles such that the skeletons are
 pseudo-double wheels~($n=3,4,5$). Each is  isohedral, as any tile is transformed to any tile with the
 vertical $n$-fold axis $\rho$ and $n$ horizontal $2$-fold axes $\eta_1,\ldots,\eta_n$. 
 \label{p_expansion}}
 \end{figure}

\emph{In the rest of this paper, we fix the
orientation of the sphere. By $\angle P Q R$, we mean the angle from
$PQ$ to $RQ$ in the orientation of the sphere, and assume that (1)
$\alpha,\beta,\gamma,\delta\in(0,\,\pi)\cup(\pi,2\pi)$, and
$a,b,c\in(0,\,\pi)$, and (2) for tiles, edges represented by
solid~$($thick, dotted, resp.$)$ lines have length $a$~$(b,c$,
resp.\/$)$.}  We say a quadrangle is \emph{concave}, if it has an inner
angle greater than $\pi$. We are concerned with all spherical isohedral
tilings by congruent possibly concave quadrangles such that the skeletons are pseudo-double wheels.

 \begin{proposition}\label{prop:vinberg}
  \begin{enumerate}
   \item \label{assert:vinberg} $($\cite[p.~62]{MR1254932}$)$
If 
$0<A,B,C<\pi,\ 
 A + B + C >\pi,\ \ -A  +  B  +  C  <\pi,\ \ A  -  B 
 +  C  <\pi$ and $A  +  B  -  C  <\pi$, 
then there exists \emph{uniquely} up to congruence a spherical
3-gon on the two-dimensional unit sphere $\Sset^2$ such that
the inner angles are $A,  B$
 and $C$. The converse is also true. 

\item Let $ABC$ be a spherical 3-gon, and let $a,b,c$ be the sides
 opposite to the inner angles $A,B,C$, respectively.  Then
 \begin{enumerate}
   \item $($Dual cosine law for the sphere~$($Spherical cosine theorem for angles$)$~\cite[p.~65]{MR1254932}$)$\label{assert:scla}
$\cos  A =- \cos  B \cos C +\sin B \sin C \cos a$.

   \item $($Cosine law for the sphere~$($Spherical cosine theorem$)$~\cite[p.~65]{MR1254932}$)$\label{assert:scl}
$\cos  a = \cos b \cos c +\sin b \sin c \cos A$.
\end{enumerate}
\end{enumerate}

 \end{proposition}
 \emph{Spherical cosine law} is obtained from
the spherical cosine theorem for angles, by exchanging the angles
$A,B,C$ and the sides $a,b,c$ with $A\leftrightarrow \pi -a,
B\leftrightarrow\pi -b, C\leftrightarrow\pi -c$. By this exchange, the
last three inequalities of
Proposition~\ref{prop:vinberg}~\eqref{assert:vinberg} become the distance inequalities for spherical 3-gons.
For every nonzero real number $x$, $\arccot\ x $ is the angle $\theta$
 such that $0<|\theta|<\pi/2$ and $\cot \theta = x$. Let $\csc x$ be
 $1/\sin x$. We say a spherical $4$-gon $Q$ is a \emph{copy} of a
 spherical $4$-gon $Q'$, if $Q$ is an orthogonal transformation of $Q'$.

\section{Combinatorial conditions for spherical monohedral
 quadrangular tilings to be isohedral tilings over pseudo-double wheels}\label{sec:char} 

\begin{definition}\label{def:pdw}
Let $\TRPZ_n$~($n\ge3$) be the set of spherical
tilings by $2n$ congruent, possibly concave quadrangles such
 that (1) the skeleton is a pseudo-double wheel, and (2) the distribution
 of inner angles and that of the edge-length on the skeleton are as in
 Figure~\ref{chart of TRPZ}. 
\end{definition}

Note that we do not assume the isohedrality in Definition~\ref{def:pdw}.
\begin{figure}[ht]
 \centering
\includegraphics[width=8cm]{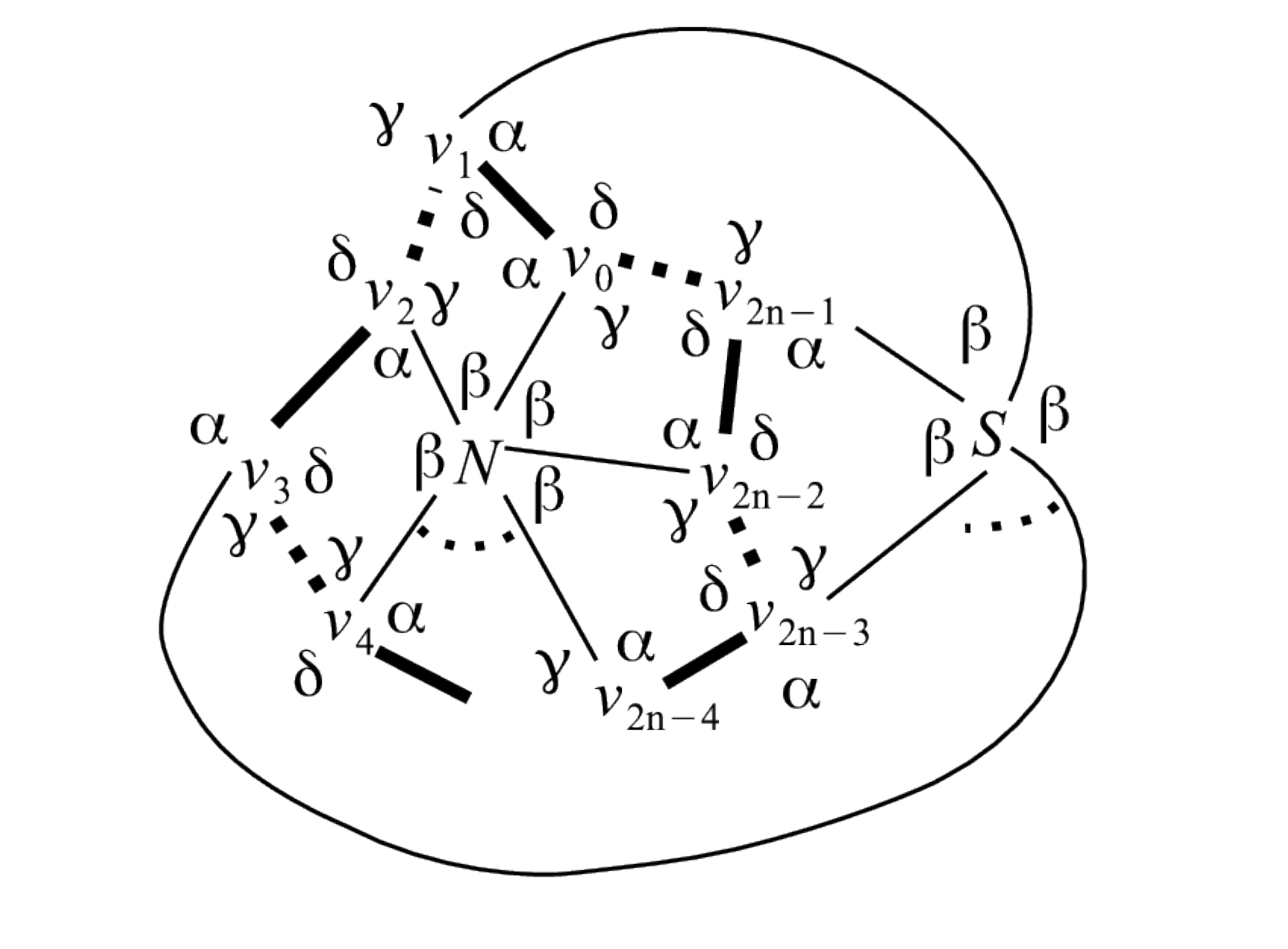}
\caption{
The solid, the thick, and the
 dotted thick edges have length $a, b$ and $c$. 
 The vertices $N$ and $S$ are $n$-valent. Some among $\alpha,\beta,\gamma,\delta$ are equal, and some among
 $a,b,c$ are equal.
 \label{chart of TRPZ}}
\end{figure}

For example, all images of Figure~\ref{summary} and Figure~\ref{p_expansion} are members of $\TRPZ_n$ $(n=3,4,5)$. The leftmost
of Figure~\ref{p_expansion} is so-called \emph{the central projection}
of the cube. They have the vertical $n$-fold axis $\rho$ of rotation, and $n$ horizontal
two-fold axes $\eta_i$~$(i=0,1,\ldots,n-1)$ of rotation such that for
each $i$,
$\eta_i$ is through the midpoint of an edge and $\eta_i\perp\rho$. By
these symmetry operations, any tile is transformed to any tile in each
tiling. So, they are isohedral.

The two vertices $N$ and $S$ of any tiling $\T$ presented in
Figure~\ref{chart of TRPZ} can be identified with the north pole and the
south pole of the unit sphere $\Sset^2$ respectively, as there are
two congruent paths from $N$ to $S$.  For each point $V~(\ne N,S)$ on
$\Sset^2$, the \emph{longitude} of $V$ is the angle $\psi\in [-\pi, \
\pi)$ from the edge $N v_0$ to a geodesic segment $N V$, measured in the
direction indicated in Figure~\ref{fig:cs}.

\begin{proposition}[\protect{\cite[Lemma~5]{akama12:_class_of_spher_tilin_by_i}}]\label{prop:aux}
Given a spherical tiling by congruent quadrangles such that the
quadrangles are as in Figure~\ref{fig:4gon} with the edge-length $c$
being the edge-length $a$.  Suppose that
$(1)$ there is a vertex incident to only three edges of length $a$, and
$(2)$ there is a $3$-valent vertex incident to two edges
      of length $a$ and to one edge of length $b$.
Then for the inner angles of the tile, we have  $\alpha\ne\delta$ and $\beta\ne\gamma$. 
\end{proposition}

\begin{theorem}\label{thm:cube}
  Let $\T$ be a spherical tiling by six congruent quadrangles.
 \begin{enumerate}
  \item \label{assert3:Pn}
	$\T\in\TRPZ_3$.
	
  \item \label{assert3:dn} $\T$ has a three-fold
	axis $\rho$ of rotation and three two-fold axes
	$\eta_1,\eta_2,\eta_3$ of rotation
	perpendicular to $\rho$.

  \item \label{assert3:isohedral} $\T$ is
	isohedral.
 \end{enumerate}
\end{theorem}
 \begin{proof} \eqref{assert3:Pn}. Sakano proved this assertion by case analysis
  \cite{sakano10:_towar_class_of_spher_tilin}. We improve the
 presentation of his proof, by using Proposition~\ref{prop:aux},
 \cite{MR2186681}, and \cite[Theorem~8]{behmaram13:_upper_pfaff}.  By
 Euler's theorem, every spherical tiling by 4-gons has a 3-valent
 vertex~\cite{agaoka:quad}. From this, we can prove that the skeleton of
 any spherical tiling by six 4-gons is the skeleton of the
 cube~(By the enumeration of spherical
 quadrangulations~\cite{MR2186681}, there is only one spherical
 quadrangulation of eight vertices).  Moreover, the cyclic list of
 edge-lengths of the tile of a spherical tiling by congruent $4$-gons is
 either $aaaa$, $aabb$, $aaab$, or $aabc$ where $a,b,c$ are mutually
 distinct~(\cite{agaoka:quad}). When the cyclic list of edge-lengths of
 the tile of $\T$ is $a a a a$ or $a a b b$, then $\T\in\TRPZ_3$, by
 \cite{sakano15:anisohedral}.

 Let the edge-lengths of the tile be $aaab$~$(a\ne b)$. Then the
 spherical tiling by six congruent $4$-gons induces a perfect
 face-matching consisting of three edges of length $b$. By a \emph{perfect
 face-matching} of a graph, we mean a perfect matching~\cite[p.~2]{Deza} of the dual
 graph.  By
 \cite[Theorem~8]{behmaram13:_upper_pfaff}, the skeleton of the cube has
 eight perfect face-matchings.
 
 \def\pfm#1#2#3#4#5#6{ \begin{scope}[scale=#1,shift={(#2,#6)}]
 \coordinate (zero) at (135:1); \coordinate (one) at (135:2);
 \coordinate (two) at (45:2); \coordinate (three) at (45:1); \coordinate
 (four) at (-45:1); \coordinate (five) at (225:1); \coordinate (six) at
 (225:2); \coordinate (seven) at (-45:2);

     \draw (zero) -- (one) -- (two) -- (three) -- cycle;
     \draw (zero) -- (five) -- (six) -- (one) -- cycle;
     \draw (zero) -- (three) -- (four) -- (five) -- cycle;
     \draw (seven) -- (four) -- (three) -- (two) -- cycle;
     \draw (seven) -- (six) -- (five) -- (four) -- cycle;
     \draw (seven) -- (six) -- (one) -- (two) -- cycle;

     \draw[ultra thick] #3 ;
     \draw[ultra thick] #4 ;
     \draw[ultra thick] #5 ;
		     \end{scope}}
 \def\vnumbering{
     \node at (135:.7) {\tiny $0$};
     \node at (135:2.3) {\tiny $1$};
     \node at (45:2.3) {\tiny $2$};
     \node at (45:.7) {\tiny $3$};
     \node at (-45:.7) {\tiny $4$};
     \node at (225:.7) {\tiny $5$};
     \node at (225:2.3) {\tiny $6$};
     \node at (-45:2.3) {\tiny $7$};
}
 
 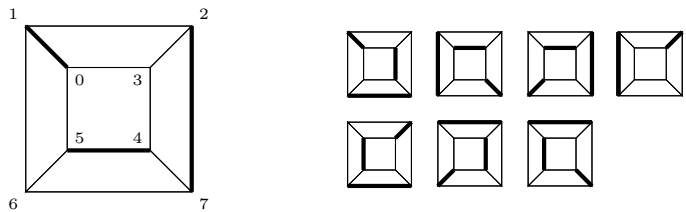
\begin{figure}\centering
   \begin{tikzpicture}[scale=.6]
   \pfm{1.3}{0cm}{(zero)--(one)}{(two)--(seven)}{(four)--(five); \vnumbering}{0cm}
   \pfm{0.5}{12cm}{(zero)--(one)}{(three)--(four)}{(six)--(seven)}{2cm}
   \pfm{0.5}{16cm}{(zero)--(three)}{(one)--(six)}{(four)--(seven)}{2cm}
    \pfm{0.5}{20cm}{(zero)--(three)}{(five)--(six)}{(two)--(seven)}{2cm}
    \pfm{0.5}{24cm}{(one)--(six)}{(two)--(three)}{(one)--(six)}{2cm}
   \pfm{0.5}{12cm}{(zero)--(five)}{(two)--(three)}{(six)--(seven)}{-2cm}
   \pfm{0.5}{16cm}{(one)--(two)}{(three)--(four)}{(five)--(six)}{-2cm}
   \pfm{0.5}{20cm}{(zero)--(five)}{(one)--(two)}{(four)--(seven)}{-2cm}
   \end{tikzpicture}
  \caption{The eight perfect face-matchings of the skeleton of the cube.\label{fig:pfm}}
 \end{figure}
       In Figure~\ref{fig:pfm}, the first perfect face-matching is
 transformed to the other seven perfect face-matchings by seven
 automorphisms of the skeleton of the cube. Enumerate the vertices
 $v_i$~($0\le i\le 7$) of the cube, as in the figure of the first
 perfect face-matching.  The seven automorphisms are represented as
 seven permutations $(2\,6)(3\,5)$, $(0\,4)(1\,7)(2\,6)(3\,5)$,
 $(0\,5)(1\,6)(2\,7)(3\,4)$, $(0\,3)(1\,2)(4\,5)(6\,7)$,
 $(0\,3\,4\,5)(1\,2\,7\,6)$, $(0\,5\,4\,3)(1\,6\,7\,2)$, and $(0\,
 4)(1\,7)$.  So we have only to consider the first perfect
 face-matching. Every inner angle around the vertex $v_3$ or the vertex
 $v_6$ is $\beta$ or $\gamma$, by Figure~\ref{fig:4gon}, because $v_3$ and $v_6$ are incident to
 only edges of length $a$. The number  of inner angles $\beta$ around
 $v_3$, say $k$, is the number  of inner angles $\beta$ around $v_6$. Otherwise, $\beta=\gamma$, which contradicts
 against Proposition~\ref{prop:aux}. 
 
 We will prove $k\ne2$. Suppose $k=2$. Without loss of
 generality, $\angle v_1 v_6 v_7=\gamma$, $\angle v_7 v_6 v_5 = \angle
 v_5 v_6 v_1=\beta$, because the automorphism $(1\ 5\ 7)(0\ 4\ 2)$ of
 the skeleton of the cube fixes
 the  face-matching edges $v_0 v_1, v_4 v_5$ and $v_2 v_7$.
 Then $\angle v_0 v_5 v_6=\gamma$, $\angle v_6 v_5 v_4 = \alpha$ and $\angle v_5 v_4 v_7=\delta$. Here $\angle v_4 v_5 v_0=\alpha$ or $\delta$. Assume $\angle v_4 v_5
 v_0=\alpha$. Then an opposite inner angle $\angle v_0 v_3 v_4$ is
 $\gamma$. Hence $\angle v_4 v_3 v_2=\beta$. So the inner angle $\angle
 v_7 v_4 v_3$ is $\gamma$. 
The three inner angles around $v_4$ are $\gamma,\delta,X$ for some
 $X\in\{\alpha,\delta\}$,  while the
 three inner angles around $v_5$ are $\alpha,\alpha,\gamma$. 
 So $\alpha=\delta$. This contradicts against
 Proposition~\ref{prop:aux}.  Hence $\angle v_4 v_5 v_0=\delta$.
 Then $\angle
 v_3 v_4 v_5 = \alpha$. Here $\angle v_7
 v_4 v_3 =\beta$ or $\gamma$.  Assume $\angle v_7 v_4 v_3 =\beta$. Then
 the three inner angles around $v_4$ are $\alpha,\beta,\delta$ and those
 around $v_5$ are $\alpha,\gamma,\delta$. So $\beta=\gamma$.  This
 contradicts against Proposition~\ref{prop:aux}. Thus $\angle v_7 v_4
 v_3 =\gamma$.  Hence $\angle v_4 v_3 v_2=\beta$.  $\angle v_2 v_3
 v_0=\gamma$, because the three inner angles around $v_3$ are
 $\beta,\beta,\gamma$. Thus $\angle v_3 v_0 v_1=\delta$. As $\angle v_5
 v_6 v_1=\beta$, an opposite inner angle $\angle v_1 v_0 v_5$ is
 $\delta$. Hence the three inner angles around $v_0$ are
 $\gamma,\delta,\delta$. On the other hand, those around $v_4$ are
 $\alpha,\gamma,\delta$. So $\alpha=\delta$. This contradicts against
 Proposition~\ref{prop:aux}.  Thus, the number $k$ of $\beta$
 around $v_3$ is not two. 
 
 In a similar argument, $k\ne1$.  If $k=3$, then
 $\T\in \TRPZ_3$. Otherwise, $k=0$. But, because the cyclic of the tile
 is $a a a b$~($a\ne b$), we have the symmetry
 $(\alpha,\beta,\gamma,\delta)\leftrightarrow(\delta,\gamma,\beta,\alpha)$. Hence,
 we have $\T\in\TRPZ_3$, too.

Suppose the cyclic list of the edge-lengths of the tile is
 $a a b c$ with $a,b,c$ are mutually distinct. The distribution of the edges
 of length $b$ is the first perfect face-matching of
 Figure~\ref{fig:pfm} without loss of generality. Since every edge of
 length $c$ should be adjacent to an edge of length $b$ and each face
 has exactly one edge of length $c$,  the tiling is
 Figure~\ref{chart of TRPZ} with $n=3$. So, $\T\in \TRPZ_3$.

       \medskip \eqref{assert3:dn}.  In $\T$, two vertices consisting of
  three inner angles $\beta$ are antipodal to each other, because there
  are three congruent paths between them: ``travel straight $a$, bend in
  $\gamma$ angle, travel straight $c$, bend in $-\gamma$ angle, and travel
  straight $a$.''  Actually there is a 3-fold axis $\rho$ of rotation
  through the two vertices, by examining the distribution of $\alpha,\beta,\gamma$
  and the edge-lengths $a,b,c$. $\rho$ is the black vertical axis in
  Figure~\ref{p_expansion}~(left).  Moreover the midpoint of an edge $e$
  of length $b$ is antipodal to the midpoint of the edge $e'$ of length
  $c$ where $e$ is not adjacent to $e'$. It is because there are two
  congruent paths between them: one is ``travel straight $b/2$, bend in
  $\delta$ angle, travel straight $a$, bend in $-\alpha$ angle, travel
  straight $a$, bend in $\beta$ angle, travel straight $c/2$.''  The
  other path is the same with the three angles inverted. Actually an
  axis through the two midpoints is a 2-fold axis of rotation by
  examining the
  distribution of $\alpha,\beta,\gamma,a,b,c$. Similarly we can find
  three $2$-fold axes $\eta_1,\eta_2,\eta_3$ of rotation. Each $\eta_i$
  is a white horizontal axis in Figure~\ref{p_expansion}~(left).

       \medskip
       \eqref{assert3:isohedral}.
       Let $T$ and $T'$ be tiles of $\T$. Let $\rho$ be the vertical 3-fold axis of
       rotation and $\eta_i$~$(i=1,2,3)$ be the horizontal 2-fold axes of rotation, given in \eqref{assert3:dn}. If $\rho$
       is through a point of $T\cap
       T'$, then $T$ is transformed to $T'$ by a rotation around the
       3-fold axis $\rho$. Otherwise, if $T$ and $T'$ are adjacent, then
       $T$ is transformed to $T'$ by some 2-fold axis $\eta_i$ that is
       through an edge $T\cap T'$ of length $b$ or $c$. By repeating
       these transformations, any tile $T$ is transformed to any other
       tile $T'$. So $\T$ is isohedral. This completes the proof of Theorem~\ref{thm:cube}.\qed
     \end{proof}

For spherical tilings by congruent quadrangles, Theorem~\ref{thm:pn} below
provides two necessary and sufficient conditions for spherical tilings such that
the skeletons are pseudo-double wheels to be isohedral. The
two conditions are somehow combinatorial, and come from those given in
Theorem~\ref{thm:cube}.  One is being $\TRPZ_n$, and the other is a condition on
the symmetry operations of tilings.

As in \cite{sakano15:anisohedral}, a \emph{kite}~(\emph{dart}, resp.) is
a convex~(non-convex, resp.)  quadrangle such that the cyclic list of
edge-lengths is $a a b b$ ($a\ne b$), and a \emph{rhombus} is a
quadrangle such that all the edges are equilateral. A kite, a dart and a
rhombus enjoy a mirror symmetry.
    \begin{lemma}\label{lem:no mirror}
     Let $\T$ be a spherical tiling by congruent polygons such that any
     edge is incident to an odd-valent vertex. If the tile does not have
     a mirror symmetry, then neither does $\T$.
   \end{lemma}
   \begin{proof} Assume $\T$ has a mirror plane $\sigma$. Then $\sigma$ does not intersect
   transversely with a tile, since the tile does not have a mirror symmetry.  Thus the intersection of $\sigma$ and $\T$ is the
   cycle of the edges, because each edge of $\T$ is straight. By $\sigma$, each vertex on the
   cycle has even degree. But all the edges of the tiling $\T$ is
   incident to an odd-valent vertex. This is a contradiction. This
    completes the proof of
   Lemma~\ref{lem:no mirror}.\qed\end{proof}
   
 \begin{theorem}\label{thm:pn}
   For 
any spherical tiling  $\T$ by $2n$
  congruent quadrangles~$(n\ge4)$, the following three conditions are
	 equivalent:
	 \begin{enumerate}
	  \item \label{assert:Pn}	$\T\in\TRPZ_n$.
	  \item \label{assert:dn} $\T$ has an $n$-fold axis $\rho$ of
		rotation and $n$ two-fold axes of rotation perpendicular
		to $\rho$.
	  \item \label{assert:isohedral} $\T$ is isohedral
	and the skeleton is the pseudo-double wheel of $2n$ faces.
	 \end{enumerate}
 \end{theorem}
  \begin{proof}

  $(\eqref{assert:Pn}\implies\eqref{assert:isohedral})$
By condition~\eqref{assert:Pn}, we compute the longitude and the
latitude~(i.e., the length of the geodesic segment from the north pole) of the vertices $v_i$'s of $\T$. 
There is an $n$-fold
axis $\rho$ of rotation through the two poles $N$ and $S$, because  there are three congruent paths between them.
 We see that
there is a two-fold axis $\ell_i$ of rotation though the midpoint of the
edge $v_i v_{i+1}$ and the midpoint of the edge $v_{(i+n \mod 2n)}
v_{(i+1+n \mod 2n)}$ and that $\ell_i$ is perpendicular to $\rho$, for
every $i$.  So we have  condition~\eqref{assert:dn}. By this and
   Figure~\ref{chart of TRPZ}, we have
condition~\eqref{assert:isohedral}.

\medskip $(\eqref{assert:dn}\implies\eqref{assert:Pn})$ 
We verify:
\begin{claim}\label{claim:anti}
If $m\ge n\ge4$,  any $m$-fold axis of
rotation of $\T$ is through two vertices.
\end{claim}

\begin{proof}The $m$-fold axis is not though the midpoint of an edge,
  by $m\ne2$.  The $m$-fold axis is not through an inner point
   of a tile. Otherwise all inner angles of the tile is equal, all edges
 are equilateral, $m = 4$. By the premise, $m=n=4$. As the tile is a
 regular quadrangle,
 all diagonal segments of
   the  tiles are less than $\pi$. Otherwise any pair of incident diagonal
   segments crosses to each other.  By drawing exactly one diagonal,
   geodesic segment in each quadrangular tile, we have a spherical
   tiling $\T'$ by $2n \times 2 = 16$ congruent isosceles spherical 3-gons. The
   inner angles of the isosceles 3-gons are $5\pi /8, 5\pi /16, 5\pi
   /16$. It is because the area of the quadrangular tile of the given
   tiling $\T$ is $\pi /2$, and the sum of the four equal inner angles
   is $5\pi /2$. However $\T'$ is impossible, by the
   classification of all spherical tilings by congruent spherical 3-gons~\cite[Table]{MR1954054}. Thus the axis is though a pair of
   antipodal vertices. This completes the proof of
Claim~\ref{claim:anti}. \qed\end{proof}

By
Claim~\ref{claim:anti}, the $n$-fold axis $\rho$ of rotation is through two vertices $u$ and $v$.
Both $u$ and $v$ are $n$-valent. Otherwise, for some positive integers
   $k,n$,
   $k n$
  equilateral
 edges are incident to $u$ and $\ell n$ equilateral edges are incident to $v$,
 because of the $n$-fold axis of rotation through $u$ and $v$. The $k n$
 pairs of neighboring edges incident to the vertex $u$ cause $k n$
 distinct tiles.  Since the number of tiles is $2n$, both of $k$ and
 $\ell$ are one or two. Let $k=2$. Then the $k n$ pairs of
 neighboring edges incident to the vertex $u$ cause already $2n$
 tiles. Then $v$ is not a vertex of any of these $2n$ tiles. To see it,
 assume some tile contains $v$ as a vertex. No vertex is adjacent to
 both $u$ and $v$. Otherwise the inner angle is $\pi$. $u$ is
 not adjacent to $v$, since the length of any edge is less than
 $\pi$. So if two vertices of a tile of $\T$ are incident to $u$, then
 some vertex other than $v$ is incident to them, because the tile is a
 quadrangle.  Thus the number of tiles is greater than $2n$.  So $k=\ell
   =1$.

   \medskip Hence there are exactly $n$ vertices $w_i$ $(0\le i\le n-1)$
   adjacent to $u$. All edges $u w_i$'s are equilateral by the $n$-fold
   axis of rotation through $u$ and $v$.  We assume that $u w_{i+1\bmod
   n}$ is next to $u w_i$, and that the two vertices $w_{i+1\bmod n}$
   and $w_i$ are adjacent to a vertex $v_i$. Let $T_i$ be a tile $u w_i
   v_1 w_{i+1 \bmod n }$. By the
 $n$-fold axis $\rho$, all $v_i$'s are distinct.

 \begin{claim}\label{claim:dn implies pdw} $v_i$  is
adjacent to  $v$~$(0\le i<n)$.
 \end{claim}
 \begin{proof} 
 Otherwise, there is a non-pole vertex $\tilde u_i$ adjacent to $v_i $ such that an
 edge $v_i  \tilde u_i$ is a neighbor of $v_i  w_{i+1 \bmod n }$ without loss of generality.  Then
 there is a quadrangular tile $T'_i$ having the three vertices $\tilde
  u_i,v_i ,w_{i+1 \bmod n }$.
  \begin{figure}
    \begin{tikzpicture}[xscale=1.3,yscale=0.4]
\draw (0,0)node[left] {$u$} -- (1,1) node[above] {$w_{i+1 \bmod n }$} -- (2,0)
     node[right] {$v_i $} -- (1,-1) node[below] {$w_i$} --cycle;
     \node at (1,0) {$T_i$};
     \draw (2,0) -- (3,1) node[right] {$\tilde u_i$} -- (2,2)
     node[above] {$x_i $};
     \draw[dashed] (2,2)-- (1,1) ;
     \node at (2,1) {$T'_i$};
     \node at (4,0) {$v$};
    \end{tikzpicture}
   \caption{Proof of Claim~\ref{claim:dn implies pdw}.}
  \end{figure}
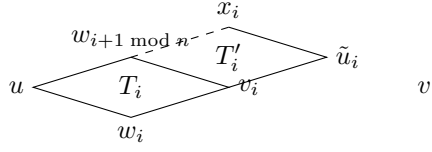

 The other vertex, say $x_i $, of the tile $T'_i$ is not the vertex $v$. Otherwise, an edge
 $w_{i+1 \bmod n } x_i $ is transformed to an edge $w_i v$ by the
  $n$-fold axis $\rho$ of
 rotation of $\T$.  The edge $w_i
 v$ cannot be transversal to the edge $v_i  \tilde u_i$. Hence, the
 lune~(that is, digon) determined by the two edges $u w_i$ and $u w_{i+1 \bmod n }$
 contains a tile other than $T_i= u w_i v_i  w_{i+1 \bmod n }$ and $T'_i=v_i  \tilde u_i v
 w_{i+1 \bmod n }$. Thus  $\T$ has more than $2n$ number of tiles, which is absurd. So $x_i \ne v$.
 
The vertex $x_i$ of the tile $T'_i$ is not the vertex $u$. Otherwise, an edge
 $x_i \tilde u_i$ has the same length as $u w_i$ and $u w_{i+1 \bmod n }$. Since the
 edge $u w_{i+1 \bmod n }$ is a neighbor of the edge $u w_i$, the edge $x_i  \tilde u_i$
 is a neighbor of $u w_i$ or $u w_{i+1 \bmod n }$.  Consider the latter case.  By the
 $n$-fold axis $\rho$ through $u$ and $v$, the tile $T_i=u w_i v_i  w_{i+1 \bmod n }$ is rotated
 to $x_i  w_{i+1 \bmod n } v_i  \tilde u_i$. But this is impossible because the vertices
 $w_i$ and $w_{i+1 \bmod n }$ of the former tile $T_i=u w_i v_i  w_{i+1 \bmod n }$ and the vertices $w_{i+1 \bmod n }$ and $\tilde
 u_i$ of the latter tile $T'_i=x_i  w_{i+1 \bmod n } v_i  \tilde u_i$
  are all incident to the vertex $v_i $. So $w_{i+1 \bmod n }$
 becomes two-valent. When the edge $x_i \tilde u_i$ is a neighbor of $u w_i$,
 we have similarly a contradiction. Thus $x_i \ne u$.

Any vertex of the tile $T'_i= w_{i+1 \bmod n } v_i  \tilde u_i x_i $ is neither the
$n$-valent $u$ nor the $n$-valent $v$, so the number of the tiles of the
tiling $\T$ is greater than $2n$. This is absurd. Thus the
vertex $v_i$ is adjacent to  $v$. This establishes
Claim~\ref{claim:dn implies pdw}.\qed
 \end{proof}

By Claim~\ref{claim:dn implies pdw}, the
skeleton of $\T$ is the pseudo-double wheel of $2n$ faces. By
the $n$-fold axis $\rho$ through $u$ and $v$, all $n$ edges incident to
the vertex $u$ have the same length $a$, and all $n$ edges incident to
the vertex $v$ have the same length $a'$.

By the assumption, $\T$ has $n$ horizontal two-fold rotation axes,
each through a pair of midpoints of edges. As they swap the vertices $u$
and $v$, we have $a=a'$. By computing the longitude and the latitude of
each non-pole vertices, the angle-assignment and the length-assignment
of $\T$ is exactly as in Figure~\ref{chart of TRPZ}.

   \medskip $(\eqref{assert:isohedral}\implies\eqref{assert:dn})$
   Suppose that the tile of $\T$ is a rhombus, a kite, or a dart.  By
   the classification of spherical monohedral (kite/dart/rhombus)-faced
   tilings~\cite[Table~1]{sakano15:anisohedral}, Sch\"onflies symbol~(\cite{cotton09:_chmic_applic_of_group_theor},\cite{Deza}) of $\T$ is $D_{nd}$. 
   In the decision tree~\cite[Fig.~3.10]{cotton09:_chmic_applic_of_group_theor}, by going from the leaf ``$D_{nd}$'' to the root,
   we see that $D_{nd}$ must have ``$n$ $C_2$'s $\perp$ to $C_n$'' ($n$ two-fold axes of rotation perpendicular to an $n$-fold axis of rotation).
   Thus \eqref{assert:dn} holds. 
   By the same reasoning, Sch\"onflies symbol $D_n$ requires \eqref{assert:dn}.
   So, to complete the proof of $(\eqref{assert:isohedral}\implies\eqref{assert:dn})$, we  show: if the tile of $\T$ is none of a kite, a dart
   and a rhombus, then  $\T$ has Sch\"onflies symbol $D_n$.

Sch\"onflies symbol of $\T$ is none of $T$, $T_d$, $T_h$, $O$, $O_h$,
   $I$, and $I_h$.  Otherwise, the tiling $\T$ has more than three
   three-fold rotation axes, by
   \cite[Sect.~3.14]{cotton09:_chmic_applic_of_group_theor}. Since the skeleton of
   $\T$ is the pseudo-double wheel of $2n$ faces~$(n\ge4)$, $\T$ has
   only two vertices $N$ and $S$ of valence more than three. So there is
   a three-fold axis $\rho$ of rotation through a three-valent
   vertex. Thus the rotation in $2\pi/3$ around $\rho$ transforms
   $N$~($S$, resp.) to $S$~($N$, resp.), or fixes both of $N$ and
   $S$. So the rotation in $4\pi/3$ around $\rho$ fixes both of $N$ and
   $S$.  This is absurd, since the three-fold rotation axis is through
   neither $N$ nor $S$.

   In any pseudo-double wheel, any edge is incident to an odd-valent
   vertex. Because we assumed that the tile of the tiling $\T$ on a
   pseudo-double wheel
   is none of a rhombus, a
   kite,  and a dart, the tile has no mirror symmetry. By
   Lemma~\ref{lem:no mirror}, $\T$ has no mirror symmetry.

   So Sch\"onflies symbol of the tiling $\T$ is $C_m$ or $D_m$ for some
   integer $m\ge2$. This is due to
   the systematic procedure to determine the
   Sch\"onflies
   symbol~\cite[Sect.~3.14]{cotton09:_chmic_applic_of_group_theor}. Then
     the tiling $\T$ has an
   $m$-fold axis $\rho$ of rotation. Let $G$ be the symmetry
   group  of the tiling $\T$. Because the tiling $\T$ is
   isohedral, $G$ acts transitively on the tiles of $\T$. So 
   \begin{quote}
   (\#)\quad  the order $\#G$ is a multiple of the number
   $2n\ge8$ of tiles of $\T$. 
   \end{quote}

Assume Sch\"onflies symbol of $\T$ is  $C_m$ for some $m\ge2$.  By
   \cite[p.~41]{cotton09:_chmic_applic_of_group_theor}, $\#G=m$.  By (\#),
   $m\ge8$. $\rho$ is through a vertex with
   the valence being a multiple of $m$. So the $m$-fold axis $\rho$ of rotation
   is through the poles $N$ and $S$, and thus $m=n$.  The symmetry
   operations of $\T$ are exactly $m$ rotations around
   $\rho$ by $C_m$~\cite[p.~41]{cotton09:_chmic_applic_of_group_theor}.
No symmetry operation of $\T$ transforms a  tile having $N$ as a vertex
   to a tile having $S$ as a vertex. However $\T$
   is isohedral.

   Thus Sch\"onflies symbol of the tiling $\T$ is $D_m$ for some
   $m\ge2$.
By \cite[p.~41]{cotton09:_chmic_applic_of_group_theor}, $\#G=2m$. 
By (\#), $m$ is a multiple of $n\ge4$. So the $m$-fold
   axis $\rho$ of rotation is not through a three-valent vertex of
   $\T$, but through $N$ and $S$ of the pseudo-double wheel, and
   $m=n$. Hence the condition~\eqref{assert:dn} holds.  This completes
   the proof of (\eqref{assert:isohedral}$\implies$\eqref{assert:dn}).
 \qed\end{proof}

\section{Tiles of spherical isohedral tilings over pseudo-double wheels \label{sec:tiles}}

 \begin{definition}\label{def:pf}For $n\ge3$, 
  an \emph{$\TRPZ_n$-quadrangle} is the tile of some $\T\in\TRPZ_n$.
      \end{definition}
      
  \begin{fact}\label{fact:q}For given $n\ge3$, $\alpha,\gamma\in (0,\,\pi)\cup (\pi,\, 2\pi)$ and
  $a\in (0,\,\pi)$, there is at
  most one $\TRPZ_n$-quadrangle, modulo $SO(3)$, such that
  \begin{itemize}\item
  the cyclic list of inner angles in the clockwise order is
$$(\alpha,\beta,\gamma,\delta)=(\alpha,2\pi/n,\gamma,2\pi-\alpha-\gamma)$$
		      $($cf. Figure~\ref{chart of TRPZ}$)$;
		      and

		      \item the edge $\alpha\beta$, the edge $\beta\gamma$, and the geodesic
  segment $\beta\delta$ have length
  $a,a,\pi-a$. \end{itemize}
\end{fact}
\begin{proof}
  From a point $N$ on the unit sphere, travel in distance $a$, bend counterclockwise in $\pi-\alpha$, and
  travel in $2\pi$. Then, by the last travel, we have a great circle $C$. 
  The bending angle intends the inner angle $\alpha$.
    By abuse of notation, we denote the bending point by $\alpha$. $C$ is through the point $\alpha$.
 Similarly, from $N$, travel in distance $a$. Here the angle of this travel from
 the travel $N\alpha$ of length $a$ is $\beta=2\pi/n$. By abuse of
 notation, we often write $\beta$ for the vertex $N$.
 Then
 bend clockwise in $\pi-\gamma$, and
  travel in $2\pi$. By the last travel, we have a great circle $C'$. By abuse of notation, we denote the bending point by $\gamma$. $C'$ is through the point $\gamma$.

  Then
  $C\ne C'$ by $\delta\ne\pi$. So
  $C$ and $C'$ share exactly two points $P,P'$. If each of $P$ and $P'$ is a vertex of the
  $\TRPZ_n$-quadrangle, then the inner angle of an $\TRPZ_n$-quadrangle
  $\alpha\beta\gamma P$
  which is diagonal to $P$ is $\beta$ if and only if the inner angle of
  an $\TRPZ_n$-quadrangle  $\alpha\beta\gamma P'$ diagonal to $P'$ is
  $2\pi-\beta$. In this case, $P'$ is inappropriate, as the tiles must
  not overlap.  Hence,  for $n,\alpha,\gamma,a$, there is at most one
  pair of $b,c$. Actually, $b$ is determined from $\alpha,a$ by a spherical cosine law~(Proposition~\ref{prop:vinberg}~\eqref{assert:scl})
  $\cos(\pi-a)= \cos a \cos b + \sin a \sin b \cos\alpha$, and $c$ is determined similarly.\qed
  \end{proof}
 \begin{definition}\label{def:ispdw Q}
  Let  $\PQ_{n,\alpha,\gamma,a}$ be an $\TRPZ_n$-quadrangle of Fact~\ref{fact:q}. We identify
  $\PQ_{n,\alpha,\gamma,a}$ modulo $SO(3)$.
 \end{definition}

 In fact, any $\TRPZ_n$-quadrangle is
 specified without mentioning a tiling of $\TRPZ_n$, as in the following
 Fact. There the vertices $A,B,C,D$ intend the vertices $N, v_0, v_1, v_2$
 of a tiling of $\TRPZ_n$.
\begin{fact}
 \label{fact:Pquad}
\begin{enumerate}
\item \label{assert:quadrangle} An $\TRPZ_n$-quadrangle is exactly a
        quadrangle $ABCD$ such that $AB=\pi - AC = AD$, the area
  of $ABCD$ is $2\pi/n$, and the inner angle $A$ is $2\pi/n$.
 
\item \label{assert:bijective} The set of $\TRPZ_n$-quadrangles
 bijectively corresponds to $\TRPZ_n$.
 \end{enumerate}
\end{fact}

 \begin{proof} \eqref{assert:quadrangle} As the edges $BC$ and $CD$ have
  length less than $\pi$, we have two spherical 3-gons $ABC$ and
  $CDA$. As noted in the caption of Figure~\ref{fig:cs},  $\angle ABC
  = \pi - \angle BCA$ and $\angle ADC = \pi - \angle DCA$. So the three
  inner angles of the vertices $B,C,D$ sum up to $2\pi$. Thus the inner
  angle of the vertex $A$ is $2\pi/n$ since the area of $ABCD$ plus
  $2\pi$ is the sum of all inner angles $A,B,C,D$. By regarding the
  inner angles $A,B,C,D$ as $\beta,\alpha,\delta,\gamma$ and then
  arranging the $2n$ copies of the quadrangle $ABCD$ as
  Figure~\ref{chart of TRPZ}, we conclude $ABCD$ is a tile of a tiling
  of $\TRPZ_n$.  \eqref{assert:bijective} Clear.\qed
\end{proof}

An edge incident to $N$ or $S$ is
called a \emph{meridian edge}.

  \begin{lemma}\label{lem:shorta}
   Suppose $n\ge3$, $\alpha,\gamma\in (0,\,\pi)\cup(\pi,\, 2\pi)$, and
   $a\in (0,\pi)$. Every $\TRPZ_n$-quadrangle
 $\PQ_{n,\alpha,\gamma,a}$ satisfies $a\ne\fra{\pi}{2}$, $\alpha\ne\fra{\pi}{2},\gamma\ne\fra{\pi}{2}$,
\begin{align}
& 0<a<\frac{\pi}{2}\iff 0<\delta<\pi;\ \mbox{and}\label{cond:deltaconvex}\\
&\alpha>\pi\ \mbox{or}\ \gamma>\pi\ \implies\
                 \frac{3\pi}{2} > \alpha > \pi > \gamma > \frac{\pi}{2}
\ \ \mbox{or}\ \ \frac{3\pi}{2} > \gamma > \pi > \alpha > \frac{\pi}{2}.
 \label{assert:alphagamma}
\end{align}
\end{lemma}
 \begin{proof} Consider a tiling of $\TRPZ_n$.
  If $a=\pi/2$, then any
 tile has three vertices on the equator and the other vertex is a
 pole. This contradicts against
 the condition ``no inner angle is $\pi$''~(see Section~\ref{sec:basic definitions}).

Assume $\gamma=\fra{\pi}{2}$. See Figure~\ref{fig:cs}. $N$ and $S$ are the poles, and $\angle N v_2 v_1 = \angle S v_1
 v_2=\fra{\pi}{2}=\gamma$. 
Then $\angle N v_1 v_2 = \pi - \angle S v_1 v_2=\pi/2$.
Thus $N v_1 v _2$ is an isosceles triangle.
Hence $\pi - a =N v_1 = N v_2=a$.
This contradicts against
 $a\ne\fra{\pi}{2}$ which we have already proved.
  Hence $\gamma\ne\fra{\pi}{2}$.
 Similarly, $\alpha\ne\fra{\pi}{2}$.

As for equivalence~\eqref{cond:deltaconvex},
 $a\in (0,\, \pi/2)$ if and only if
$v_0$ and $v_2$ are located in the northern hemisphere and $v_1$ is in the
 southern. This is equivalent to $\delta\in (0,\pi)$. 
We will prove the implication~\eqref{assert:alphagamma}. First assume the case where $\gamma$ is too large.
Then, the vertex $v_1$
 is in the northern hemisphere and the edge $v_0 v_1$ crosses to the
 edge $N v_2$. 
 To think of the situation, the leftmost lower tiling in Figure~\ref{fig:af} may be useful. 
 In the critical situation, 
 $\alpha+\gamma+\delta=2\pi$ implies $\alpha=\angle v_1 v_0 N = \pi/2$,
 $\gamma=\fra{3\pi}{2}$, and $\delta=0$. So $\fra{\pi}{2}<\alpha<
 \pi<\gamma <\fra{3\pi}{2}$. The same inequalities with $\alpha$ and
 $\gamma$ swapped follows when $\alpha$ is too
 large. \qed
 \end{proof}

For $\T\in\TRPZ_n$,
let $a$ be the length of the geodesic segment between $N$ and $v_0$, and let $\varphi$ be the
longitude of the vertex $v_2$ minus that of the vertex $v_1$. See Figure~\ref{fig:cs}.

 \begin{definition}\label{def:ani}For $n\ge3$, define open sets $A^{(i)}_n$ in $\Rset^2$
  $(i=1,2,3,4)$ as $\left(\frac{2\pi}{n}-\pi,\,
 0\right)\times\left(0,\,\frac{\pi}{2}\right)$,
$\left(0,\, \frac{2\pi}{n}\right)\times\left(0,\,\frac{\pi}{2}\right)$,
$\left(\frac{2\pi}{n},\, \pi\right) \times \left(0,\,
  \frac{\pi}{2}\right)$, and
$\left(0,\, \frac{2\pi}{n}\right)\times\left(\frac{\pi}2,\, \pi\right)$.
  Let $A_n$ be $\bigcup_{i=1}^4 A_n^{(i)}$. See Figure~\ref{fig:af}.
 \end{definition}

 \begin{figure}[ht]\centering
 \includegraphics[width=5cm,height=5cm]{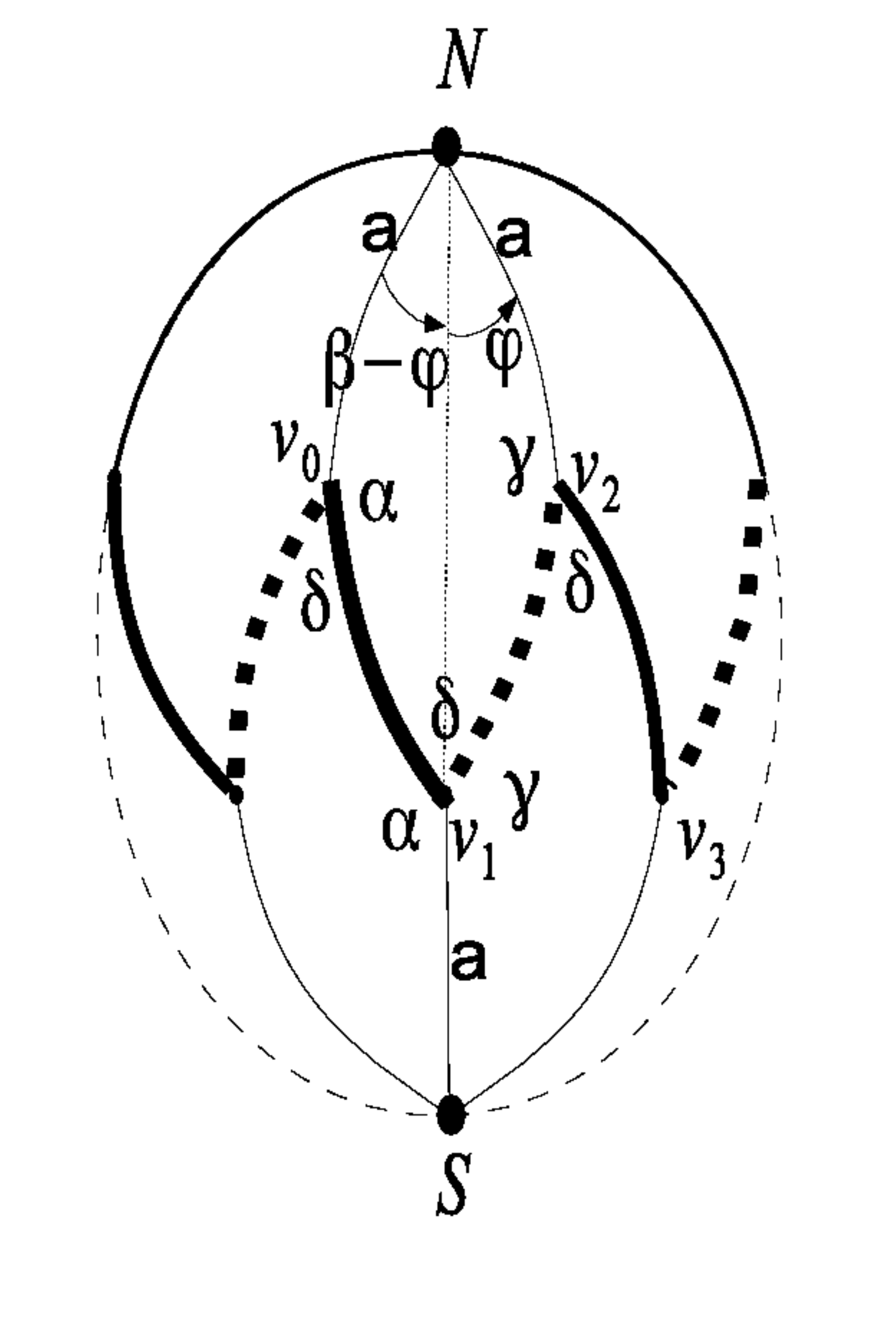}
   \caption{The coordinate system $\langle\varphi,a\rangle$ of a tiling  $\T$ of
$\TRPZ_n$.  See the caption of Figure~\ref{chart of TRPZ}. 
 Possibly $\varphi<0$ and possibly $\varphi>\beta$. Because a straight line from $N$ to
 the antipodal vertex $S$ is through $v_1$, and because $v_1 S=a$, we have $N v_1=\pi -a$ , $\angle v_2 v_1
N=\pi - \gamma$ and  $\angle N v_1 v_0=\pi - \alpha$.
 \label{fig:cs}}
 \end{figure}

\begin{theorem}[A coordinate system of $\TRPZ_n$]
\label{thm:phia}For each integer $n\ge3$, a function
 $\T\in\TRPZ_n\mapsto\coordi{a}{\varphi}{2n}\in A_n$ is a bijection.
\end{theorem}

\begin{proof} We prove  $\coord{a}{\varphi}\in A_n$ for any $\T\in\TRPZ_n$, as
follows: For $\pi/2<a<\pi$, we have $0<\varphi< \beta=\fra{2\pi}{n}$.
Otherwise, we can see that an edge crosses to an opposite edge. 
To think of the situation, see the right upper tiling in Figure~\ref{fig:af}. 

For $0<a<\pi/2$, $\varphi$ is
strictly between $\beta-\pi = 2\pi/n - \pi$ and $\pi$. Otherwise one of the
edge $v_0 v_1$ and the edge $v_1 v_2$ contains a pair of antipodal
points. Obviously $a\ne 0$. By Lemma~\ref{lem:shorta}, $a\ne
 \pi/2$. $\varphi\ne\beta=2\pi/n$, by $\alpha\ne\pi$.

We show that the function $\T\mapsto\coord{a}{\varphi}$ is onto
 $A_n$. Take an arbitrary $\coord{a}{\varphi}$ of $A_n$.  We first construct 
 a quadrangle as follows: Take a point $v_2$ on the sphere such
 that the geodesic segment $v_2 S$ has length $\pi - a$. Since $\varphi$ is
 given and $\beta=2\pi/n$ is known, the vertex $v_1$ and $v_0$ is
 determined, as in Figure~\ref{fig:cs}. 

 When $0<a<\pi$, a pair of antipodal
 points appears neither in the edge $N v_0$ nor in the edge $N
 v_2$.   No inner angle is $\pi$, as $\varphi\ne 0, \fra{2\pi}{n}$ and $a\ne\pi/2$.

We verify no edge contains a pair of antipodal points.  Since $\coord{a}{\varphi}$ is in the union $A_n$ of the four
open rectangles of Figure~\ref{fig:af}, a hemisphere contains all the vertices $v_0, v_1, v_2$ and
the pole $N$ as inner points.  Hence, a pair of
antipodal points appears in neither the edge $v_1 v_2$ nor the edge $v_0 v_1$,
and lengths of the edges $N v_0$ and $N v_2$ are $a<\pi$.

Moreover, any of the four edges of the tile does not cross to the
 opposite edge, because when $0<a<\pi/2$ the
 vertex $v_1$ is located in the southern hemisphere and the edges $N
 v_0$ and $N v_2$ are in the northern hemisphere. On the other hand,
 $\pi/2<a<\pi$ implies $0<\varphi<2\pi/n$.

Arranging the $2n$ copies of the quadrangle as Figure~\ref{fig:af}
 results in a tiling of $\TRPZ_n$. So the function
 $\T\mapsto\coord{a}{\varphi}$ is onto $A_n$. $\pi - a$ is the distance
 of the vertex $v_1$ from the pole $N$ while $2\pi/n - \varphi$ is the
 longitude of $v_1$, i.e., $\angle  v_0 N v_1$. So $\T\mapsto\coord{a}{\varphi}$ is injective.  Hence
 Theorem~\ref{thm:phia} is proved.\qed
\end{proof}

\begin{figure}[ht]\centering
    \begin{tikzpicture}[axis/.style={very thick, ->, >=stealth'}]
     \def\acs{-.71}
    \node at (-2.6,1.3)
    {\pgftext{\includegraphics[scale=.1]{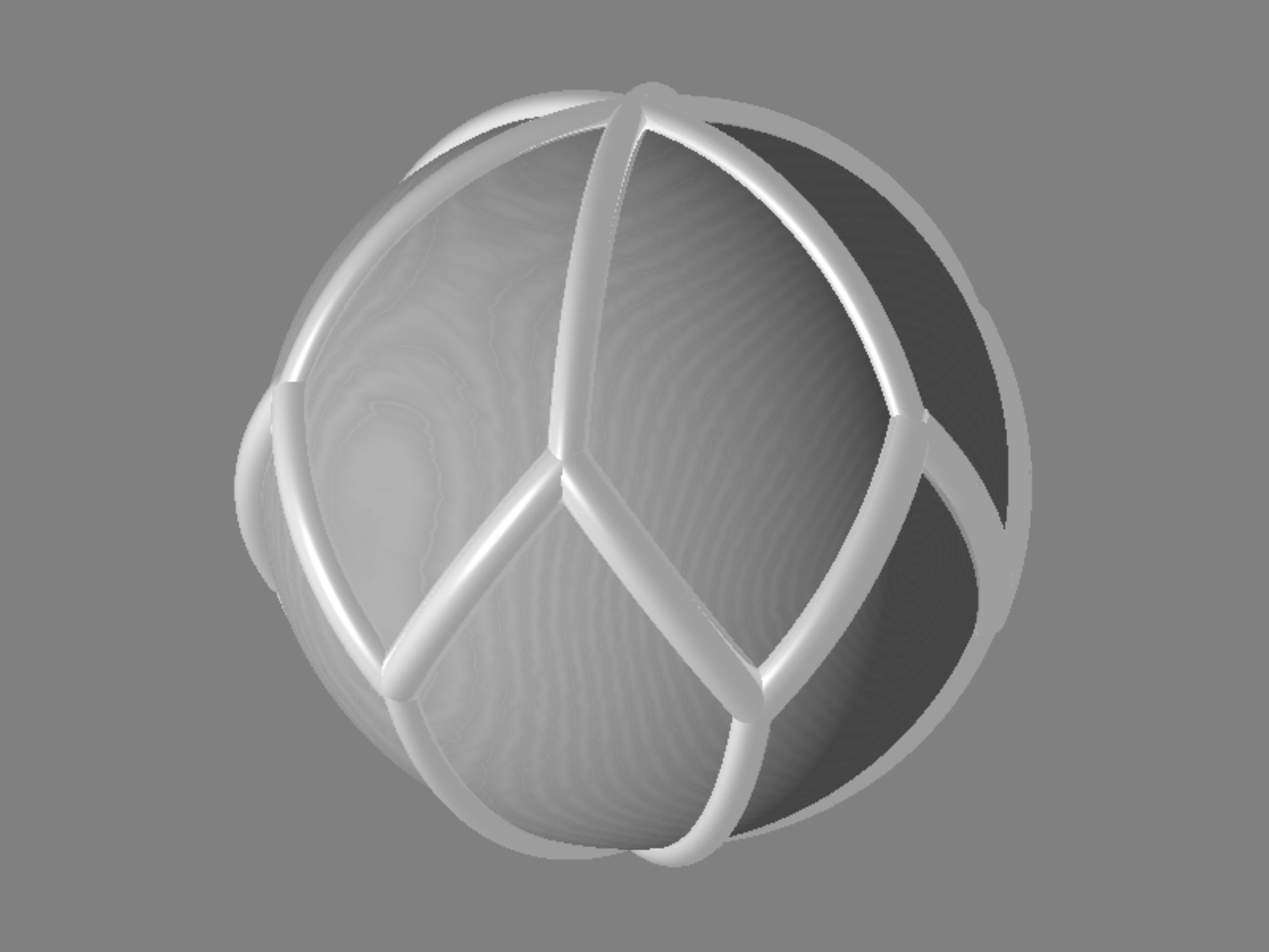}}};
    \node at (-2.5,1.8) {$\beta$};
    \node at (-2.6,1.4) {$\alpha$};
    \node at (-2.2,1.4) {$\gamma$};
    \node at (-2.3,1.1) {$\delta$};
    \draw[<-] (-.1,\acs) -- (-1.2,1.3);
    \node at (-.1,\acs) {$\circ$};
    \node at (2.4,1.3)
    {\pgftext{\includegraphics[scale=.1]{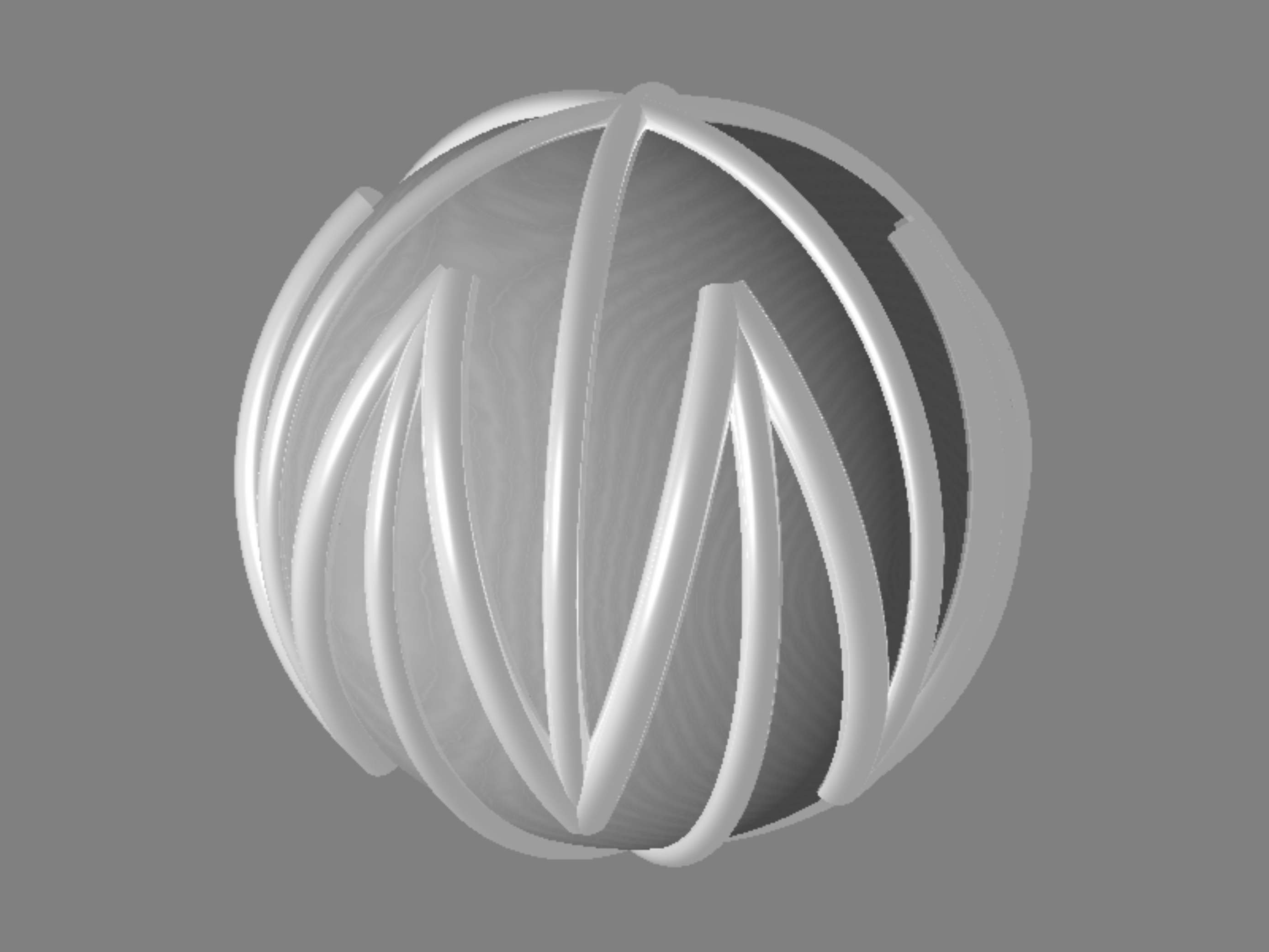}}};
    \node at (2.4,1.9) {$\beta$};
    \node at (2.4,1.3) {$\alpha$};
    \node at (2.9,1.4) {$\gamma$};
    \node at (2.6,1.8) {$\delta$};

    \draw[<-] (-.1,0.65) -- (1,1.3);
    \draw[dotted, thick] (-.7,0.65) -- (-.1,0.65);
    \node at (-.1, 0.65) {$\times$};
    \draw[dotted, thick] (-.1,0.65) -- (-.1,-2.2);
    \node at (-.1,-2.5) {$\frac{\pi}{n}$};
    
    \node at (-1.9,-1.2) {$A^{(1)}_n$};
    \node at (0,-1.2) {\Large $A^{(2)}_n$};

        \draw[dashed,line width=1mm] (.5,-2.2) -- (.5,1.6) --(-.8,1.6)
    -- (-.8,-2.2); 

    \node at (-1,1.6) {\Large $\pi$};
    
     \draw[dashed, line width=1mm] (2.8,-2.2) -- (-3.1,-2.2) --    (-3.1,-.3)-- (2.8,-.3)--cycle;
    
    \node at (-4.6,-.9)
    {\pgftext{\includegraphics[scale=0.1]{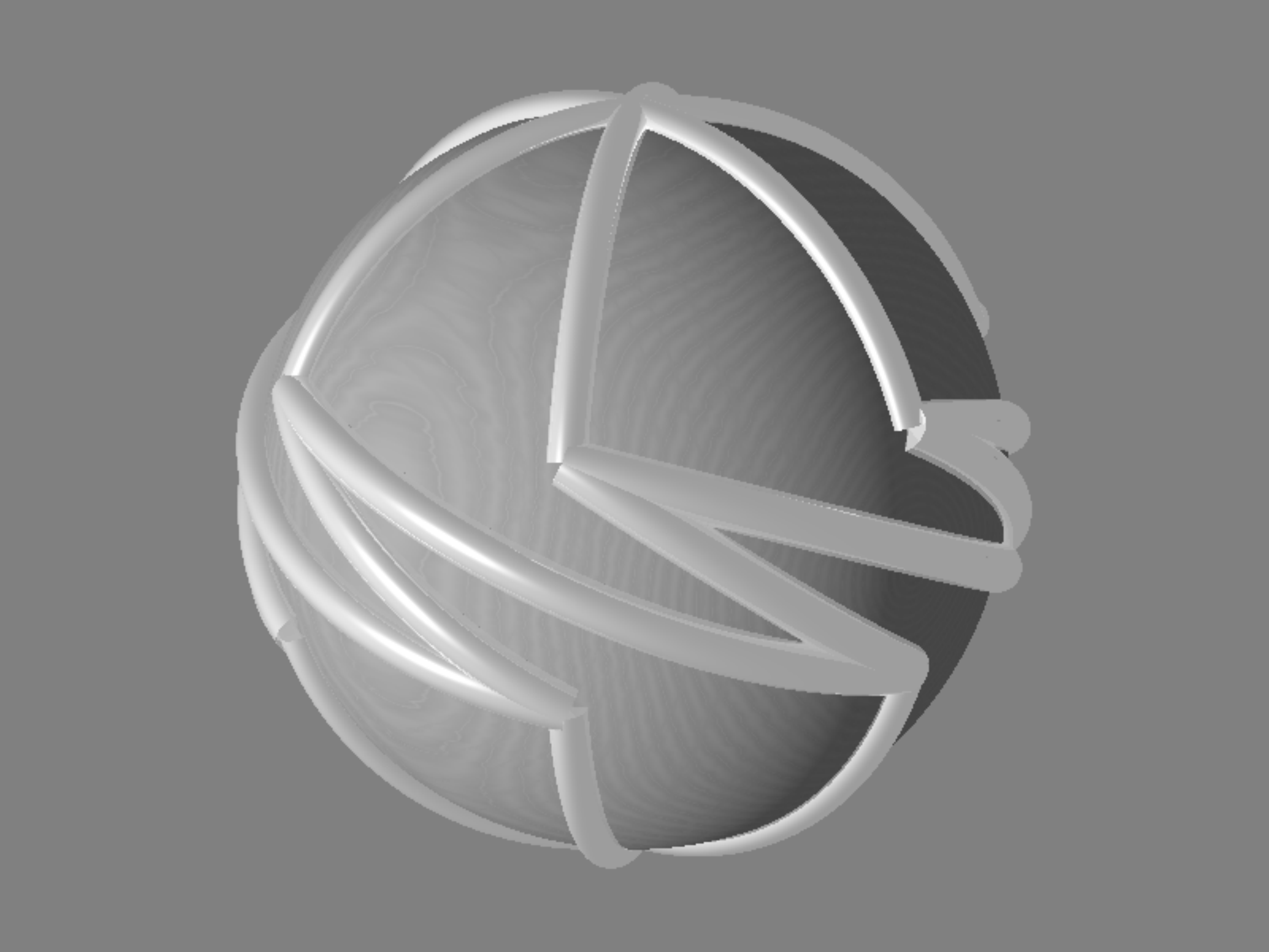}}};
    \draw[->] (-3.5,-1.2) -- (-1.9,\acs);
    \draw[dotted,thick]     (-1.9,\acs) -- (-1.9,-2.2);
    \node at (4.5,-.9) {\pgftext{\includegraphics[scale=0.1]{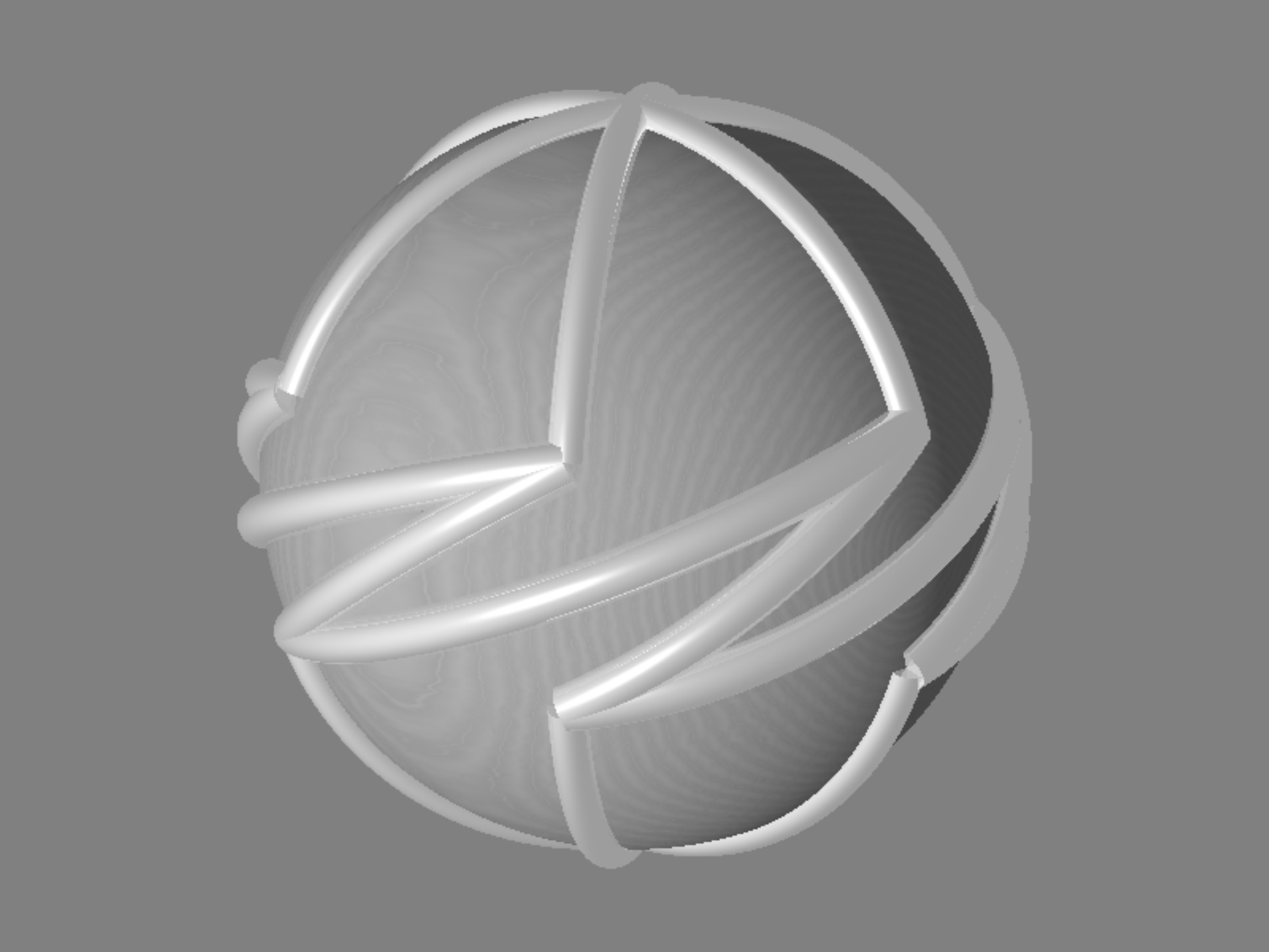}}};

    \node at (-4.6, .1) {$N$};
    \node at (-5.2,-.69) {$v_0$};
    \node at (-4.9,-.9) {$\gamma$};  \node at (-4.6,-.83) {$v_2$};
    \node at (-3.9,-1.4) {$v_1$};
    \node at (4.5,-.9) {$\alpha$};
    \draw[->] (3.1,-1.2) -- (1.7,\acs);
    \node at (-1.5+.6,\acs-.1) {$\arccos\frac{1}{3}$};
    \draw[dotted, thick] (-1.9,\acs) -- (1.7,\acs);
    \node at (-1.9,-2.5) {$-\frac{\pi}{3}$};
    \draw[dotted,thick] (1.7,\acs) -- (1.7,-2.2);
    \node at (1.7,-2.5) {\Large $\frac{2\pi}{3}$};
    \node at (-1,0) {\Large $\frac{\pi}{2}$}; 
    
    \node at (-3.1,-2.5) {$\frac{2\pi}{n}-\pi$};
    \node at (-.7,-2.5) {$0$};
    \node at (.5,-2.5) {$\frac{2\pi}{n}$};
    \node at (2.8,-2.5) {$\pi$};

    \node at (2,-1.2) {\Large $A^{(3)}_n$};
    \node at (0,1.2) {\Large $A^{(4)}_n$};

\draw[axis] (-6, -2.2) -- (5.5, -2.2) node(xline)[right] {$\varphi$ };
    \draw[axis] (-.8, -2) -- (-.8, 3) node(yline)[above] {$a$};

    \end{tikzpicture}
 \caption{The open set $A_n$~($n=6$)~(Definition~\ref{def:ani}) of $\coordi{a}{\varphi}{12}$~(Figure~\ref{fig:cs}) of
 $\T\in\TRPZ_n$.  $A_n$ bijectively corresponds to $\TRPZ_n$~(Theorem~\ref{thm:phia}). The four
 images are tilings of $\TRPZ_n$. 
   \label{fig:af} }
 \end{figure}

\section{Quadratic equation of tiles}\label{sec:q}

\begin{theorem} \label{thm:existence Pn} Suppose $n\ge3$, $\alpha,\gamma\in
 (0,\,\pi/2)\cup (\pi/2,\, \pi)\cup(\pi,\, 3\pi/2)$, $a\in
 (0,\,\pi/2)\cup (\pi/2,\, \pi)$. Then
a quadrangle is an $\TRPZ_n$-quadrangle $\PQ_{n,\alpha,\gamma,a}$, if and only if
 $f_{n,\alpha,\gamma}(\cos a)=0$ where
\begin{align*}
 f_{n,\alpha,\gamma}(x)&:=  x^2 - \left(\cot\frac{\pi}{n}\right) (\cot \alpha + \cot \gamma ) x
- \cot \alpha \cot \gamma .  
\end{align*}
\end{theorem}

    \begin{proof} Assume we are given a quadrangle $N v_0 v _1 v_2$.
By our definition of quadrangles~(see Section~\ref{sec:basic definitions}),  the
    quadrangle is a subset of the interior of an hemisphere. So,
      $N v_0 v_1$ and
     $N v_1 v_2$ are spherical 3-gons.  Let $\varphi$ be the angle from
     a geodesic segment $N v_1$ to the edge $N v_2$ and $\varphi'$ be
     the angle from the edge $N v_0$ to the geodesic segment $N v_1$.

    The given quadrangle is $\PQ_{n,\alpha,\gamma,a}$, if and only if there are $\varphi$ and $\varphi'$ such that
      \begin{align}
&      \varphi+\varphi'=\frac{2\pi}{n},\ \varphi\ne0,\ \varphi'\ne0\label{eq:w}\\
&\cos\gamma= \cos \varphi \cos\gamma  - \sin \varphi  \sin \gamma  \cos a,\  \mbox{and}\label{eq:u}\\
&\cos\alpha= \cos  \varphi' \cos\alpha - \sin  \varphi' \sin  \alpha  \cos
      a.\label{eq:v}\end{align}
Two equations~\eqref{eq:u} and \eqref{eq:v} are equivalent to
spherical cosine theorems for angles~(Proposition~\ref{prop:vinberg}~\eqref{assert:scla})
 to spherical 3-gons.  It is
  because of applying the last two equations $\angle v_2 v_1
 N=\pi - \gamma$ and  $\angle N v_1 v_0=\pi - \alpha$ in the
  caption of Figure~\ref{fig:cs}.

    In the $xy$-plane, consider two lines
  \begin{align*}
   \ell:\ x - y\tan \gamma \cos a = 1, \qquad m:\ x - y \tan\alpha \cos
   a = 1.
  \end{align*}
  They are well-defined, by the premise. Then
  \begin{align*}
(*)\qquad   \eqref{eq:u} \iff (\cos\varphi,\, \sin\varphi)\in \ell,\qquad
   \eqref{eq:v} \iff (\cos\varphi',\, \sin\varphi')\in m. 
  \end{align*}
Let $R$ be  the reflection with respect to the $x$-axis  followed by rotation in $2\pi/n$ around the origin $O$. 
Then,  \eqref{eq:w} implies
$\eqref{eq:v} \iff (\cos\varphi,\, \sin\varphi)\in \Refl{m}$.
  To sum up, under equation~\eqref{eq:w},
  \begin{align}
\label{ctldag}\qquad  \eqref{eq:u}\ \&\ \eqref{eq:v}\ \iff\ (\cos\varphi,\, \sin\varphi)\in \ell \cap \Refl m.
  \end{align}
  Let $P$ be a point $(1,\, 0)$ and $C$ be the unit circle
  $x^2+y^2=1$. 
  \begin{claim}
\label{claim:P'Q'}
  \begin{enumerate}
   \item \label{claim:P'Q'1}For all
	 $\alpha,\gamma$, there is
	 a unique point $P'\in C\cap
	 \ell  \setminus\{P\}$. Moreover $P'=(\cos\varphi,\, \sin\varphi)$.
	 
   \item \label{claim:P'Q'2}For all
	 $\alpha,\gamma$, there is a  point $Q'\in C\cap
	 m  \setminus\{P\}$. Moreover $Q'=(\cos\varphi',\, \sin\varphi')$ and $\angle POQ'=\varphi'$.
\end{enumerate}
  \end{claim}
  \begin{proof} \eqref{claim:P'Q'1}. 
By the premise, $\varphi\ne0$. So $P'$ is unique.
   By equivalence~($\ast$), $\varphi=\angle POP'$.
\eqref{claim:P'Q'2}.  Similar to \eqref{claim:P'Q'1}.  \qed
  \end{proof}

Let $S$ be a point on the $x$-axis in the $xy$-plane. The ray starting
     from $S$ in the direction of the positive part of $x$-axis is
     denoted by $xS$ or $Sx$.
  The sum of the three inner angles of the plane triangle $O P P'$ is $\pi$. So,
\begin{align}  
\label{ddag} u:=\angle x P P'= 
 \frac{\varphi+\pi}{2}.
\end{align}

  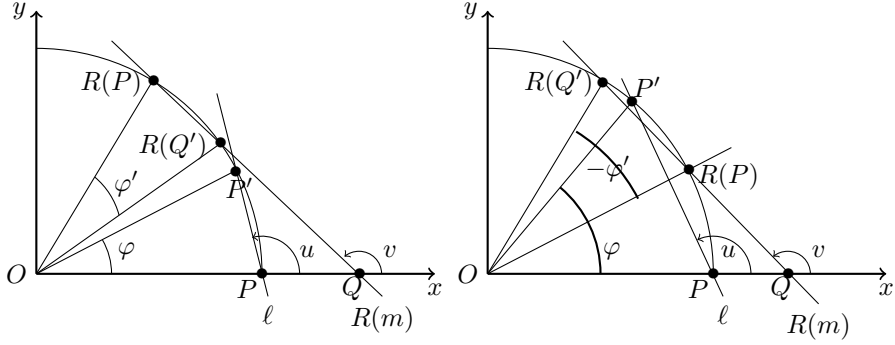
\begin{figure}[ht]
   \centering
      \begin{tikzpicture}
       \begin{scope}[xshift=-3cm]
	\draw (1.1,0.78) arc [start angle=15, end angle=50, radius=1];
      \node at (1.2,1.2) {$\varphi'$};
	\draw (1,0) arc [start angle=0, end angle=30, radius=0.9];
      \node at (1.2,0.3) {$\varphi$};
     \draw (3,0) arc [start angle=0, end angle=90,radius=3];
     \draw[thick, ->] (0,0) -- (5.3,0) node [below] {$x$};
     \draw[thick, ->] (0,0) -- (0,3.5) node [left] {$y$};
      \draw (0,0) -- (1.53,2.54) node [left] {$\Refl{P}$};
      \node at (1.56,2.56) {$\bullet$};
      \draw (1,3.1) -- (4.6,-.3) node [below]    {$\Refl{m}$};
      \node at (4.2,-0.2) {$Q$};
      \node at (4.3,0) {$\bullet$};
     \node at (4.7,0.3) {$v$};
      \draw[->] (4.59,0) arc [start angle=0, end angle=130,radius=0.3];
      \draw (2.4, 2.4) -- (3.08,-.3) node [below] {$\ell$};
      \node at (3,0) {$\bullet$};
      \node at (2.8,-0.2) {$P$};
      \node at (3.6,0.3) {$u$};
      \draw[->] (3.5,0) arc [start angle=0, end angle=105,radius=0.5];
     \draw (0,0) -- (2.4, 1.7) node [left]
      {$\Refl{Q'}$};
      \node at (2.45,1.74) {$\bullet$};
      \draw (0,0) -- (2.7, 1.4) node [below] {$P'$};
      \node at (2.65,1.35) {$\bullet$};
      \draw (0,0) node [left] {$O$} ;
      \end{scope}
       \begin{scope}[xshift=3cm]
	\draw[thick] (2,1) arc [start angle=27.4, end angle=58,	radius=2.3];
      \node at (1.6,1.4) {$-\varphi'$};
	\draw[thick] (1.5,0) arc [start angle=0, end angle=50, radius=1.5];
      \node at (1.65,0.3) {$\varphi$};
     \draw (3,0) arc [start angle=0, end angle=90,radius=3];
     \draw[thick, ->] (0,0) -- (5.3,0) node [below] {$x$};
	\draw[thick, ->] (0,0) -- (0,3.5) node [left] {$y$};
	\def\rqqx{1.53}
	\def\rqqy{2.54}
      \draw (0,0) -- (1.53,2.54) node [left] {$\Refl{Q'}$};
      \node at (1.53,2.54) {$\bullet$};
      \draw (1,3.1) -- (4.4,-.4) node [below]    {$\Refl{m}$}; 
      \node at (3.9,-0.2) {$Q$};
      \node at (4,0) {$\bullet$};
     \node at (4.4,0.3) {$v$};
      \draw[->] (4.29,0) arc [start angle=0, end angle=130,radius=0.3];
      \draw (1.78, 2.6) -- (3.13,-.3) node [below] {$\ell$};
      \node at (3,0) {$\bullet$};
      \node at (2.8,-0.2) {$P$};
      \node at (3.2,0.3) {$u$};
      \draw[->] (3.5,0) arc [start angle=0, end angle=115,radius=0.5];
	\def\ax{1.92}
	\def\ay{2.29}
	\draw (0,0) -- (\ax, \ay);
	\node at (\ax+.2,\ay+0.2) {$P'$};
	\node at (\ax,\ay) {$\bullet$};
	\def\ppx{2.7}
	\def\ppy{1.4}
	\draw (0,0) -- (1.2*\ppx, 1.2*\ppy);
	\node at (\ppx+.5,\ppy-.1) {$\Refl{P}$};
	\node at (2.68,1.38) {$\bullet$};
      \draw (0,0) node [left] {$O$} ;
      \end{scope}
      \end{tikzpicture}   \caption{\label{FigA} Proofs of \eqref{ddag} and \eqref{A}. Case
   $\varphi'>0$ (left) and case $\varphi'<0$ (right).}
  \end{figure}
The line $\Refl{m}$ is not the $x$-axis. It is because $\Refl{P}\in C\cap \Refl{m}\setminus
  (\Rset\times\{0\})$ by
  $\angle x O \Refl{P}= 2\pi/n$. Hence,
  $\#(\Refl{m}\cap(\Rset\times\{0\}))\le 1$.
Let a point $Q$ be the intersection of the line $\Refl{m}$ and $x$-axis, if it exists.
  Define
  \begin{align*}
v:=\begin{cases} \pi& (\Refl{m}\cap
			     (\Rset\times\{0\})=\emptyset);\\
			     \angle x Q \Refl{P}&(\mbox{otherwise}).
			    \end{cases}   
  \end{align*}
See Figure~\ref{FigA}.

\begin{claim}\label{claim:uv}If
 equation~\eqref{eq:w} holds and
 $\pi/n\le\varphi<\pi$, then
 \begin{align}
&\tan u = (\tan \gamma  \cos a )^{-1}.\qquad
					      \tan v = \frac{ \sin {\frac {2\pi}{n}} \sin  \alpha \cos  a -\cos\frac {2\pi}{n} \cos  \alpha }{ \cos \frac {2\pi }{n} \sin \alpha \cos a  
 +\sin {\frac {2\pi }{n}}  \cos \alpha}.   
  \label{grad:u}	\\
&\eqref{eq:u}\ \&\ \eqref{eq:v} \iff f_{n,\alpha,\gamma}(\cos a) =0.\label{assert:v-u}	\end{align}
\end{claim}

 \begin{proof}
  \eqref{grad:u} The first equation is by the definition of $\ell$ and Claim~\ref{claim:P'Q'}.
  The denominator of the left-hand side of the second equation is not
  zero, by Claim~\ref{claim:P'Q'}~\eqref{claim:P'Q'2}.

Next, we  prove
\begin{align*}
	 v  =	\begin{cases}
 \frac{2\pi}{n} + \pi - \arctan\left( (\tan \alpha \cos a)^{-1}\right)
		 &\left(\frac{\pi}{n} \le \varphi <
		 \frac{n-2}{n}\pi\right);\\
		 \pi &\left(\varphi=\frac{n-2}{n}\pi\right);\\
 \frac{2\pi}{n}  - \arctan\left( (\tan \alpha \cos a)^{-1}\right)  &\left(\frac{n-2}{n}\pi < \varphi <	 \pi\right).		 %
	\end{cases}
\end{align*}The proof is as follows:
  Suppose $\pi/n\le \varphi<(n-2)\pi/n$.
Let a point $\overline{Q'}$ be the
  reflection of the point $Q'$ with respect to the $x$-axis. $\angle x P
  \overline{Q'}= \pi - \arctan \left((\tan\alpha\cos
  a)^{-1}\right)$. Thus $v=\angle x P \overline{Q'}+2\pi/n$.
Suppose $\varphi=(n-2)\pi/{n}$. By
  equation~\eqref{eq:w}, $\varphi'=(4-n)\pi/n$. By the definition,
  $\angle x O \Refl{Q'}=\angle x O \Refl{P}+\varphi'=\varphi+\varphi'$, which is
  $2\pi/n$ by \eqref{eq:w}. Hence, $(\angle x O \Refl{P} + \angle x O
  \Refl{Q'})/2 = \pi/2$. As $\Refl{P}, \Refl{Q'}\in C$, $\Refl{m}$ does not intersect
  with the $x$-axis. Hence $v=\pi$ by the definition of $v$.
  The proof for  $(n-2)\pi/n<\varphi\le \pi$ is similar to the proof for
  $\pi/n\le \varphi<(n-2)\pi/n$. This establishes the desired
representation of $v$.

If $\varphi\ne(n-2)\pi/n$, then by the addition formula of $\tan$, $\tan v$
  is as desired.
Consider the case $\varphi=(n-2)\pi/n$. Then $\varphi'=(4-n)\pi/n$. By
  \eqref{eq:v}, $(\tan\alpha \cos a)^{-1} = \tan (2\pi/n)$. By the addition formula of $\tan$, $\tan v$
  is as desired.
  This completes the proof of equation~\eqref{grad:u} of Claim~\ref{claim:uv}.

  \medskip
  \eqref{assert:v-u} First we claim
 \begin{align}
\label{hash}\quad \ell\cap\Refl m\ni(\cos\varphi,\,\sin\varphi)\iff \Refl{Q'}=P'.  
 \end{align}
The proof is as follows: $(\cos\varphi,\,\sin\varphi)=P'$ is $\Refl \P$ or
 $\Refl{Q'}$, because $m\cap C=\{P,Q'\}$ by Claim~\ref{claim:P'Q'}~\eqref{claim:P'Q'2}. Here $\Refl \P$ is $(\cos(2\pi/n),\, \sin(2\pi/n))$. If
 $\Refl \P=\Q$, then $2\pi/n=\varphi$, and thus $\varphi'=2\pi/n -
  \varphi=0$, by equation~\eqref{eq:w}. This is a contradiction. This completes
  the proof of \eqref{hash}.

  Next we claim
  \begin{align}
\label{dhash}\quad  \tan(v - u)=\tan\frac{\pi}{n} \iff f_{n,\alpha,\gamma}(\cos a)=0.
  \end{align}
  The left-hand side of equivalence~\eqref{dhash} is
\begin{align*}
 \frac{\tan u -\tan v}{1+\tan u\tan v } + \tan \frac{\pi}{n}=0. 
\end{align*}
Observe that the denominator of the first term of the left-hand side
  cannot be $0$. Assume  otherwise. Then  $u-v= \pi/2 +i\pi$ for some
  integer $i$. Thus $\angle P
  P' (\Refl{P})=\pi/2 + i \pi$  for some
  integer $i$. Thus $\varphi+\varphi'=\pi$, which contradicts
  against equation~\eqref{eq:w}. Hence the denominator $1+\tan u \tan v$
  of the first
  term of the left-hand side is not $0$. Also note that
  the denominator $\cos\pi/n$ of the second term of the left-hand side
  is not $0$. 
  Substitute \eqref{grad:u} in 
  the left-hand side. Then we have a quadratic equation
  of $\cos a$, by calculation.
  Because $\sin\alpha \sin\gamma\sin(\pi/n)\ne
  0$, the quadratic equation is equivalent to the quadratic equation
  $f_{n,\alpha,\gamma}(x)=0$ of $x=\cos a$. This completes the proof of \eqref{dhash}.

Hence,  by equivalences~\eqref{ctldag}, \eqref{hash}, and \eqref{dhash},
  we have only to prove
  \begin{align}
   \Refl{Q'}=P' \iff \tan(v-u)=\tan\frac{\pi}{n},\label{equiv:tan}  \end{align}
to show \eqref{assert:v-u}.
If $\Refl{Q'}=\Q$, then $\angle P O \Refl{Q'}= \angle P O P'=\varphi$. Thus $v-u=(\varphi+\varphi')/2 + k\pi$
 for some integer $k$. Equation~\eqref{eq:w} implies $\tan(v-u)=\tan(\pi/n)$.
  To prove the converse of \eqref{equiv:tan}, we  derive
  \begin{align*}
\tan(v-u)=\tan(\pi/n)\implies   \angle P O \Refl{Q'} = \varphi.
  \end{align*}

\medskip
  Case~A. $\pi/n \le \varphi <(n - 2)\pi/n$ (See Figure~\ref{FigA}).
  
The mean $M$ of $\angle x O \Refl{P}=2\pi/n$ and $\angle x O \Refl{Q'} = 2\pi/n -
  \varphi'$ is $2\pi/n  - \varphi'/2 = \pi/n + \varphi/2$ by equation~\eqref{eq:w}. 
  Then $M<\pi/2$, by $\varphi < (n - 2)\pi/n$. 
  Therefore, $\Refl{m}\cap ( (0,\,\infty)\times \{0\})$ consists of a
  unique point $Q$, where $\Refl{m}$ is a line
  through the two points $\Refl{P}$ and
  $\Refl{Q'}$.  We claim
   \begin{align}
\label{A} \frac{\pi}{n}\le \varphi<\frac{n-2}{n}\pi \implies
    v-u=\frac{\pi}{n} - \varphi + \angle PO\Refl{Q'} \in
    \left(\frac{4-n}{2n}\pi,\, \frac{\pi}{2}\right).
   \end{align}
  The proof is as follows: Observe $v=(\varphi'+\pi)/2 + \angle PO\Refl{Q'}$. It is clear
 when $\varphi'>0$. In case $\varphi'<0$, the observation follows from $v=(-\varphi'+\pi)/2 + (\angle
  PO\Refl{Q'} + \varphi')$.
  By equation~\eqref{grad:u}, $v>\pi/2 + 2\pi/n$. Clearly,
 $v<\pi$.  Hence, by $u\in (\pi/2,\,
  \pi) $, $v-u$ is in the desired
 interval.  From equations~\eqref{ddag} and \eqref{eq:w}, the desired equation of \eqref{A} follows. This completes the proof of \eqref{A}.

 Assume $\tan(v-u)=\tan(\pi/n)$. By \eqref{A}, $\varphi=\angle P O \Refl{Q'}$.

 \medskip
 Case~B. $(n - 2)\pi/n < \varphi <\pi$ (See Figure~\ref{FigB}). 
 
    \begin{figure}[ht]
   \centering
      \begin{tikzpicture}
       \begin{scope}[xshift=-3cm,scale=0.95]
	\node at (-1.1,1.4) {$-\varphi'$};
	\draw (1.3,.72) arc [start angle=27, end angle=177,radius=1.4];
     \draw (3,0) arc [start angle=0, end angle=180,radius=3];
     \draw[thick, ->] (-6,0) -- (4,0) node [below] {$x$};
     \draw[thick, ->] (0,0) -- (0,3.5) node [left] {$y$};
      \draw (2.7,1.4) -- (-6,-.2) node [below]    {$\Refl{m}$};
      \node at (-5,-0.3) {$Q$};
	\node at (-5,0) {$\bullet$};
	\draw[->] (-1.8,0) arc [start angle=0, end angle=16,radius=1.9];
	\node at (-1.65,0.35) {$v$};
	\draw (0,0) -- (-2.95,0.33);
	\node at (-3.5,0.6) {$\Refl{Q'}$};
	\node at (-2.95,0.33) {$\bullet$};
      \draw (2.4, 2.4) -- (3.08,-.3) node [below] {$\ell$};
      \node at (3,0) {$\bullet$};
      \node at (2.8,-0.2) {$P$};
      \node at (3.6,0.3) {$u$};
      \draw[->] (3.5,0) arc [start angle=0, end angle=105,radius=0.5];
      \draw (0,0) -- (2.7, 1.4) node [right] {$\Refl{P}$};
      \node at (2.65,1.35) {$\bullet$};
      \draw (0,0) node [below] {$O$} ;
      \end{scope}
      \end{tikzpicture}   \caption{ Proof of \eqref{B}. $(n-2)\pi/n<\varphi<\pi$. \label{FigB}}
    \end{figure}
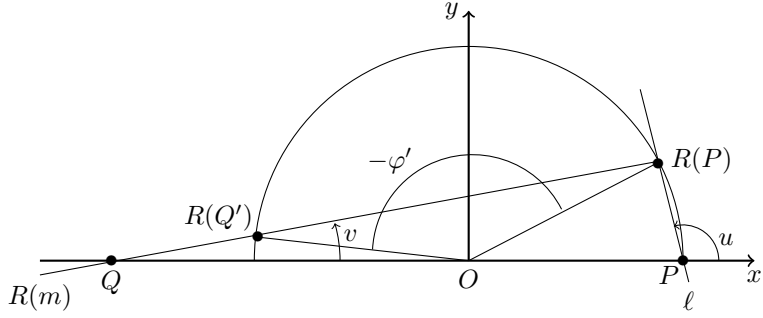
    
    We claim:
  \begin{align}
\label{B}  \frac{n-2}{n}\pi <\varphi < \pi \implies
   v-u = \frac{\pi}{n} - \varphi + \angle P O \Refl{Q'} - \pi \in
   \left(-\pi,\, \frac{2-n}{2n}\pi\right).
  \end{align}
 The proof is as follows: By Figure~\ref{FigB}, $v=(-\varphi' + \pi )/2 + 2\pi/n
 -\pi >0$. 
 
 So, the desired equation follows from 
 equations~\eqref{ddag} and \eqref{eq:w}, in a similar argument as Case~A. Equation~\eqref{eq:w} implies $v=\pi/n +\varphi/2 -
  \pi/2<\pi/n$. By condition of Case~B and the definition~\eqref{ddag}
  of the angle $u$, we have
  $(n-1)\pi/n<u<\pi$. So $v-u$ is indeed in the desired interval of \eqref{B}. This completes the
 proof of \eqref{B}.

Assume $\tan(v-u)=\tan(\pi/n)$. The interval $(-\pi,\, (2-n)\pi/(2n))$
 contains $\pi/n + k \pi$ for a unique integer $k=-1$. By \eqref{B},
 $\varphi = \angle PO \Refl{Q'}$.

 \medskip
 Case~C. $\varphi=(n - 2)\pi/n$. Then the line $\Refl{m}$ does not
 intersect with the $x$-axis. As $\Refl{m}\cap C=\{\Refl{P}, \Refl{Q'}\}$,
 $\pi/2=\left(\angle P O \Refl{Q'} + \angle P O \Refl{P}\right)/2$,
 $\angle \Refl{Q'} O \Refl{P} = -\varphi'$, $\angle P O
  \Refl{P}=\varphi$, and equation~\eqref{eq:w}, it holds that $\varphi=\angle P O \Refl{Q'}$.
 Thus we have proved the converse of \eqref{equiv:tan}.
 This completes the proof of Claim~\ref{claim:uv}.  \qed\end{proof}

  By the symmetry
  $(\varphi,\gamma)\leftrightarrow(\varphi',\alpha)$, 
\eqref{assert:v-u} of Claim~\ref{claim:uv} implies: \emph{If equation~\eqref{eq:w} holds and $\varphi<\pi/n$, then
  $\eqref{eq:u}\ \&\ \eqref{eq:v} \iff f_{n,\gamma,\alpha}(\cos a) =0$.}
Here $f_{n,\gamma,\alpha}(\cos a) =f_{n,\alpha,\gamma}(\cos a)$. So, $T$ is a $\PQ_{n,\alpha,\gamma,a}$ if and only if $f_{n,\alpha,\gamma}(\cos a) =0$.
    This establishes Theorem~\ref{thm:existence Pn}.\qed \end{proof}

\section{Range of inner angles of $\TRPZ_n$-quadrangles}
\label{sec:main}

To classify the two opposite inner angles
$\alpha,\gamma$ and the edge-length $a$ of $\TRPZ_n$-quadrangles $\PQ_{n,\alpha,\gamma,a}$'s, we solve
$f_{n,\alpha,\gamma}(\cos a)=0$, taking the condition
Lemma~\ref{lem:shorta} (Proposition~\ref{prop:vinberg}~\eqref{assert:vinberg}, resp.) of
quadrangles~(spherical 3-gons, resp.) into account.
This classifies all tilings of $\TRPZ_n$, because of 
Fact~\ref{fact:Pquad}~\eqref{assert:bijective}.

\subsection{Discriminant\label{subsec:discriminant}}

The equation
$f_{n,\alpha,\gamma}(\cos a)=0$ has at most two solutions $a\in (0,\,\pi)$, as $\cos a$ is
strictly decreasing for $a\in (0,\,\pi)$ and $f_{n,\alpha,\gamma}(x)$
is quadratic.
The smaller solution $a=a_{n,\alpha,\gamma}^-$ is the arccosine of
\begin{align*}
\frac{1}{2}\cot\frac{\pi}{n}
 \left(\cot\alpha+\cot\gamma
+
 \sqrt{\Delta_{n,\alpha,\gamma}}\right),
\end{align*}
while the larger solution $a=a_{n,\alpha,\gamma}^+$ of
$f_{n,\alpha,\gamma}(\cos a)=0$ is obtained from $a_{n,\alpha,\gamma}^-$ by inverting
the sign in front of the square root. Here
\begin{align*}
\Delta_{n,\alpha,\gamma}:=\cot^2\gamma+ 2\left(2\tan^2\frac{\pi}{n}+1\right)\cot\alpha\cot\gamma
	+ \cot^2\alpha .
       \end{align*}

\begin{lemma}\label{lem:degenerate}Let 
 $\pi/2<\alpha<\pi<\gamma<3\pi/2$. Then
\begin{enumerate}
 \item \label{assert:dgn}
       $\Delta_{n,\alpha,\gamma}\ge0\iff\gamma\le\degenerate_n(\alpha)$.
Moreover, the equality of one side implies that of the other side.
Here     $\degenerate_n:(\pi/2,\ \pi)\to (\pi,\ 3\pi/2)$ is defined as
\begin{align*}
     \degenerate_n(\psi)&:=\pi-\arctan\left( \cos^2\frac{\pi}{n}
 \left(\sin\frac{\pi}{n}+1\right)^{-2} \tan \psi \right). 
\end{align*}

\item \label{usual} The curve
 $\gamma=\degenerate_n(\alpha)$ is strictly decreasing, convex, and
 has the tangential line $\gamma=2\pi - \fra{\pi}{n} - \alpha$ at
      $\alpha=3\pi/4-\fra{\pi}{(2n)}$.
      
\item \label{unusual} 
$2\pi - \fra{\pi}{n} - \alpha < \degenerate_n(\alpha)<2\pi - \alpha$ for all $\alpha\in
 \left( \fra{\pi}{2},\, \fra{3\pi}{4}-\fra{\pi}{(2n)}
      \right)$.
\end{enumerate}
\end{lemma}

\begin{figure}[htp]
\centering
\begin{tikzpicture}[scale=0.3,axis/.style={very thick, ->, >=stealth'}]

\pgftext{\includegraphics{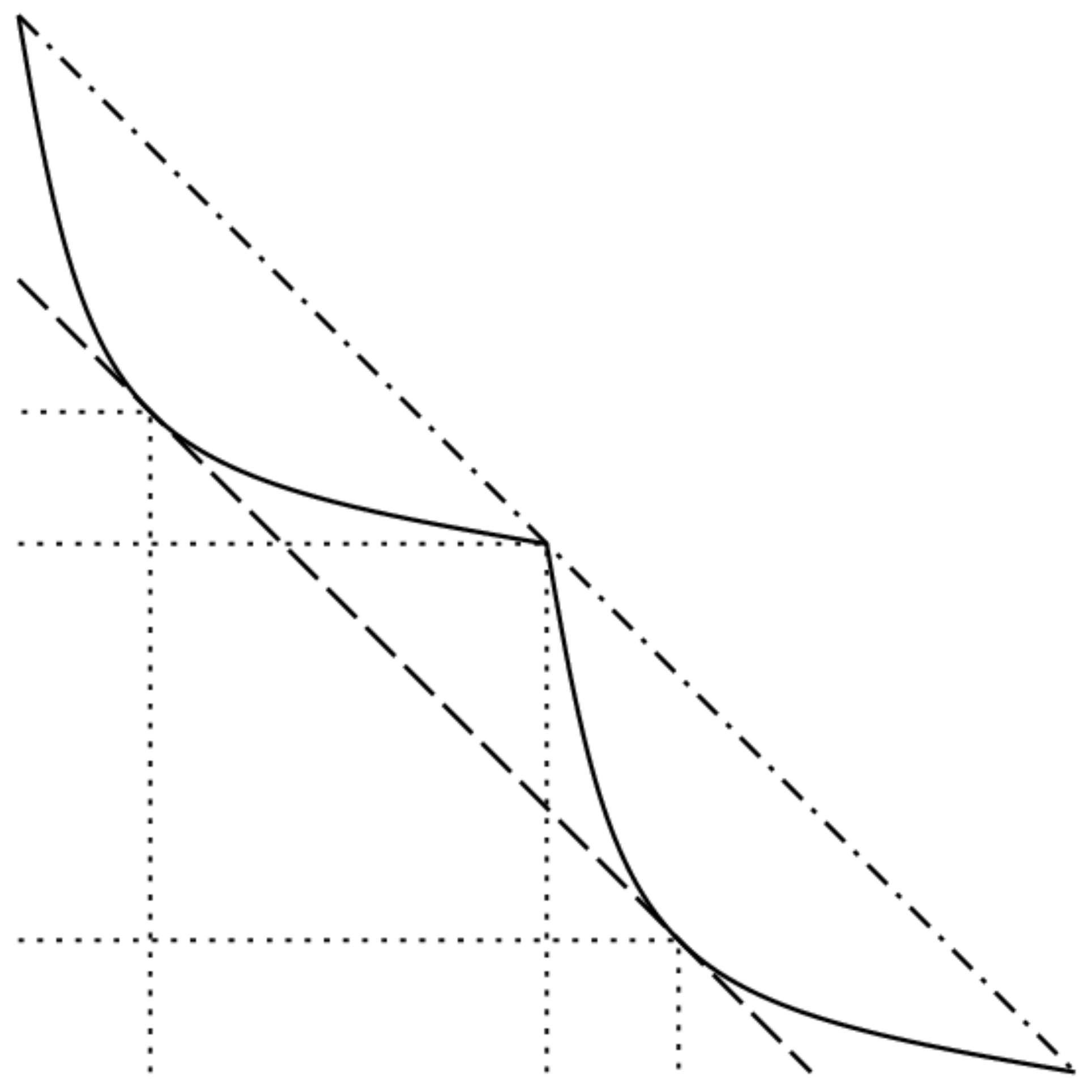}}

 \node at (-7,-11) {\small $\frac34 -\frac{1}{2n}$};
 \node at (-12.5,-7) {\small $\frac34 -\frac{1}{2n}$};

 \node at (-12.5,5) {\small $\frac32 -\frac{1}{n}$};
 \node at (5,-11) {\small $\frac32 -\frac{1}{n}$};
 \node at (-12.5,2.5) {\small $\frac54 - \frac{1}{2n}$};
 \node at (2.5,-12.5) {\small $\frac54 - \frac{1}{2n}$};
 \node at (-11,10) {\small $\frac32$};
 \node at (10,-11) {\small $\frac32$};
\node at (0,-11) {\small $1$};
 \node at (-11,0) {\small $1$};
 
\node at (-11,-11) {\small $(\frac12,\frac12)$};

 \draw[axis] (-9.5, -9.5) -- (12, -9.5) node(xline)[right] {$\alpha$
 [$\pi$ rad]};
 \draw[axis] (-9.5, -9.5) -- (-9.5, 12) node(yline)[above] {$\gamma$  [$\pi$ rad]};
\end{tikzpicture}
\caption{The curves $\gamma=\degenerate_n(\alpha)$~($\pi/2<\alpha<\pi$),
 $\alpha=\degenerate_n(\gamma)$~($\pi/2<\gamma<\pi$),
 $\alpha+\gamma=2\pi-\pi/n$~(dash), and $\alpha+\gamma=2\pi$~(dash-dot),
 when $n=4$. See Lemma~\ref{lem:degenerate}.
 \label{thecurves}}
\end{figure}
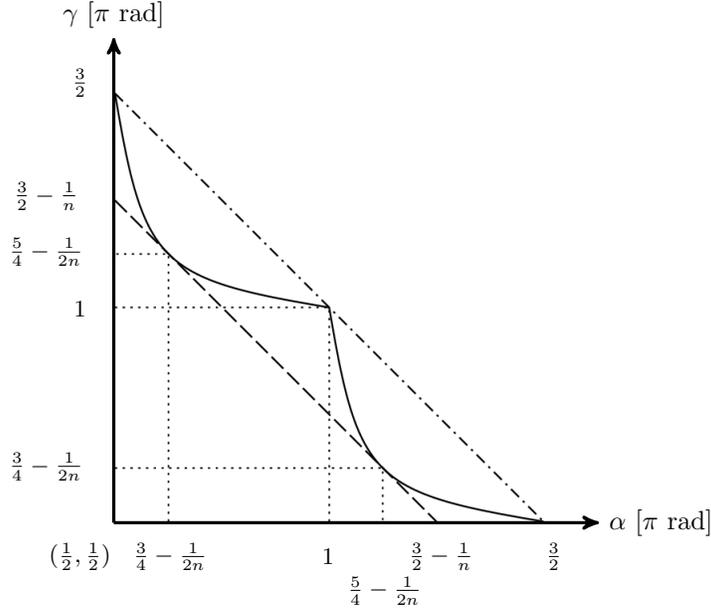

 \begin{proof} \eqref{assert:dgn}. Let $s_n:=\sin(\fra{\pi}{n}) -1<0$,
  $t_n:=\sin(\fra{\pi}{n}) +1>0$, and
  \begin{align*}
z_1:=-s_n^2 \cot\alpha\,  \sec^2 \frac{\pi}{n},\qquad z_2:=-t_n^2
   \cot\alpha\, \sec^2 \frac{\pi}{n}.
   \end{align*}We
prove
$z_1 <\cot\gamma\ \implies \mbox{Lemma~\ref{lem:degenerate}~\eqref{assert:dgn}}$, as follows: By calculation,
$z_1$ and $z_2$ are the zeros of the
  quadratic polynomial 
\begin{align*}
p(z):=z^2+ 2\cot \alpha\left(2 \tan ^2\frac{\pi}{n} +1 \right) z +\cot ^2
 \alpha. 
\end{align*}
Here $\Delta_{n,\alpha,\gamma}=p(\cot\gamma)$.
$z_1<z_2$ by $\pi/2<\alpha<\pi$. Clearly
  $\gamma\le\degenerate_n(\alpha)$ if and only if $\gamma-\pi\le -\arctan \left( {\tan
  \alpha \cos^2 { (\fra {\pi }{n})} t_n ^{-2}} \right) $. 
  $\gamma-\pi\in(0,\,\pi/2)$ by the premise. So, by applying the strictly
  decreasing function $\cot$, $\gamma\le\degenerate_n(\alpha)$ is
  equivalent to $z_2\le \cot\gamma$. Hence
   \begin{claim}Let
    $\pi/2<\alpha<\pi<\gamma<3\pi/2$. Then
    \begin{align*}
   \gamma\le \degenerate_n(\alpha)\iff
  -\cot\alpha \left(1+\sin\frac{\pi}{n}\right)^2 \sec^2 \frac{\pi}{n}\le\cot \gamma.
     \end{align*}
  The equality of one side implies that of the
   other side.   \end{claim}
 Therefore,  Lemma~\ref{lem:degenerate}~\eqref{assert:dgn} follows from
  $z_1<\cot\gamma$, because the polynomial $p(z)$ is quadratic.

We first prove $z_1< -\cot \alpha $ as follows: The premise
$\alpha\in(\pi/2,\, \pi)$ implies $\tan \alpha <0$. So,  $z_1
 < - \cot \alpha $ if and only if $s_n^2\sec^2(\fra{\pi}{n})<1$.
As $s_n<0$, the inequality $\sec^2(\fra{\pi}{n})s_n^2<1$
 is equivalent to $ - \cos (\fra{\pi}{n})< s_n$, which is
 equivalent to $1/\sqrt2< \sin(\fra{\pi}{n} +\pi/4)$. The last
 inequality holds for $n=3$ by calculation. It also holds for $n\ge4$, by
 $\fra{\pi}{n}+\pi/4\in (\pi/4,\, \pi/2]$. Thus $z_1< -\cot \alpha $.

Assume $z_1\ge \cot\gamma$. Then $\cot \gamma <
-\cot \alpha $ by $z_1< -\cot \alpha $.  By 
$\alpha\in(\pi/2,\,\pi)$ and $\gamma\in(\pi,\,3\pi/2)$, we have  $\cot \gamma >0$,
$\tan \alpha \tan \gamma  <0$, and thus
$-\tan \alpha <\tan \gamma $. Hence $\tan(\alpha+\gamma)=(
\tan \alpha  +\tan \gamma )/(1-\tan \alpha \tan \gamma  )>0$, which implies
$\alpha+\gamma>2\pi$ by the premise $\alpha\in(\pi/2,\,\pi),\gamma\in(\pi,\,
 3\pi/2)$. This contradicts against $\alpha+\gamma+\delta=2\pi$. Hence $z_1<\cot\gamma$.
  
  \medskip To prove
Lemma~\ref{lem:degenerate}~\eqref{usual}, we first verify the curve $\gamma=\degenerate_n(\alpha)$
  and the line $\gamma=2\pi - (\pi/n) - \alpha$ intersect at
  $\alpha=3\pi/4 - \pi / (2n)$, as follows:
  \begin{align} 
\degenerate_n\left(\frac{3\pi}{4}-\frac{\pi}{2n} \right)=2\pi
   -\frac{\pi}{n} - \left(\frac{3\pi}{4}-\frac{\pi}{2n}\right),
 \label{tangential line}
 \end{align}
  if and only if $\arctan \left( \cos^{2} \left( {\fra {\pi }{n}}
  \right) \tan \left( \pi /4+\pi /(2n) \right) \left( \sin \left( {\fra
  {\pi }{n}} \right) +1 \right) ^{-2} \right)$ is $\pi /4-\pi /(2n)
  $. As the right-hand side $\pi /4-\pi /(2n) $ is strictly between
  $(0,\,\pi/2)$, the condition is equivalent to ${ \cos^2 \left( {\fra
  {\pi }{n}} \right) \tan^2 \left( \pi /4+{\fra {\pi }{(2n)}} \right)
  \left( \sin \left( {\fra {\pi }{n}} \right) +1 \right) ^{-2}}=1
  $. Hence the square root of the left-hand side is unity, as $n\ge3$
  implies $0<(1/4 + 1/(2n))\pi < \pi/2$. By calculation, it is indeed
  unity from the double-angle formulas. Thus \eqref{tangential line} is
  proved.

By calculation, we have a partial derivative 
  \begin{align*}
\partial_\alpha\degenerate_n(\alpha)&=
-t_n^2\cos^2\frac{\pi}{n}\left(\left(t_n^4- \cos^4 \frac {\pi
}{n}\right) \cos^2 \alpha + \cos^4 \frac{\pi
}{n}\right)^{-1}.\end{align*} It is negative because $t_n>1$. So
  $\gamma=\degenerate_n(\alpha)$ is decreasing.  Since
  \begin{align}
  \partial_\alpha
   \degenerate_n\left(\frac{3\pi}{4}-\frac{\pi}{2n} \right)=-1
\label{derivative of dgn}      
  \end{align}
by calculation, a line
$\alpha+\gamma=2\pi-\fra{\pi}{n}$ is the tangential line of the curve
$\gamma=\degenerate_n(\alpha)$. By calculation, the second-order
  derivative $\partial^2_\alpha
\degenerate_n(\alpha)$ is
\begin{align*}
-t_n^2
\cos^2\frac{\pi}{n}
  \left(t_n^4 - \cos^4 {\frac {\pi }{n}}   \right)  
 \sin 2\alpha
 \left(   \left( t_n^4 - \cos^4\frac{\pi}{n} \right) \cos^2\alpha + \cos^4\frac{\pi}{n} \right)^{-2}.
  \end{align*}
It is positive since $\pi/2< \alpha< \pi$ by the premise and $t_n>1$.

 \medskip To prove
Lemma~\ref{lem:degenerate}~ \eqref{unusual}, observe that the first inequality $2\pi -\fra{\pi}{n} - \alpha < \degenerate_n(\alpha)$ follows from
Lemma~\ref{lem:degenerate}~\eqref{usual}. As for
the second inequality $\degenerate_n(\alpha)<2\pi - \alpha$, note that
$\degenerate_n(\alpha)\to 3\pi/2-0$, as $\alpha\to\pi/2+0$. By 
  equality~\eqref{derivative of dgn} and the convexity of the
 curve $\gamma=\degenerate_n(\alpha)$, we have $\partial_\alpha
 \degenerate_n(\alpha)\le -1$ for all $\alpha\in \left(\pi/2,\, 3\pi/4 - \pi/(2n)\right)$. Thus
 $\degenerate_n(\alpha)<2\pi-\alpha$. 
This establishes Lemma~\ref{lem:degenerate}.\qed
\end{proof}

\subsection{The statement}\label{subsec:statement}

\begin{definition}For $n\ge3$, define
\begin{align*}
B_n:=\left\{\,(\alpha,\gamma)\in\left((0,\,\pi)\cup(\pi,\,2\pi)\right)^2\,\mid\,
      \mbox{A $\PQ_{n,\alpha,\gamma,a}$ exists for some $a\in(0,\,\pi)$}\,\right\}.
\end{align*}
\end{definition}
By simple trigonometric formulas, we  describe $B_n$
and the length $a$ of the meridian edge  of each $\TRPZ_n$-quadrangles
$\PQ_{n,\alpha,\gamma,a}$.

\begin{theorem}[Inner angles and edge-length of
 $\TRPZ_n$-quadrangles]\label{thm:main}Assume $n\ge3$.
\begin{enumerate}
\item \label{assert:clas}
$B_n=\bigcup_{i=1}^8\,B_n^{(i)}$ where $B_n^{(i)}$ is defined in \eqref{def:bns} below.

\item \label{assert:amb}
 $(\alpha,\gamma)\in B_n^{(4)}\cup B_n^{(8)}$,  if and only if there
 exist exactly two $\TRPZ_n$-quadrangles. Here the edge-length $a$ is      
      $a_{n,\alpha,\gamma}^+$ or $a_{n,\alpha,\gamma}^-$.

\item \label{assert:unamb}
 $(\alpha,\gamma)\in \bigcup_{1\le i\le 8,i\ne 4,8}B_n^{(i)}$, if and
      only if there exists a unique $\TRPZ_n$-quadrangle
      $\PQ_{n,\alpha,\gamma,a}$.  Here the edge-length $a$ is
   $a_{n,\alpha,\gamma}^-$ for $(\alpha,\gamma)\in B_n^{(1)}$; $a_{n,\alpha,\gamma}^+$ for
   $(\alpha,\gamma)\in B_n^{(2)}\cup
   B_n^{(5)}\cup B_n^{(6)}$;
       \begin{align}
      \pi -\arccos \left( \left(\sec\frac {\pi }{n}\right) \left( \sin \frac {\pi }{n} +1 \right)
 \cot \alpha \right)
	=a_{n,\alpha,\gamma}^\pm<\frac{\pi}{2} \label{degenerate sol3}
       \end{align}
      for $(\alpha,\gamma)\in B_n^{(3)}$;
      \begin{align}
      \pi -\arccos \left( \left(\sec\frac {\pi }{n}\right)  \left( \sin \frac {\pi }{n} +1 \right)
 \cot \gamma \right)
	=a_{n,\alpha,\gamma}^\pm<\frac{\pi}{2} \label{degenerate sol7}
      \end{align}
      for $(\alpha,\gamma)\in B_n^{(7)}$.
      
\item \label{def:bns} Here
\begin{itemize}
\item $B_n^{(1)}$ is
an open 
pentagon $\{(\alpha,\gamma) \mid \pi/2<\alpha < \pi,\ \pi/2<\gamma<\pi,\
      \alpha+\gamma<2\pi-\fra{\pi}{n}\}$;

\item $B_n^{(2)}$ is an open rectangular triangle 
$\{(\alpha,\gamma) \mid \fra{\pi}{2}<\alpha,\ \pi<\gamma,\ \alpha+\gamma<2\pi-\fra{\pi}{n}\}$; 

\item $B_n^{(3)}$ is a curve 
$\{\,(\alpha,\degenerate_n(\alpha)) \mid
      \pi/2<\alpha<3\pi/{4}-\fra{\pi}{(2n)}\}$; 

\item $B_n^{(4)}$ is a nonempty open set $\{\,(\alpha,\gamma) \mid
\pi/2<\alpha,\ 2\pi - \fra{\pi}{n} - \alpha < \gamma< \degenerate_n(\alpha)\}$;   

\item $B_n^{(5)}$ is  symmetric to $B_n^{(1)}$
around the line $\alpha+\gamma=\pi$; and

\item $B_n^{(i+4)}$ $(2\le i\le 4)$ is  symmetric to 
      $B_n^{(i)}$ around the line $\gamma=\alpha$.
\end{itemize}
\end{enumerate} 
\end{theorem}

\begin{figure}[ht]
 \centering
\begin{tikzpicture}[scale=.3,axis/.style={very thick, ->, >=stealth'}]

\node at (-2.55,11.15) {\pgftext{\includegraphics[scale=0.57]{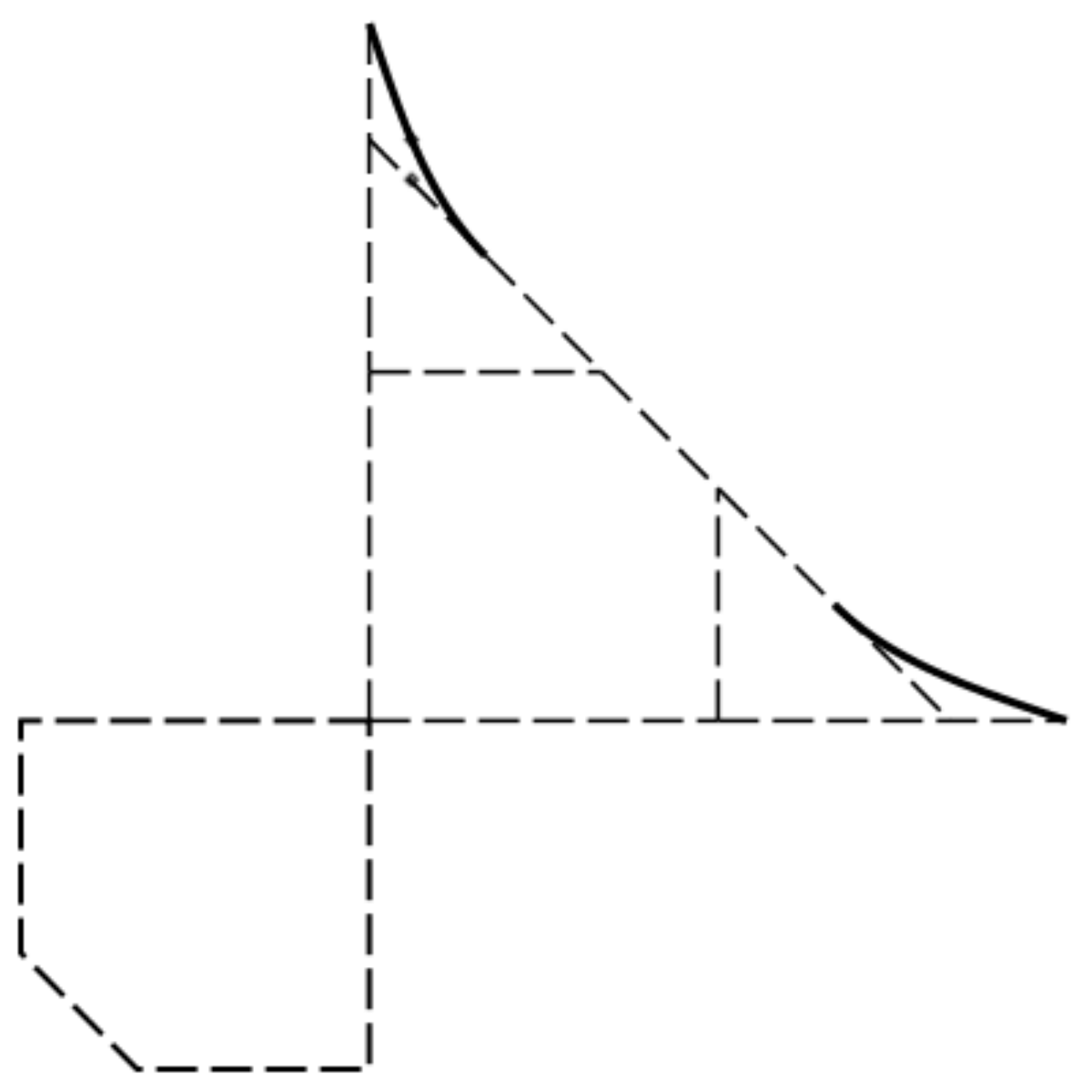}}};
 \node at (-16.5,24) {$\frac32$};
  \draw[dotted] (-15.6,24) -- (-7,24);
\node at  (-10.5, 22.6) {\Large $B_n^{(4)}$}; 
\draw[->] (-9.5, 22.1) -- (-6.4,21.4); 
\node at (-17.5,21) {$\frac32-\frac{1}{n}$};
\draw[dotted] (-15.6,21.1) -- (-7,21.1);%
 %
 %
 %
 \node at (-17.7,18) {$\frac54 - \frac{1}{2n}$};
\draw[dotted] (-15.6,18.3) -- (-4.5,18.3);
 \draw[dotted] (-15.6,15.4) -- (-5,15.4);
 \node at (-16,15.4) {$1$}; 
 \node at (4.5,25)
 {\pgftext{\includegraphics[scale=0.09]{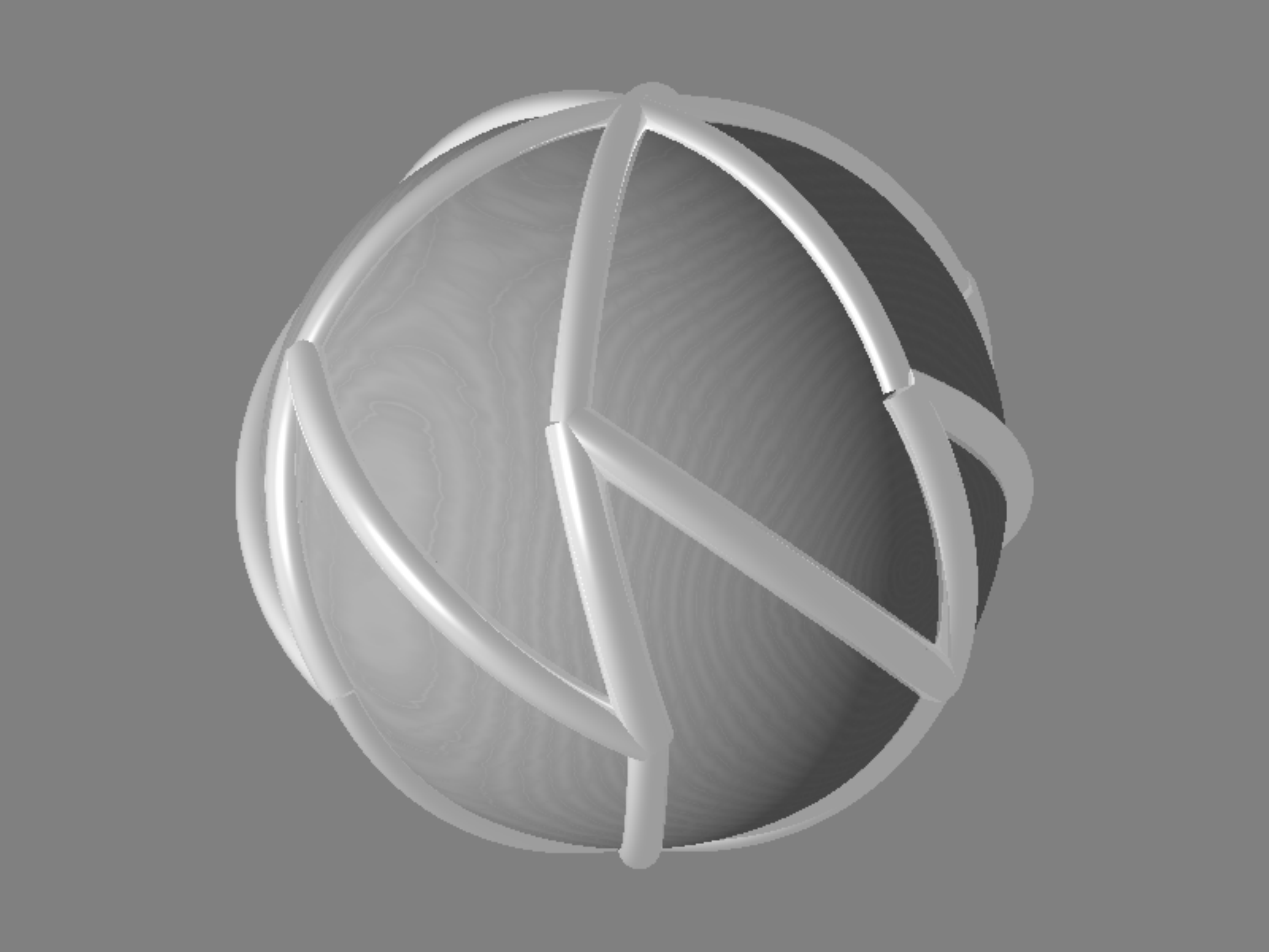}}};
 \draw[<-] (-3.6,17.1) -- (.4,21.8);
 \node at (-4,16.7) {$+$};
 \node at (-17.5,12.5) {$1-\frac{1}{n}$};
\draw[dotted] (-15.6,12.5) -- (1.6,12.5);
 \node at (-17.7,9.6) {$\frac34 -\frac{1}{2n}$};
 \draw[dotted] (-15.6,9.6) -- (4.5,9.6);
 \draw[dotted] (-1.2,15) -- (-1.2,-1.9);
 \node at (-1.2,-2.9) {$1-\frac{1}{n}$};
\node at (-3.8,10) {\Large $B_n^{(1)}$};
\node at (-5.4,17) {\Large $B_n^{(2)}$}; 
 \node at (-9, 27.3) {\Large $B_n^{(3)}$ : };
 \node at (-3.5,27) {$\gamma=\degenerate_n(\alpha)$};
 \node at (-5.5,  25.5) {$\left(\frac{\pi}{2}<\alpha<\frac{3\pi}{4}-\frac{\pi}{2n}\right)$}; 
 \draw[<-] (-6,22) -- (-5,24.3);

\node at (-16.3,7) {$\frac12$};		     
\node at (-12.5,-3) {$\frac{1}{n}$}; 
\node at (-16.3,1.2) {$\frac{1}{n}$};		     
\node at (-7.3,-.9) {$\frac12$}; 
 \node at (11,   10) {\Large $B_n^{(7)}:$};
 \node at (16.5, 9.7) {$\alpha=\degenerate_n(\gamma)$};
 \node at (15,   8) {$\left(\frac{\pi}{2}<\gamma<\frac{3\pi}{4}-\frac{\pi}{2n}\right)$};
 \draw[<-](7.5,8) -- (10,8);
\draw[->] (8,6.1) -- (7.6,7.3);
\node at (9,5) {\Large $B_n^{(8)}$};	      
 \draw[dotted] (-4,18) -- (-4,-1.9);
 \node at (-4,-0.9) {$\frac34-\frac{1}{2n}$};
 \draw[dotted] (1.6,7) -- (1.6,-1.9);
 \node at (1.6,-2.9) {$1$}; 
 \node at (3.4,8.5) {\Large $B_n^{(6)}$};
 \draw[dotted]  (4.5,9.5) -- (4.5,-1.9);
 \node at (16,15)
 {\pgftext{\includegraphics[scale=.09]{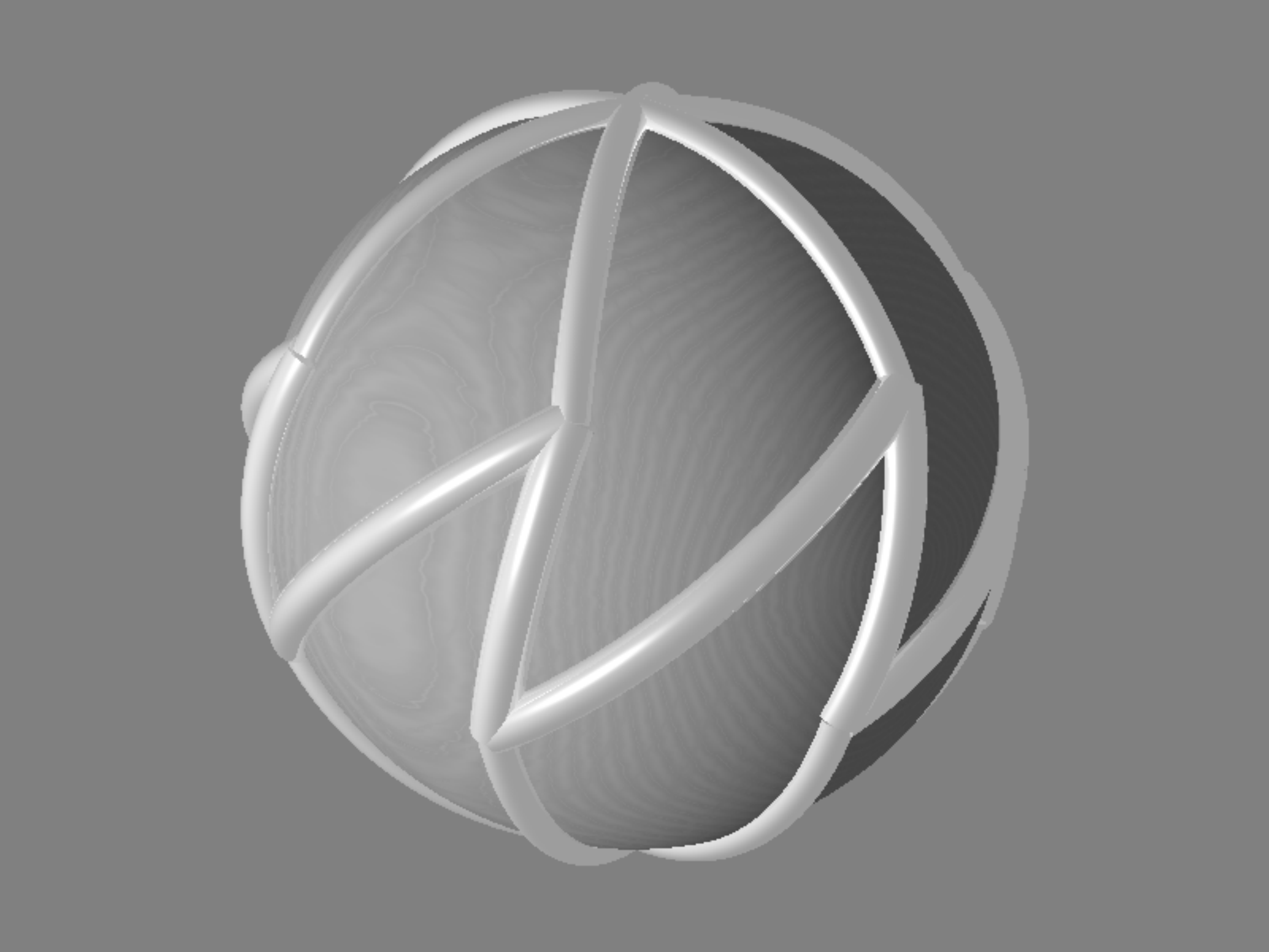}}};
 %
 \draw[<-] (3.6,9.7) -- (11.9,12);
 \node at (3.05,9.6) {$*$};
 \node at (4.5,-0.9) {$\frac54-\frac{1}{2n}$}; 
\draw[dotted] (7.1,6.9) -- (7.1,-2);
\node at (7.1,-2.9) {$\frac32-\frac{1}{n}$};
 \node at (-10,-.2) {\Large $B_n^{(5)}$};	  
\draw[dotted] (10.2,6.9) -- (10.2,-1.9);
\node at (10.2,-2.9) {$\frac32$};
\node at (11,21)
 {\pgftext{\includegraphics[scale=.09]{fig_af_upper_left.pdf}}}; 
 %
 \draw[<-] (-1.8, 12) -- (7,18); 
 \node at (-2.1, 11.7) {$\circ$};  
 \node at (-13.6, .18) {$\times$}; 
 \draw[<-] (-13, 0.4)  -- (-12, 0.9); 

 \node at (-11,3.8)  {\pgftext{\includegraphics[scale=.08]{fig_af_upper_right.pdf}}};
\draw[axis] (-15.5, -1.8) -- (12, -1.8) node(xline)[right] {$\alpha$ [$\pi$ rad]};
\draw[axis] (-15.5, -1.8) -- (-15.5, 26) node(yline)[above] {$\gamma$  [$\pi$ rad]};
 \end{tikzpicture}
  \caption{The set $B_n$~($n=6$). See Theorem~\ref{thm:main}.  The
  edge-length $a$ is $a_{n,\alpha,\gamma}^+$ for $B_n^{(2)}\cup
  B_n^{(5)}\cup B_n^{(6)}$; $a_{n,\alpha,\gamma}^-$ for
  $(\alpha,\gamma)\in B_n^{(1)}$; $a_{n,\alpha,\gamma}^-$ or
  $a_{n,\alpha,\gamma}^+$ for $B_n^{(4)}\cup B_n^{(8)}$; and
  $a_{n,\alpha,\gamma}^-=a_{n,\alpha,\gamma}^+$ for $B_n^{(3)}\cup
  B_n^{(7)}$. The point designated by $\circ$~($\times$, resp.)
  corresponds to the upper left~(upper right, resp.) image of tiling in
  Figure~\ref{fig:af}.  \label{fig:phasediagram3} }
\end{figure}

\begin{remark}\rm
The inner angles $(\alpha,\gamma)$
 of $\T\in\TRPZ_n$ ranges over Figure~\ref{fig:phasediagram3}.
The coordinate system $\langle\varphi,\, a\rangle$ for $\TRPZ_n$ is introduced in Definition~\ref{def:ani}. 
$\langle\varphi,\,
 a\rangle$ ranges over Figure~\ref{fig:af}.
 Figure~\ref{fig:phasediagram3} corresponds to Figure~\ref{fig:af}, as follows:

In Figure~\ref{fig:phasediagram3}, the open set above $\gamma=\pi$, the open set right to
 $\alpha=\pi$, and the open set below $\gamma=\pi/2$ correspond to
 $A^{(1)}_n$, $A^{(3)}_n$, and $A^{(4)}_n$ of Definition~\ref{def:ani} and
 Figure~\ref{fig:af}, respectively.

Here is the proof.  By Theorem~\ref{thm:main}~\eqref{assert:clas} and
Figure~\ref{fig:cs}, $\bigcup_{i=2}^4B_n^{(i)}$ corresponds to
$A_n^{(1)}$, and $\bigcup_{i=6}^8B_n^{(i)}$  to $A_n^{(3)}$. By
Theorem~\ref{thm:main}~\eqref{assert:clas} and Lemma~\ref{lem:shorta},
$B_n^{(5)}$ corresponds to $A_n^{(4)}$. So, by Theorem~\ref{thm:main}
and Theorem~\ref{thm:phia}, $B_n^{(1)}$ corresponds to $A_n^{(2)}$.

 Suppose that 
 the length $a$ of the meridian edges is 0. Then $2n$ non-meridian edges
 split the sphere where the inner angle of each digon is $\delta=2\pi
 -\alpha-\gamma$. So $\delta=\pi/n$. Hence, $a=0$ implies $\alpha+\gamma=\pi(2-1/n)$. 

If $a=\pi$, then the tile is the union of a digon of angle $\alpha$ and that of
 angle $\gamma$, so $\beta=\alpha+\gamma=2\pi/n$.

$a=\pi/2$ corresponds to $\alpha=\pi/2$ or $\gamma=\pi/2$, by the proof of Lemma~\ref{lem:shorta}.
\end{remark}

Theorem~\ref{thm:main}  follows from Theorem~\ref{thm:phase1}~(Subsection~\ref{subsec:contain}) and
Theorem~\ref{thm:near}~(Subsection~\ref{subsec:ambiguity}). 
Proposition~\ref{prop:vinberg}~\eqref{assert:vinberg} plays an important role in
Subsection~\ref{subsec:contain}. 
In Figure~\ref{thecurves},
the two disjoint regions circumscribed by a solid curve, dashed line and dotted line do not correspond to $\TRPZ_n$-quadrangles.
It is because every zero of $f_{n,\alpha,\gamma}$ is greater than $1$. See Lemma~\ref{lem:jogai} of Subsection~\ref{subsec:ambiguity}.

\subsection{$\TRPZ_n$-quadrangles containing the meridian
 diagonal geodesic segment} \label{subsec:contain}

In any tiling of $\TRPZ_n$,  the tile $N v_0 v_1 v_2$ contains a segment $N v_1$, if
and only if $\alpha,\gamma \in (0,\, \pi/2)$ or $\alpha,\gamma\in
(\pi/2,\, \pi)$.
 Theorem~\ref{thm:phase1}~\eqref{assert:phase} and Theorem~\ref{thm:phase1}~\eqref{lem:convex}  correspond
to the two open pentagons $B_n^{(5)}$ and $B_n^{(1)}$ of
Theorem~\ref{thm:main}, respectively.
 For given $n,\alpha,\gamma,a$, there is at most one
$\TRPZ_n$-quadrangle $\PQ_{n,\alpha,\gamma,a}$~(See Fact~\ref{fact:q} and Definition~\ref{def:ispdw Q}). We observe
  \begin{align}\label{sharp}\qquad
  f_{n,\alpha,\gamma}(\pm 1)=\mp \csc \alpha\csc \gamma \csc \frac{\pi}{n} \sin
 \left (\pm \frac{\pi}{n}+\alpha+\gamma\right). 
  \end{align}
  The axis of the parabola $y=f_{n,\alpha,\gamma}(x)$ is
  \begin{align}
                    axis(n,\alpha,\gamma):=\frac{1}{2}\cot\frac{\pi}{n}(\cot\alpha+\cot\gamma) \label{def:axis}
  \end{align}
\begin{theorem}\label{thm:phase1} Let $n\ge3$.
\begin{enumerate}
\item \label{assert:phase} Let $\alpha,\gamma \in \left(0,\,
      \fra{\pi}{2}\right)$. Then an $\TRPZ_n$-quadrangle
      $\PQ_{n,\alpha,\gamma,a}$ exists for some $a\in (0,\,\pi)$, if and only if
      $\alpha+\gamma>\fra{\pi}{n}$.
In this case,  $a=a_{n,\alpha,\gamma}^+$.

\item 
\label{lem:convex}
Let $\alpha,\gamma \in \left(\fra{\pi}{2},\, \pi\right)$. 
Then an $\TRPZ_n$-quadrangle $\PQ_{n,\alpha,\gamma,a}$ exists for some $a\in (0,\,\pi)$,
      if and only if $\alpha+\gamma<2\pi-\fra{\pi}{n}$.  In this case, 
      $a=a_{n,\alpha,\gamma}^-$.
\end{enumerate} 
\end{theorem}
 \begin{proof}
(Only-if part of Theorem~\ref{thm:phase1}~\eqref{assert:phase}). See
Figure~\ref{fig:cs}.  By $\alpha,\gamma\in (0,\, \pi/2)$, a segment $N
v_1$ is in the quadrangle.  To the two spherical 3-gons $v_0 N v_1$ and
$v_2 N v_1$, apply the last inequality of
Proposition~\ref{prop:vinberg}~\eqref{assert:vinberg}. Then $- \alpha +
\angle v_0 N v_1 + \angle N v_1 v_0 < \pi$ and $- \gamma + \angle v_1 N
v_2 + \angle v_2 v_1 N < \pi$. The sum of the left-hand sides of the two
inequalities is $-\alpha-\gamma+\beta +\delta$, and is less than $2\pi$.
By $\delta=2\pi - \alpha - \gamma$ and $\beta=2\pi / n$, we have $\pi>
\alpha+\gamma>\fra{\beta}2=\fra{\pi}{n}$.

  \medskip (If part of
  Theorem~\ref{thm:phase1}~\eqref{assert:phase}). For any
  $a\in(0,\,\pi)$, there is
a spherical isosceles 3-gon $v_0 N v_2$ such that $N v_0 = N v_2 = a$
  and $\angle v_0 N v_2 = 2\pi/n$.   By $\alpha,\gamma>0$, there is a
  vertex $v_1$ between $N v_0$ and $N v_2$ such that $\angle v_1 v_0 N =
  \alpha$ and $\angle N v_2 v_1=\gamma$. So, $v_0
  v_1$ does not cross to $N v_2$. Hence, by Theorem~\ref{thm:existence Pn},
  $Q_{n,\alpha,\gamma,a}$ exists if and only if there are a spherical 3-gon $v_0
  v_1 v_2$ and $a\in (0,\,\pi)$ such that $N v_0 = N v_2 = \pi - N v_1 =
  a$ and $f_{n,\alpha,\gamma}(\cos a) = 0$.

  Put four vertices $N,v_0,v_1,v_2$ such that $N$ is the north pole, $N
  v_0=N v_2= \pi - N v_1 =a$. Then, $v_0$ lies in the southern hemisphere if and only if $v_1$
  lies in the northern hemisphere.  As $\angle v_1 v_0 N=\alpha<\pi/2$
  by the premise, $v_0$
  lies in the southern hemisphere and $\pi>a>\pi/2$. The geodesic segment between $v_0$
  and $v_2$ lies inside the southern hemisphere.   Hence,
  the 3-gon $v_0 v_1 v_2$ is a subset of the 3-gon $v_0 N v_2$.  
  Let
  $\theta=\angle v_2 v_0 N = \angle N v_2 v_0$. 
  The inner angles at
  $v_1$ of  the 3-gon $v_0 v_1 v_2$ is $\alpha+\gamma=2\pi - \delta$,
  because the assumption $\alpha,\gamma\in (0,\,\pi/2)$ implies
  $\alpha+\gamma<\pi$.
  The other two inner angles of $v_0 v_1 v_2$ are $\theta -\alpha,
  \theta - \gamma$. Therefore,  by Proposition~\ref{prop:vinberg}~\eqref{assert:vinberg}, the existence of the 3-gon $v_0 v_1 v_2$ is equivalent to $2\theta>\pi,
  2\alpha<\pi, 2\gamma<\pi$, and  $2\theta -2\alpha -2\gamma<\pi$. Thus,
  $N v_0 v_1 v_2$ is a quadrangle if and only if
  $\pi/2<\theta<\pi/2+\alpha+\gamma$. As a spherical 3-gon $v_0 S v_2$
  exists, $2(\pi - \theta) + \beta>\pi$,
  i.e.,  $\theta<\pi/2+\beta/2 = \pi/2 + \pi/n$. The assumption $\pi/n<\alpha+\gamma$ implies
  $\theta <\pi/2 + \alpha+\gamma$.  Moreover, $\pi/2<\theta$, as $v_0$ and $v_2$ are in the southern hemisphere.
  Hence, if there is an $a\in (0,\,
  \pi)$ such that $f_{n,\alpha,\gamma}(\cos a) =0$, then 
  the 3-gon $v_0 v_1 v_2$ exists
  such that $N v_0 = N v_2 = \pi - N v_1 = a$.

 In this case, we show $a=a^+_{n,\alpha,\gamma}\in
  (0,\,\pi)$ exists and $f_{n,\alpha,\gamma}(\cos
  a^+_{n,\alpha,\gamma})$ is 0.  As $a\in(\pi/2,\, \pi)$, $\cos a\in (-1,\,0)$. By
  the assumption, $\pi>\alpha+\gamma>-\pi/n+\alpha+\gamma>0$. 
By \eqref{sharp},
  $f_{n,\alpha,\gamma}(-1)>0>f_{n,\alpha,\gamma}(0)=-\cot\alpha\cot\gamma$.
  The axis $axis(n,\alpha,\gamma)$ of the parabola
  $y=f_{n,\alpha,\gamma}(x)$ is positive, as $\alpha,\gamma\in(0,\, \pi/2)$. So the intersection of 
  $(-1,\,0)\times\{0\}$ and the parabola $y=f_{n,\alpha,\gamma}(x)$ is the smaller intersection
  point of $\Rset\times\{0\}$ and the parabola. Hence
  $a=a^+_{n,\alpha,\gamma}$.

\eqref{lem:convex}. $Q_{n,\alpha,\gamma,a}$ exists if and only if
$Q_{n,\pi-\alpha,\pi-\gamma,\pi - a}$ does so. It is because from any
$\T\in\TRPZ_n$, we obtain $\T'\in\TRPZ_n$, by joining the vertices
$N$ and $S$ to the opposite vertices $v_{2i+1}$ and $v_{2i}$
  respectively, and then deleting the $2n$ meridian edges $N v_{2i}$ and
  $S v_{2i+1}$  of $\T$. As $\pi - \alpha, \pi - \gamma \in
(0,\, \pi/2)$, Theorem~\ref{thm:phase1}~\eqref{assert:phase} implies
$\pi - a = a^+_{n,\pi - \alpha, \pi - \gamma}$.  Hence, $a=
a^-_{n,\alpha,\gamma}$, by the explanation at the beginning of
Subsection~\ref{subsec:discriminant} and $\arccos(x)=\pi -
\arccos(-x)$.  This establishes Theorem~\ref{thm:phase1}.  \qed
\end{proof}

\subsection{$\TRPZ_n$-quadrangles with $\alpha$ or $\gamma$  greater than $\pi$} \label{subsec:ambiguity}

Theorem~\ref{thm:near}~\eqref{assert:ineqdgn} and
Theorem~\ref{thm:near}~\eqref{assert:unique} correspond to an open set
 $B_n^{(4)}\cup B_n^{(8)}$ and a set $B_n^{(2)}\cup B_n^{(3)}\cup
 B_n^{(6)}\cup B_n^{(7)}$ of Theorem~\ref{thm:main}, respectively.  

 \begin{theorem}\label{thm:near}
  Suppose $\alpha>\pi$ or $\gamma>\pi$. Then we have the following:
\begin{enumerate}

\item \label{assert:ineqdgn}
There are more than one, actually, exactly \emph{two} $\TRPZ_n$-quadrangles $\PQ_{n,\alpha,\gamma,a}$, if and only if
\begin{eqnarray}
\frac{\pi}{2}<\alpha< \frac{3\pi}{4}-\frac{\pi}{2n}  &
\& &2\pi-\frac{\pi}{n}-\alpha< \gamma< \degenerate_n(\alpha); \ \mbox{or}
 \label{ineq:discriminant1} \\
\frac{\pi}{2}<\gamma< \frac{3\pi}{4}-\frac{\pi}{2n} &
\& & 2\pi-\frac{\pi}{n}-\gamma< \alpha< \degenerate_n(\gamma).
 \label{ineq:discriminant1 dual} 
\end{eqnarray}
      There is indeed a pair $(\alpha,\,\gamma)$ satisfying
\eqref{ineq:discriminant1} or \eqref{ineq:discriminant1 dual}. Here
      the edge-length $a$ is $a_{n,\alpha,\gamma}^+$ or $a_{n,\alpha,\gamma}^-$.

 \item \label{assert:unique}
There is a unique $\TRPZ_n$-quadrangle, if and only if
\begin{eqnarray}
\frac{\pi}{2}< \alpha< 2\pi-\frac{\pi}{n} - \gamma\ &\& \ &\pi<\gamma; \label{ineq:unique}\\
\gamma=\degenerate_n(\alpha) \ &\&\ &\frac{\pi}{2}<\alpha<\frac{3\pi}{4} - \frac{\pi}{2n};\label{eq:dgn}\\
\frac{\pi}{2}< \gamma< 2\pi-\frac{\pi}{n} - \alpha\ &\& \ &\pi<\alpha;\ \mbox{or} \label{ineq:unique:gamma}\\
\alpha=\degenerate_n(\gamma)\ &\&\ &\frac{\pi}{2}<\gamma<\frac{3\pi}{4} -
\frac{\pi}{2n}. \nonumber
\end{eqnarray}
	    If \eqref{ineq:unique} or \eqref{ineq:unique:gamma} hold, then the
	    edge-length  $a$ is
	    $a_{n,\alpha,\gamma}^+$. If $\gamma=\degenerate_n(\alpha)$, then
       the edge-length $a$ is 
       \begin{align*}
	      \pi -\arccos \left( \left( \sec\frac {\pi }{n}\right)   \left( \sin \frac {\pi }{n} +1 \right)
 \cot \alpha \right)=a_{n,\alpha,\gamma}^+=a_{n,\alpha,\gamma}^-.
       \end{align*}
       If $\alpha=\degenerate_n(\gamma)$, then
       the edge-length $a$ is 
       \begin{align*}
	      \pi -\arccos \left( \left(\sec\frac {\pi }{n} \right)   \left( \sin \frac {\pi }{n} +1 \right)
 \cot \gamma \right)=a_{n,\alpha,\gamma}^+=a_{n,\alpha,\gamma}^-.
       \end{align*}
\end{enumerate}
\end{theorem}

To prove Theorem~\ref{thm:near}, we prove the following lemma. 
\begin{lemma}\label{lem:concavePn}
If some of condition~\eqref{ineq:discriminant1},
	condition~\eqref{ineq:unique}, condition~\eqref{eq:dgn}
	and the three conditions with $\alpha$ and $\gamma$ swapped hold,
	then for every $a\in (0,\,\pi/2)$ with $f_{n,\alpha,\gamma}(\cos
	a)=0$, there exists a $\PQ_{n,\alpha,\gamma,a}$-quadrangle.
\end{lemma}
  \begin{proof}
   The
  assumption implies
  \begin{align}
   0< \alpha+\gamma-\frac{3\pi}2 <\frac{\pi}2. \label{range17}
  \end{align}
Consider a  quadrangle $N v_0 v_1 v_2$ such that the
inner angle between two edges of length $a$ is $\beta=2\pi/n$, and the
two inner angles neighboring to $\beta$ are $\alpha$ and $\gamma$. We
 prove  $N v_0 v_1 v_2$ is indeed a \emph{spherical}
4-gon.  $\angle v_0 N v_2 = 2\pi/n<\pi$ and $v_0 N = v_2 N
   =a<\pi/2$. So, $\theta:=\angle v_2 v_0 N$ is strictly between $0$ and $\pi/2$.

  As a spherical 3-gon $v_0 N v_2$ clearly exists, a quadrangle
  $N v_0 v_1 v_2$ exists, if and only if a spherical 3-gon
$v _{0}v _{1}v _{2}$ exists. The last condition holds, if and only if $\alpha -
\gamma + \delta <\pi$, $ -\alpha + \gamma +\delta <\pi$, $(\alpha
-\theta )+ ( \gamma -\theta )-\delta <\pi$, and $(\alpha -\theta )+(
\gamma -\theta )+\delta >\pi$, by
Proposition~\ref{prop:vinberg}~\eqref{assert:vinberg}. $\alpha,\gamma
>\pi/2$ holds from the assumption of this lemma.  So,
$\alpha+\gamma+\delta=2\pi$ implies the first and the second of the four
inequalities.  The last inequality follows from $0<\theta<\pi/2$ and $\alpha+\gamma+\delta=2\pi$. The third inequality is $\alpha+\gamma
  -3\pi/2<\theta$. By $0<\theta<\pi/2$ and the range~\eqref{range17} of $\alpha+\gamma$,
\begin{align*}
  \mbox{ the quadrangle $N v_0 v_1
  v_2$ exists}
\iff
-\tan( \alpha+\gamma)>\cot\theta. 
  \end{align*}

Here $\cot\theta = \cos a \tan(\pi/n)$.  To see this,
represent the length of the base edge of a spherical isosceles 3-gon
 $v_0 N v_2$, in terms of $a,n$,  by using spherical cosine
 law~(Proposition~\ref{prop:vinberg}~\eqref{assert:scl}). By applying the spherical cosine law for angles~(Proposition~\ref{prop:vinberg}~\eqref{assert:scla})
 to  $v _{0} N v _{2}$, we get $\cos \left( 2\,{\fra {\pi }{n}} \right)
   =-  \cos^2 \theta +  \sin^2 \theta  \left(   \cos^2 a  +  \sin^2 a
   \cos \left( 2 \fra {\pi }{n} \right)  \right) $. As $\cos(\pi/n)>0$ by $n \ge 3$, $$\sin \theta=
 \frac{\cos\frac{\pi}{n}}{\sqrt{\sin^2\frac{\pi}{n} \cos^2 a + \cos^2 \frac{\pi}{n}}}.$$ So, as
 $\cos a > 0$ by the premise, $0<\theta<\pi/2$ implies 
 $\cot\theta = \cos a \tan(\fra{\pi}{n})$.

Hence, by Theorem~\ref{thm:existence Pn}, Lemma~\ref{lem:concavePn} is equivalent to: For every $n\ge3$, if
\eqref{ineq:discriminant1}, \eqref{ineq:unique}, or \eqref{eq:dgn}, then for every $a\in
(0,\, \pi/2)$ with $f_{n,\alpha,\gamma}(\cos a)=0$, we have
\begin{align}
  \cos a< -\cot \frac{\pi
  }{n}\tan \left( \alpha+\gamma  \right). \label{uvwx}
\end{align}

When condition~\eqref{ineq:unique} holds,  $-\tan( \alpha+\gamma
 )>\tan(\pi /n)$, and thus inequality~\eqref{uvwx} holds.

Assume condition~\eqref{ineq:discriminant1} or condition~\eqref{eq:dgn}.  By
   $\degenerate_n$,
 \begin{align}
 \frac{\pi}{2}<\alpha< \frac{3\pi}{4}-\frac{\pi}{2n},\quad\pi< \frac{5\pi}{4}-\frac{\pi}{2n}<\gamma< \frac{3\pi}{2}.\label{multiple}
 \end{align}
Thus $f_{n,\alpha,\gamma}(0)=-\cot\alpha\cot\gamma>0$. So two solutions
 of the quadratic equation $f_{n,\alpha,\gamma}(x)=0$ are of the same
 sign. Hence inequality~\eqref{uvwx} follows from
    \begin{align}
n\ge3,\ \eqref{ineq:discriminant1}\ \&\ f_{n,\alpha,\gamma}(x)=0\
   \Longrightarrow\ x < -\cot \frac {\pi }{n}\tan \left( \alpha+\gamma
   \right). \label{afo}
  \end{align}
As $f_{n,\alpha,\gamma}(x)$  is quadratic, condition~\eqref{afo} is
  equivalent to the conjunction of
 \begin{align}
& f_{n,\alpha,\gamma}\left(-\cot \frac {\pi
 }{n}{\tan ( \alpha+\gamma)}\right)>0,\label{goal1}\end{align}
and the condition  $axis(n,\alpha,\gamma)< - \cot (\pi/n){\tan ( \alpha+\gamma)}$:
  \begin{align}
&\frac{1}{2}\cot\frac{\pi}{n}(\cot\alpha+\cot\gamma) <-\cot \frac {\pi
  }{n}{\tan ( \alpha+\gamma)}\label{goal2}.
 \end{align}
 Inequality~\eqref{goal1} is proved as follows: By calculation, the left-hand
   side  is 
   \begin{align*}
\frac{   
    \sin \left( \alpha+\gamma  + \fra{\pi}{n}\right)
    \sin \left( \alpha+\gamma  - \fra{\pi}{n}\right)
    }{{\cos
 \alpha  \sin \alpha  \cos  \gamma \sin  \gamma  } \left( \tan \alpha \tan
    \gamma -1 \right) ^2 \sin ^2( \pi /n)}.
   \end{align*}
   This is positive, because
 \eqref{multiple} implies $\cos \alpha
   \sin\alpha\cos\gamma\sin\gamma<0$, and
   Lemma~\ref{lem:degenerate}~\eqref{unusual} implies 
$ 2\pi - \pi/n < \alpha+\gamma <2\pi$.  So \eqref{goal1} is established.

In inequality~\eqref{goal2}, the right-hand side divided by the left-hand side has
absolute value $M=|2\sin \alpha\sin \gamma /\cos(\alpha + \gamma
)|$. The right-hand side of \eqref{goal2} is positive by
   \eqref{ineq:discriminant1} and $\alpha+\gamma+\delta=2\pi$. So, we have only to show $M>1$.
The numerator $2\sin\alpha\sin\gamma$ is negative by \eqref{multiple}
and the denominator $\cos(\alpha+\gamma)$ is positive by
\eqref{ineq:discriminant1}. So $M>1$ is equivalent to
$\cos(\alpha-\gamma)<0$. By \eqref{multiple}, $\pi/2=(\fra{5\pi}{4}-
\fra{\pi}{(2n)}) - \left(\fra{3\pi}{4}-\fra{\pi}{(2n)}\right) <
\gamma-\alpha< \fra{3\pi}{2} - \fra{\pi}{2}=\pi$. This proves inequality~\eqref{goal2} and thus the implication~\eqref{afo}. So the spherical 3-gon $v_0 v_1
  v_2$ exists.  This establishes Lemma~\ref{lem:concavePn}.\qed\end{proof}

 By calculation,
 \begin{align*}
(\dag)\qquad axis(n,\alpha,\degenerate_n(\alpha))
 =-\sec\frac{\pi}{n}\left(\sin\frac{\pi}{n} +1\right)\cot\alpha.\end{align*}

\begin{lemma}\label{lem:jogai} Let $n=3,4,5,\ldots$.
Suppose \[
\frac{3\pi}{4} - \frac{\pi}{2n}<\alpha<\pi<\gamma<\frac{5\pi}{4} - \frac{\pi}{2n},\quad
2\pi - \frac{\pi}{n}<\alpha+\gamma,\ \mbox{and}\ 
\gamma<\degenerate_n(\alpha).\]
Then there is no $a\in(0,\,\pi)$ such that $f_{n,\alpha,\gamma}(\cos a)=0$.
\end{lemma}
\begin{proof} We  prove that $f_{n,\alpha,\gamma}(x)=0\implies x\ge 1$. We have only to verify $axis(n, \alpha, \gamma)>1$ and $f_{n,\alpha,\gamma}(1)\ge 0$.

We show $axis(n, \alpha, \gamma)>1$.   By the premise,
 $\cot(\pi/n)>0$ and $\pi/2<\alpha<\pi$. By \eqref{assert:dgn}
 and \eqref{usual} of  Lemma~\ref{lem:degenerate}, we have
 $\pi<\degenerate_n(\alpha)<3\pi/2$. So, by the premise, $\pi<\gamma<\degenerate_n(\alpha)<3\pi/2$.
Thus, by \eqref{def:axis},
 $axis(n,\alpha,\gamma)>axis(n,\alpha,\degenerate_n(\alpha))$. By (\dag)
 and  $\pi/2<\fra{3\pi}{4} - \fra{\pi}{(2n)}<\alpha<\pi$,
$axis(n,\alpha,\degenerate_n(\alpha))> axis(n,\,3\pi/4 -
 \pi/(2n),\,\degenerate_n(3\pi/4 -
 \pi/(2n)))$. The last is 1 by calculation.
Hence $axis(n, \alpha, \gamma)>1$. 

 Next, we  verify  $f_{n,\alpha,\gamma}( 1)\ge 0$.
 By the first premise $\fra{3\pi}{4} -
 \fra{\pi}{(2n)}<\alpha<\pi<\gamma<\fra{5\pi}{4} - \fra{\pi}{(2n)}$, we have
 $\alpha+\gamma+\pi/n<2\pi+\pi/4+\pi/(2n)$. So, by $n\ge3$ and the
 second premise,  $2\pi<\alpha+\gamma+\pi/n< 2\pi+5\pi/12$. Thus, by the first
 premise and \eqref{sharp}, $f_{n,\alpha,\gamma}( 1)\ge 0$.
This completes the proof of Lemma~\ref{lem:jogai}. \qed
\end{proof}

 \emph{Proof of Theorem~\ref{thm:near}}.\/ Let $\alpha>\pi$ or
$\gamma>\pi$.  The edge-length $a$ is smaller than $\pi/2$. Otherwise,
equivalence~\eqref{cond:deltaconvex} of Lemma~\ref{lem:shorta}
implies $\delta>\pi$. So $\alpha$ and $\gamma$ are both less than $\pi$,
which is absurd. So $0<\cos a<1$.

Theorem~\ref{thm:near}~\eqref{assert:ineqdgn} is proved as follows:
By Theorem~\ref{thm:existence Pn}, the following two assertions are equivalent:
\begin{itemize}
\item more
than one $\TRPZ_n$-quadrangles $\PQ_{n,\alpha,\gamma,a}$ exist.
\item the quadratic polynomial $f_{n,\alpha,\gamma}(x)$ has two
distinct zeros $x_1, x_2$ in an open interval $(0,\,1)$ such that
a quadrangle
$\PQ_{n,\alpha,\gamma,\arccos x_i}$ exists for each $x_i$ $(i=1,2)$.
\end{itemize}
Here the quadratic polynomial $f_{n,\alpha,\gamma}(x)$ has two
distinct zeros $x_1, x_2$ in an open interval $(0,\,1)$
if and only if  the following three are all true:
\begin{enumerate}[(i)]
\item \label{assert:i} $f_{n,\alpha,\gamma}(0)>0$ and $f_{n,\alpha,\gamma}(1)>0$; 

\item \label{assert:ii} $0< axis(n, \alpha, \gamma)<1$; and
       
\item\label{assert:iii} $\Delta_{n,\alpha,\gamma}>0$.
\end{enumerate}

Hence, more than one $\TRPZ_n$-quadrangles $\PQ_{n,\alpha,\gamma,a}$
exist, if and only if condition~\eqref{ineq:discriminant1}
or condition~\eqref{ineq:discriminant1 dual} holds. It is due to
\eqref{assert:alphagamma} of Lemma~\ref{lem:shorta},
Lemma~\ref{lem:concavePn} and the following:
\begin{claim}\label{claim:degenerate} For the three conditions mentioned
 above, the following holds:
\begin{enumerate}
 \item \label{assert:cond1:equiv} 
In case $\pi/2<\alpha<\pi<\gamma<3\pi/2$, inequality~\eqref{ineq:discriminant1}  $\iff$
       \eqref{assert:i} $\&$ \eqref{assert:iii}.

 \item \label{assert:cond1:equiv dual} 
In case $\pi/2<\gamma<\pi<\alpha<3\pi/2$, 
       inequality~\eqref{ineq:discriminant1 dual}  $\iff$
\eqref{assert:i} $\&$ \eqref{assert:iii}.

 \item \label{assert:2c} In each of the above-mentioned two cases, \eqref{assert:i}$\ \&\ \Delta_{n,\alpha,\gamma}\ge0\implies$ \eqref{assert:ii}.
 \end{enumerate}
\end{claim}
 \begin{proof}
 Claim~\ref{claim:degenerate}~\eqref{assert:cond1:equiv} is proved as follows:
  $f_{n,\alpha,\gamma}(0)=-\cot \gamma \cot \alpha>0$ by
the premise.  Hence, condition~\eqref{assert:i} is
equivalent to $f_{n,\alpha,\gamma}(1)>0$. Thus, by the premise and \eqref{sharp},
 \begin{align}
\mbox{condition~\eqref{assert:i}}\iff \alpha+\gamma>2\pi - \frac{\pi}{n}.   \label{aux}
 \end{align}
  By the premise and Lemma~\ref{lem:degenerate}~\eqref{assert:dgn},
  \begin{align}
\mbox{condition~\eqref{assert:iii}}\iff
   \gamma<\degenerate_n(\alpha). \label{cond:iiidgn}
  \end{align}
  See Figure~\ref{thecurves}.
By the premise and Lemma~\ref{lem:jogai}, 
\begin{align*}
(\ddag)\qquad \mbox{condition~\eqref{assert:i}}\ \&\
\Delta_{n,\alpha,\gamma}\ge0\implies \alpha<\frac{3\pi}{4} -
\frac{\pi}{2n}\ \mbox{or}\ \frac{5\pi}{4}- \frac{\pi}{2n}<\gamma
.\end{align*}
By \eqref{aux}, condition~\eqref{assert:i} implies $\alpha<\fra{3\pi}{4} -\fra{\pi}{(2n)}\iff \fra{5\pi}{4}- \fra{\pi}{(2n)}<\gamma$.
Thus, by \eqref{aux} and \eqref{cond:iiidgn}, we have $\eqref{ineq:discriminant1}\iff
       \eqref{assert:i} \ \&\ \eqref{assert:iii}.$
So, Claim~\ref{claim:degenerate}~\eqref{assert:cond1:equiv} holds. 
 The same argument with $\alpha$ and $\gamma$ swapped proves
Claim~\ref{claim:degenerate}~\eqref{assert:cond1:equiv dual}.

\medskip
We prove
 Claim~\ref{claim:degenerate}~\eqref{assert:2c}. 
 $axis(n,\alpha,\gamma)>0$ in either case, because
       $\cot\alpha+\cot\gamma=\sin(\alpha+\gamma)/\sin\alpha\sin\gamma>0$
 follows from $3\pi/2<\alpha+\gamma=2\pi-\delta<2\pi$.

Consider the case $\pi/2<\alpha<\pi<\gamma<3\pi/2$.
 As $\cot \gamma >0$,
\begin{align*}
\mbox{condition~\eqref{assert:ii}}\iff \gamma >\arccot\left({2 \tan \frac{\pi}{n}-\cot
 \alpha}\right)+\pi.
\end{align*}
Hence, by equivalence~\eqref{aux}, condition~\eqref{assert:ii}
 follows from
\begin{align}
\pi-\alpha-\frac{\pi}{n} &\ge \arccot\left({2 \tan \frac{\pi}{n}-\cot
 \alpha}\right). \label{ineq:u}
\end{align} 
The right-hand side is positive, because $\cot\alpha<0$ by $\pi/2<\alpha<\pi$. 
So, inequality~\eqref{ineq:u} is equivalent to
$\cot(\pi-\alpha-\fra{\pi}{n})\le 2\tan(\fra{\pi}{n}) - \cot \alpha
$. Subtract $\tan(\alpha-\pi/2)+\tan(\pi/n)$ from both hand sides of the
last inequality, and then divide them by $\tan(\pi/n)$. Thus, by the
addition formula of $\tan$, inequality~\eqref{ineq:u} is equivalent to
$\tan(\alpha-\pi/2)\tan(\alpha-\pi/2+\fra{\pi}{n})\le 1$. By the
assumption $\pi/2<\alpha<\pi$ and implication (\ddag), this holds because
the two arguments $\alpha-\pi/2$ and $(\alpha-\pi/2+\fra{\pi}{n})$ are
both in the interval $(0,\,\pi/2)$ and have mean less than $\pi/4$. Thus
inequality~\eqref{ineq:u} holds. The other case
$\pi/2<\gamma<\pi<\alpha<3\pi/2$ is proved by the same argument with
$\alpha$ and $\gamma$ swapped. This completes the proof of
Claim~\ref{claim:degenerate}.\qed \end{proof}

\medskip
 Theorem~\ref{thm:near}~\eqref{assert:unique} is proved as
 follows: First observe that there exists exactly one $\TRPZ_n$-quadrangle, if and only if a
 quadrangle $N v_0 v_1 v_2$ exists and
 \begin{enumerate}[(a)]
  \item \label{cond:degenerate one} $f_{n,\alpha,\gamma}(x)$ has
	degenerate~(i.e., double) zero strictly between $0$ and $1$; or
	
  \item \label{assert:one} $f_{n,\alpha,\gamma}(x)$ has
distinct two zeros, but only one  in the interval $(0,\, 1)$.
 \end{enumerate}
  Here the edge-length $a$ is less than
 $\pi/2$, from $\alpha>\pi$ or $\gamma>\pi$, by equivalence~\eqref{cond:deltaconvex} of
 Lemma~\ref{lem:shorta}.
 
We prove 
 that \eqref{cond:degenerate
 one}$\iff(\gamma-\degenerate_n(\alpha))(\alpha-\degenerate_n(\gamma))=0$,
 as follows: Note that condition~\eqref{cond:degenerate one} holds if and
 only if we have all of condition~\eqref{assert:i}, 
 condition~\eqref{assert:ii} and $\Delta_{n,\alpha,\gamma}=0$. By
 Claim~\ref{claim:degenerate}~\eqref{assert:2c} and \eqref{aux}, the
 condition~\eqref{cond:degenerate one} is equivalent to
 $\alpha+\gamma>2\pi-\pi/n\ \&\ \gamma=\degenerate_n(\alpha)$ or to
 $\alpha+\gamma>2\pi-\pi/n\ \&\ \alpha=\degenerate_n(\gamma)$. Because $2\pi - \fra{\pi}{n} - \alpha < \degenerate_n(\alpha)$ by
 Lemma~\ref{lem:degenerate}~\eqref{unusual}, the equation $\gamma=\degenerate_n(\alpha)$ implies  $\alpha+\gamma>2\pi-\pi/n$. So, the
 condition~\eqref{cond:degenerate one} is equivalent to $\gamma=\degenerate_n(\alpha)$ or $\alpha=\degenerate_n(\gamma)$.
 This establishes the desired equivalence.

The quadratic equation $f_{n,\alpha,\gamma}(\cos a)=0$ of $\cos a$ has
the two solutions $a=a_{n,\alpha,\gamma}^+, a_{n,\alpha,\gamma}^-$,
presented at the beginning of Subsection~\ref{subsec:discriminant}.  If
the two solutions are degenerate solution
$a=a_{n,\alpha,\gamma}^+=a_{n,\alpha,\gamma}^-=\arccos(\cot(\pi/n)(\cot\alpha+\cot\gamma)/2)$,
then $\Delta_{n,\alpha,\gamma}=0$.

\begin{claim}\label{claim:degeneratea}The arccosine of the
 degenerate~$($i.e., double$)$
 solution $x$ of\linebreak[4] $f_{n,\alpha,\gamma}(x)=0$ is \eqref{degenerate sol3}
 for $\gamma=\degenerate_n(\alpha)$ and \eqref{degenerate sol7} for
$\alpha=\degenerate_n(\gamma)$.\end{claim}
\begin{proof} By
 Lemma~\ref{lem:degenerate}~\eqref{assert:dgn}, either
 $\gamma=\degenerate_n(\alpha)$ and $\pi/2<\alpha<\pi<\gamma<3\pi/2$, or
 $\alpha=\degenerate_n(\gamma)$ and
 $\pi/2<\gamma<\pi<\alpha<3\pi/2$. Consider the first case.
 By (\dag),
 $a$ is
 \eqref{degenerate sol3} for $\gamma=\degenerate_n(\alpha)$. The proof
 for case $\alpha=\degenerate_n(\gamma)$ is similar. This completes the
 proof of Claim~\ref{claim:degeneratea}.\qed\end{proof}

It is easy to see that condition~\eqref{assert:one}$\iff f_{n,\alpha,\gamma}(0)f_{n,\alpha,\gamma}(1)<0$. As
  $f_{n,\alpha,\gamma}(0)>0$ by the
  implication~\eqref{assert:alphagamma} of
  Lemma~\ref{lem:shorta},
  \begin{align*}
   \eqref{assert:one}\iff
  f_{n,\alpha,\gamma}(1) = - {\sin\left(\alpha+\gamma+\frac{\pi}{n}\right)}{\csc\alpha\csc\gamma\csc\frac{\pi}{n}}<0.
  \end{align*} By  equivalence~\eqref{aux}, \eqref{assert:one} is
  equivalent to condition~\eqref{ineq:unique} of
  Theorem~\ref{thm:near}, for $\pi/2<\alpha<\pi<\gamma$; and is
  equivalent to condition~\eqref{ineq:unique:gamma} of
  Theorem~\ref{thm:near}, for $\pi/2<\gamma<\pi<\alpha$. 
 Because $f_{n,\alpha,\gamma}(0)>0$ and $f_{n,\alpha,\gamma}(1)<0$ hold, 
  the unique solution $x$ of the quadratic equation $f_{n,\alpha,\gamma}(x)=0$ strictly between $0$
  and $1$ is the smaller solution of the equation. Therefore the edge-length $a$ is $a_{n,\alpha,\gamma}^+$.
Hence Lemma~\ref{lem:concavePn} establishes
 Theorem~\ref{thm:near}~\eqref{assert:unique}.  Thus
 Theorem~\ref{thm:near} is proved. \qed

\bigskip
From Theorem~\ref{thm:phase1} and Theorem~\ref{thm:near}, Theorem~\ref{thm:main} follows.

\section{A quadrangle organizing  both
 non-isohedral tiling and isohedral one over the same skeleton\label{sec:noniso}}
 
Recall an $\TRPZ_6$-quadrangle $\PQ_{6,\alpha,\gamma,a}$ from
Definition~\ref{def:ispdw Q}. 
\begin{theorem}\label{thm:do} Copies of a spherical $4$-gon
$T:=\PQ_{6,\arccos\frac{-1}{2\sqrt7},\frac{4\pi}{3},\arccos\frac{1}{3}}$
 organize both an \emph{isohedral}
 tiling $\T'$ $($Figure~\ref{fig:B234}~$($middle,right$))$ and a
 \emph{non-isohedral} tiling
 $\T$~$($Figure~\ref{fig:B234}~$($middle,left$)$,
 \cite{akama12:_class_of_spher_tilin_by_i}$)$ such that the skeletons
 are the same pseudo-double wheel. The quadratic
                 equation associated to  $T$ of
Theorem~\ref{thm:existence Pn} is $(x-1/3)^2$.\end{theorem}

\begin{proof} The edge-lengths and inner angles of $\T$ are as in
 Figure~\ref{chart:a}~(lower). So, a
tile~(designated $N 2 3 4$ in the figure) of $\T$ has two
edges of length $a$, both incident to the vertex $N$. $N$ is antipodal
to a vertex $S$, because there are two congruent paths between the two vertices in
Figure~\ref{chart:a}~(lower). The edge between a vertex $\delta$~(designated by
$3$ in Figure~\ref{chart:a}~(lower)) and $S$ is $a$, by the
figure. The area of the tile of $\T$ is $4\pi/12$, as
$\T$ is a spherical tiling by twelve congruent tiles. So, the tile
of $\T$ is an $\TRPZ_6$-quadrangle, by
Fact~\ref{fact:Pquad}~\eqref{fact:Pquad}.
By the definition of $f_{n,\alpha,\gamma}(x)$ in Theorem~\ref{thm:existence Pn}, we have $f_{6,\arccos(\fra{-1}{2\sqrt7}),\fra{4\pi}{3}}(x)=(x-1/3)^2$. \qed
\end{proof}

We conjecture that
$\PQ_{6,\arccos\frac{-1}{2\sqrt7},\frac{4\pi}{3},\arccos\frac{1}{3}}$ is
the only spherical 4-gon such that copies of it organize both a
non-isohedral tiling and an isohedral tiling over a pseudo-double wheel.
The conjecture is true by
\cite[Theorem~2]{akama12:_class_of_spher_tilin_by_i}, once the
following is proved: from any spherical non-isohedral tiling by
congruent $\PQ_{n,\alpha,\gamma,a}$ over a pseudo-double wheel, we can
obtain such a tiling $\T$ satisfying the condition~(2) of
\cite[Theorem~2]{akama12:_class_of_spher_tilin_by_i}.

To generalize Theorem~\ref{thm:do}, we want to enumerate all spherical
polygons which organize both \emph{non-isohedral} tilings and
\emph{isohedral} tilings over the same skeletons.  This is a weak
inverse problem of the following theorem:
\begin{proposition}[\protect{Gr\"unbaum-Shephard~\cite{MR661770}}]The
skeleton of a spherical isohedral tiling is exactly a pseudo-double
wheel, the skeleton of a bipyramid, that of a Platonic solid, or that of
 an Archimedean dual.
\end{proposition}

\begin{figure}\centering
  \begin{tikzpicture}[scale=0.3,axis/.style={very thick, ->, >=stealth'}]
  \node at (4,7)
  {\pgftext{\includegraphics[scale=0.3]{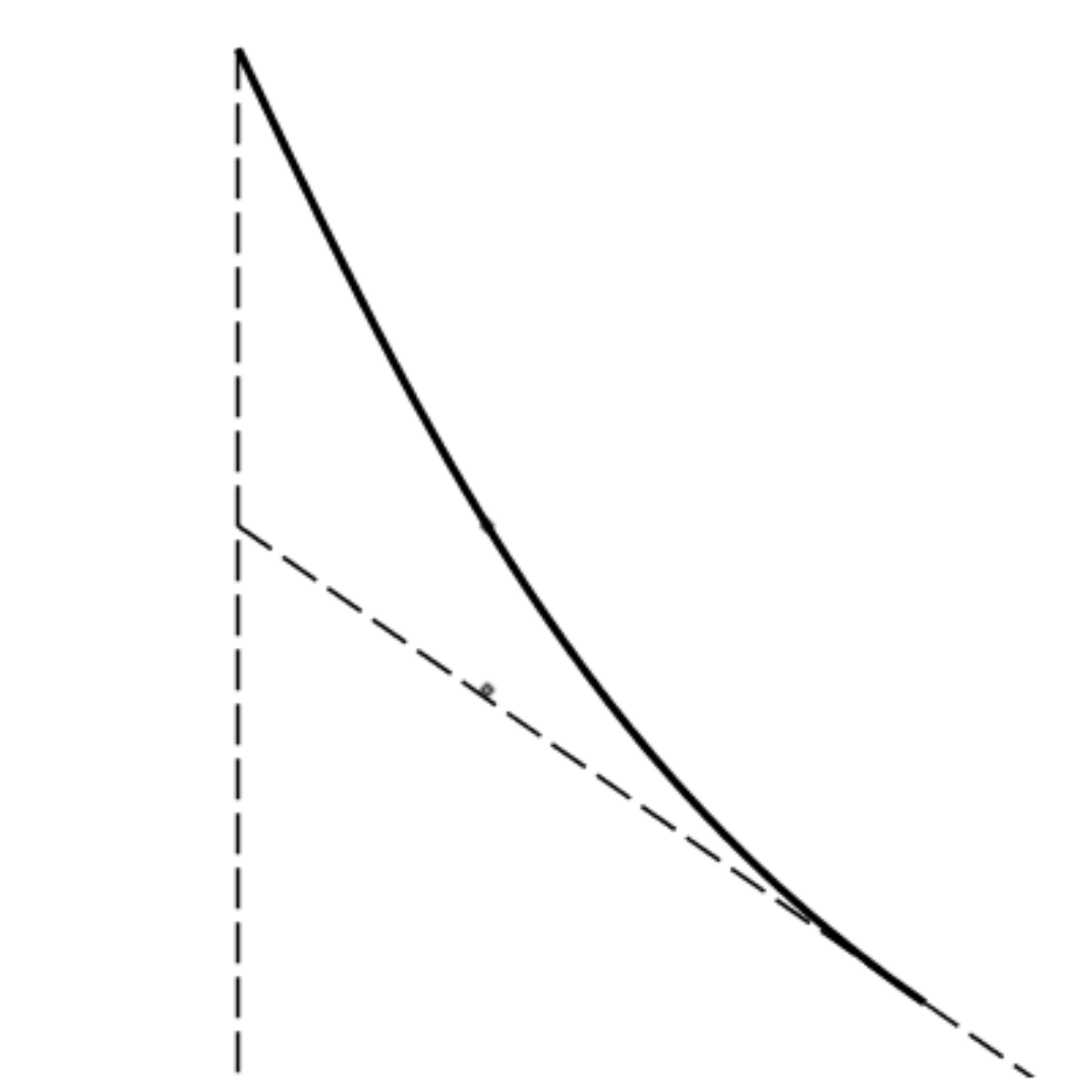}}};
  \draw[axis] (0, 0) -- (0, 16) node(xline)[above] {$\gamma$ [rad]};
  \draw[axis] (0, 0) -- (16,0) node(yline)[below] {$\alpha$  [rad]};
      \draw[<-]  (3.5,7.25) -- (14,8.9);
   \node at (17,8.7)   {\pgftext{\includegraphics{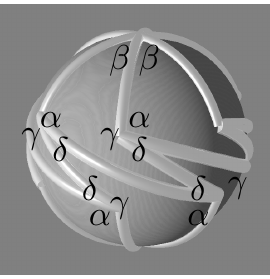}}};
  \node at (-1,13.3) {$\frac{3\pi}{2}$};
  \node at (-2, 7.25) {$\frac{3\pi}{2} - \frac{\pi}{n}$};
  \draw[dotted] (0,7.25) -- (3.2,7.25);
  \draw[dotted] (3.2,7.2) -- (3.2,0);
  \node at (3.2,7.25) {$\bullet$};
  \node at (1.2,8.6) {$B_n^{(4)}$};
  \node at (5,2) {$B_n^{(2)}$};
  \node at (3.35,-1.3) {\scriptsize $\arccos\dfrac{-1}{2\sqrt7}$};
  \draw[dotted] (8.81, 1.1) -- (8.81, 0);
  \draw[dotted] (8.81, 1.1) -- (0, 1.1);
  \node at (-2.2,1.1) {$\frac{5\pi}{4} - \frac{\pi}{2n}$};
  \node at (8.81,-1) {$\frac{3\pi}{4} - \frac{\pi}{2n}$};

  \node at (-0.0,-1) {$\frac{\pi}{2}$};
  \node at   (.1683097875, 10.98907025) {$\#$};
  \draw[<-,dotted]  (.8, 10.9) -- (12, 15);
  \node at (14, 15.1) {\pgftext{\includegraphics[scale=0.07]{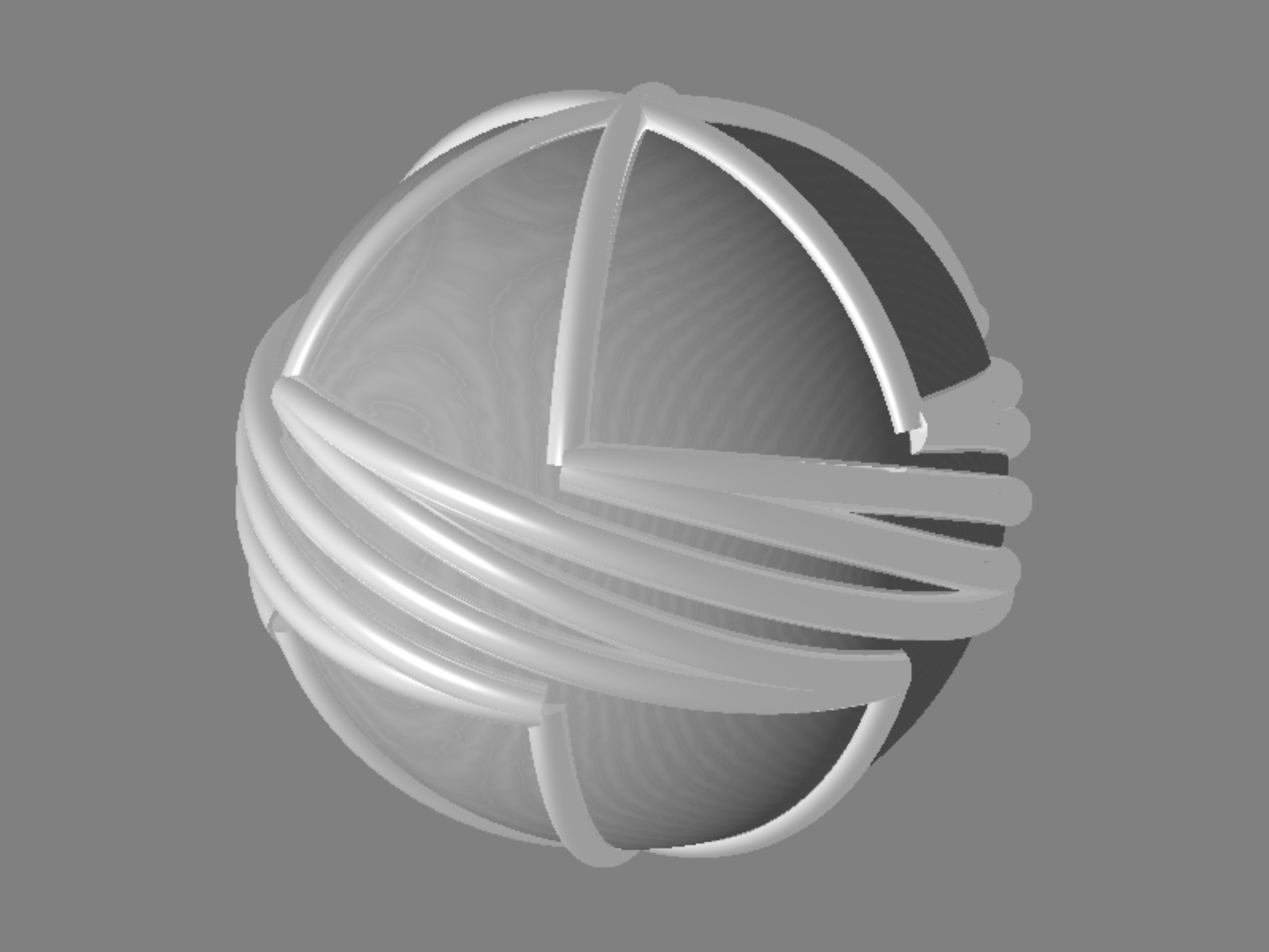}}};
  \node at (20, 15.1) {\pgftext{\includegraphics[scale=0.07]{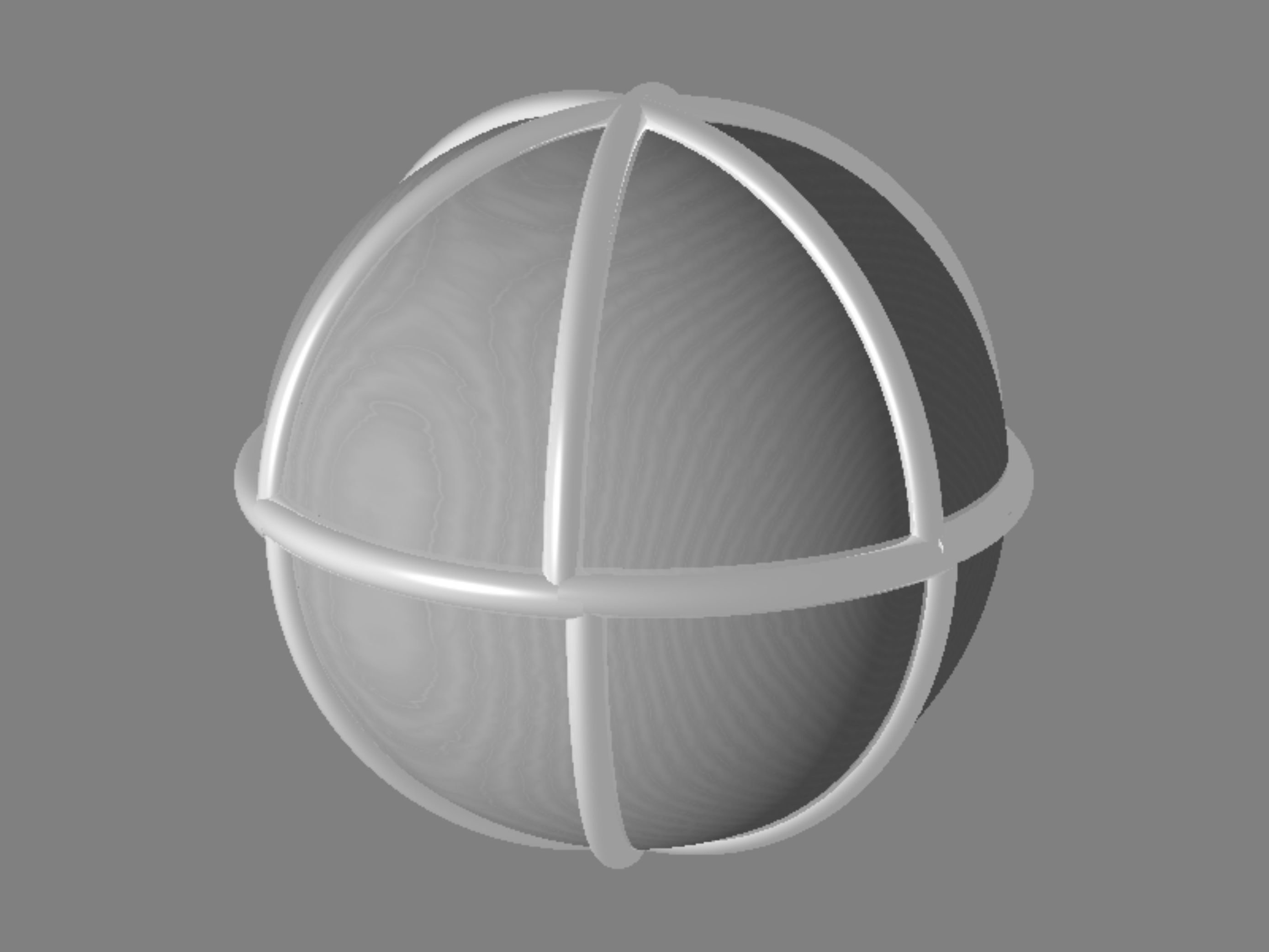}}};
   \draw[->]  (-9,10) -- (2.8,7.35);
\node at (-10,8.9)
  {\pgftext{\includegraphics{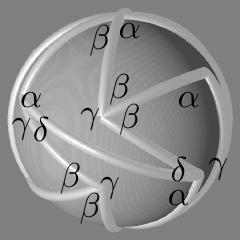}}};
   \node at (-15,11) {$\T$};
   \node at (11.5,11) {$\T'$};
  \node at (3.2, 6.8) {$\circledcirc$};
\draw[dotted,<-]  (3.84,6.45) -- (11,4);  
 \node at (14,2.6) {\pgftext{\includegraphics[scale=0.07]{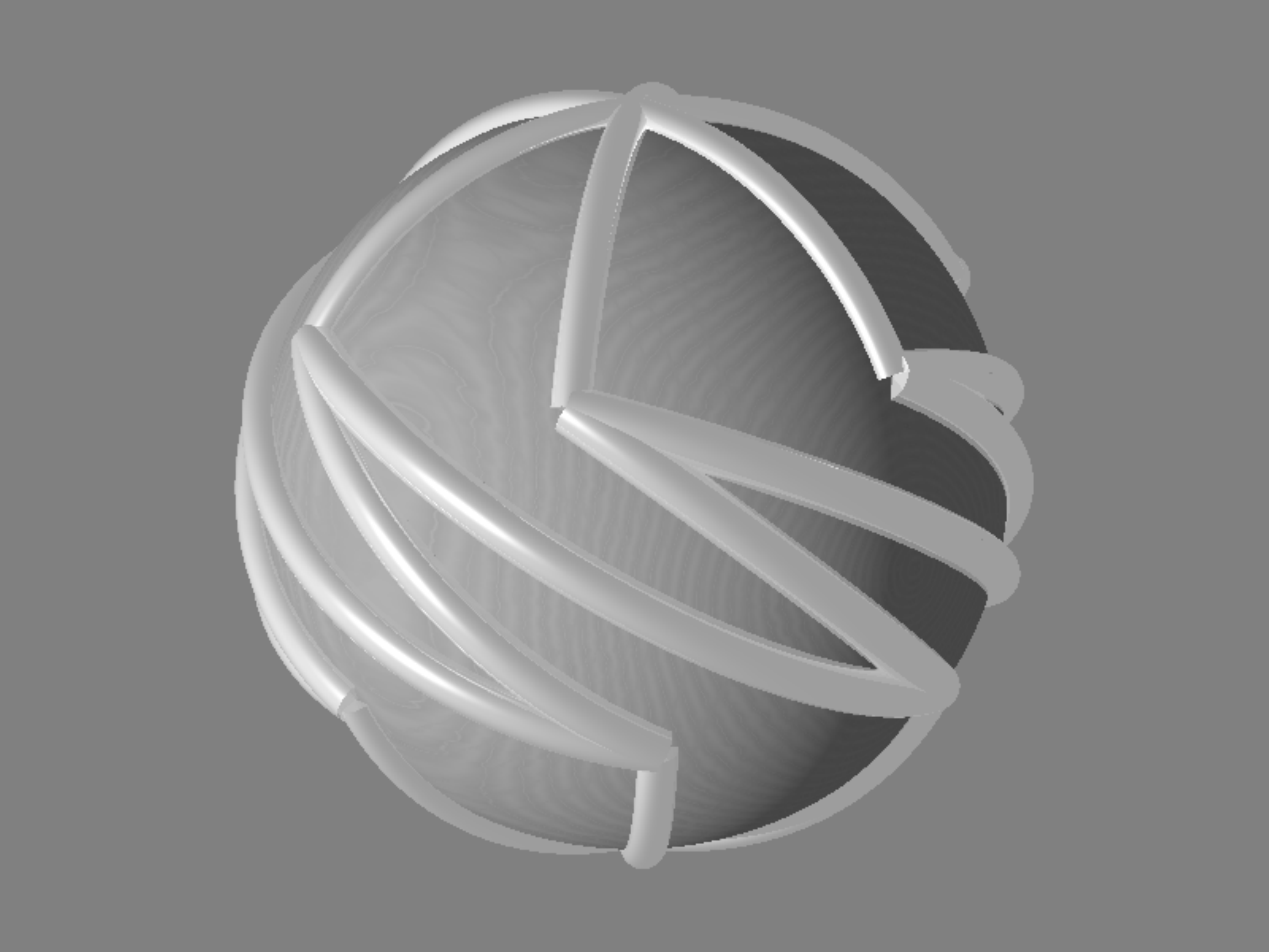}}};
 \node at (20,2.6) {\pgftext{\includegraphics[scale=0.07]{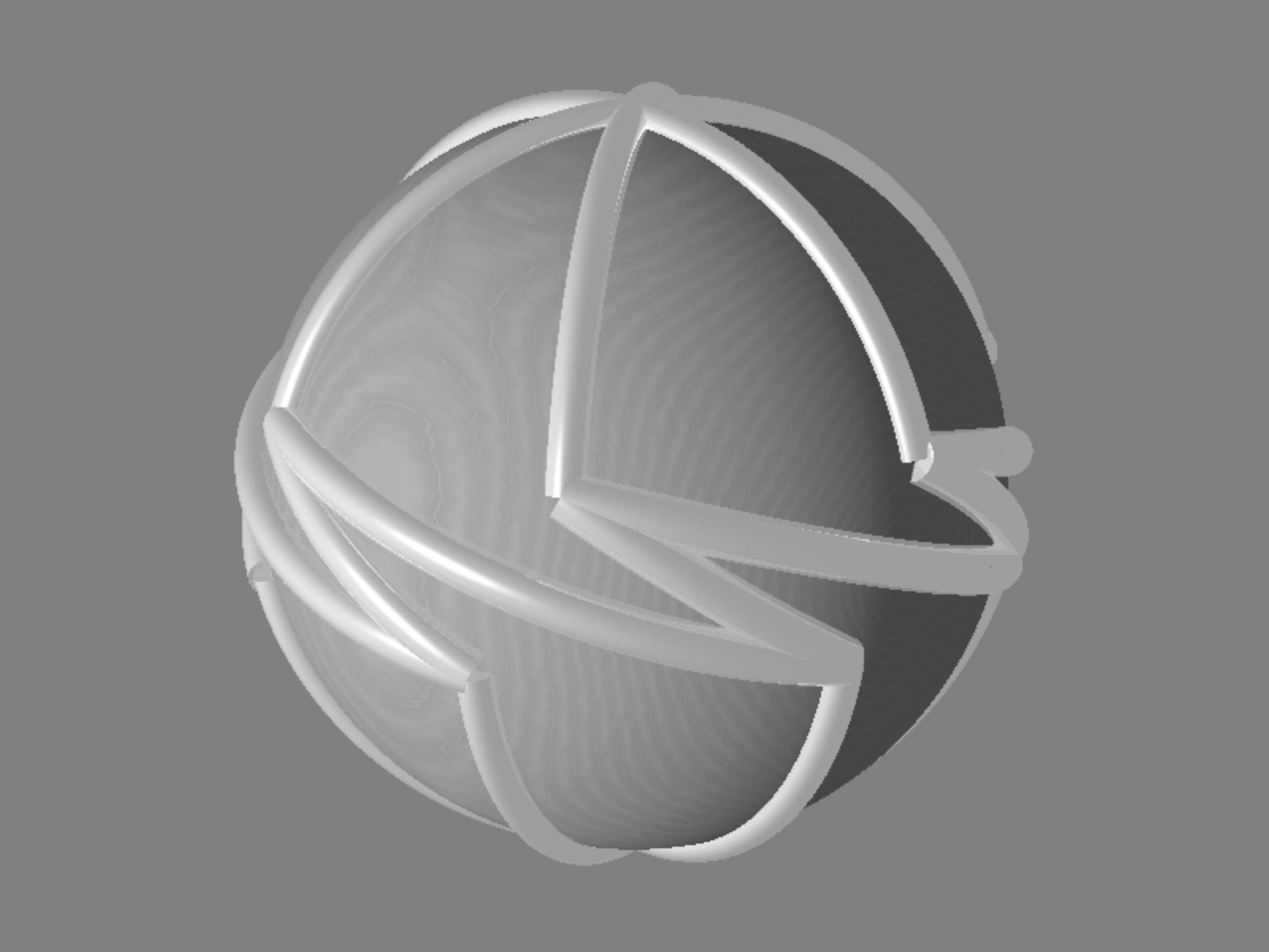}}};
  \end{tikzpicture}
 \caption{The copies of the tiles of the rightmost, middle
 spherical isohedral tiling $\T'$ organize
 a spherical \emph{non-isohedral} tiling $\T$ over the skeleton of $\T'$.
 The middle graph is an excerpt of
 Figure~\ref{fig:phasediagram3}.
 The right four images are \emph{the} spherical isohedral
 tilings by $\PQ_{n,\alpha,\gamma,a}$  for  $n=6$ and
 designated  $(\alpha,\gamma)$ on the graph.
 The distribution of inner angles and that of edge-length on the
 skeleton of $\T$ is the reflection of Figure~\ref{fig:Iso_by_A_tiles}.
 \label{fig:B234}
 }
\end{figure}
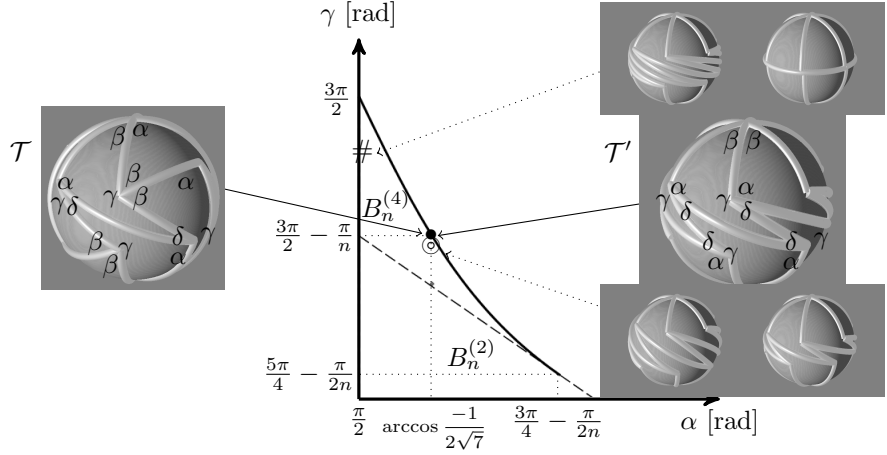
\begin{figure}[ht]\centering
 \begin{tikzpicture}[axis/.style={very thick, ->, >=stealth'}]
  \node at (0,-6)  {\includegraphics[scale=0.27]{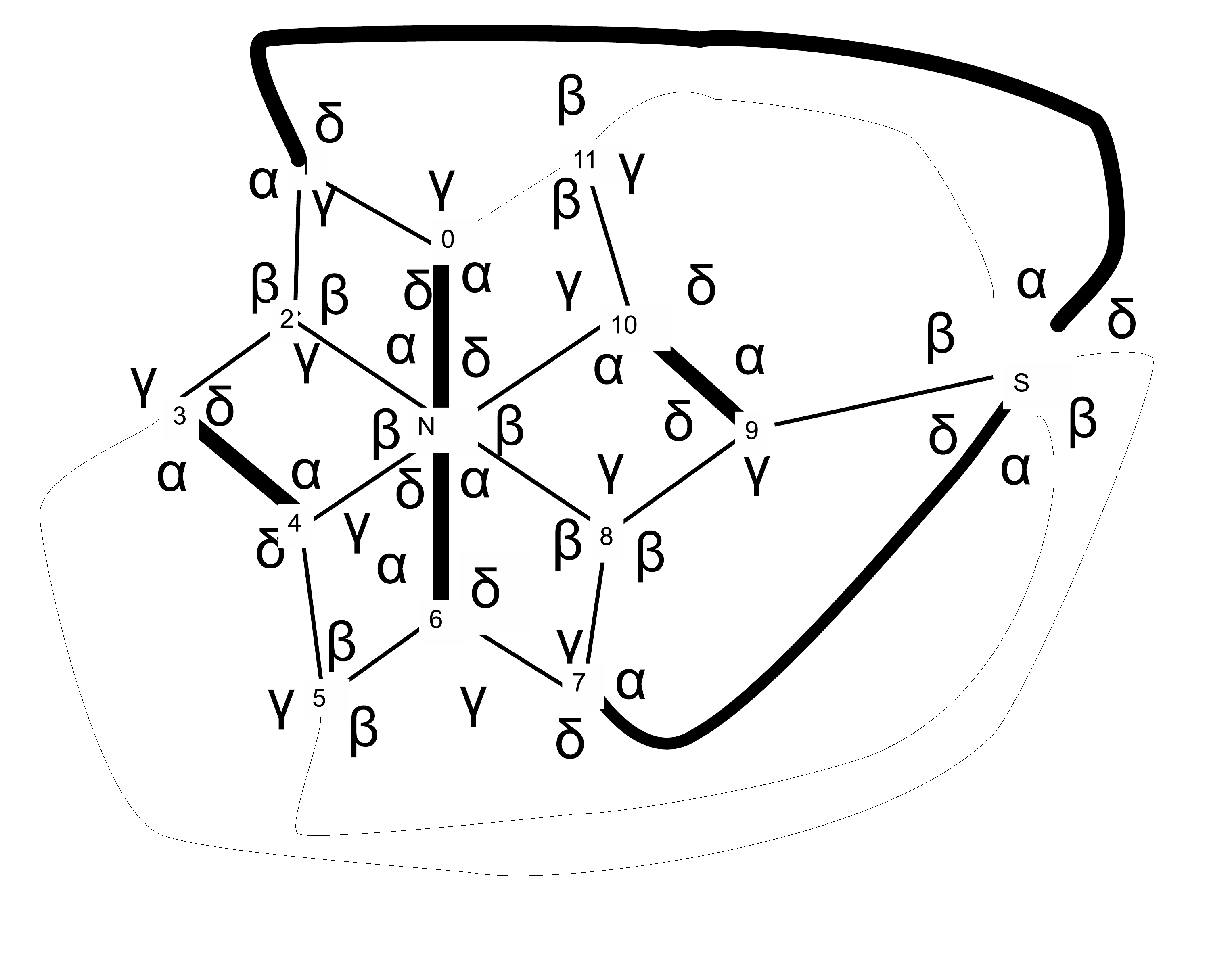}};
 \end{tikzpicture}
 \caption{The skeleton, edge-lengths, and inner angles of the reflection
 of $\T$. The solid, and the thick edges
have length $a=c=\arccos(1/3)$ and
 $b=\arccos(-5/9)$. $\alpha=\arccos(-1/(2\sqrt7))$, $\beta=\pi/3$,
 $\gamma=4\pi/3$, and $\delta=\arccos(5/(2\sqrt7))$.  See
 \cite{akama12:_class_of_spher_tilin_by_i} for detail of $\T$.\label{chart:a}\label{fig:Iso_by_A_tiles} }
\end{figure}

\section*{Acknowledgement}
The author thanks an anonymous referee.  Thanks also goes to Kosuke Nakamura for the earlier manuscript.

\end{document}